%% file: document.tex
\documentclass[12pt, a4paper]{amsart}
\pdfoutput=1
\usepackage[margin = 2.9cm]{geometry}
\usepackage[title]{appendix}
\usepackage{mathtools}
\usepackage{setspace} 
\usepackage{textcomp}
\usepackage{enumerate}
\usepackage{graphicx}
\usepackage{subfigure}
\usepackage[dvipsnames]{xcolor}

\usepackage{amsmath}
\usepackage{graphics} 
 \usepackage[colorlinks=true]{hyperref}
\hypersetup{urlcolor=blue, citecolor=red}

\newtheorem{thm}{Theorem}[section]
\newtheorem{corollary}{Corollary}

\newtheorem{lm}{Lemma}
\newtheorem{pp}{Proposition}
\newtheorem{claim}{Claim}

\theoremstyle{definition}
\newtheorem{df}{Definition}
\newtheorem{remark}{Remark}

\newtheorem{assumption}{Assumption}
\newcommand*\diff{\mathop{}\!\mathrm{d}}

\newcommand{\ep}{\epsilon}
\newcommand{\zm}{{Z^m}}
\newcommand{\zmm}{{Z^{m-1}}}
\newcommand{\zi}{{\zeta^{(i)}}}
\newcommand{\zj}{{\zeta^{(j)}}}
\newcommand{\zk}{{\zeta^{(k)}}}
\newcommand{\zl}{{\zeta^{(l)}}}
\newcommand{\pmt}{{(0, T) \times \partial N}}
\newcommand{\nxxi}{{\mathcal{N}(\vec{x}, \vec{\xi})}}
\newcommand{\ntxxi}{{{\mathcal{R}}(\vec{x}, \vec{\xi})}}

\newcommand{\uniongamma}{{\bigcup_{j=1}^{4}\gamma_{x_j, \xi_j}((0, \infty))}}
\newcommand{\unionGamma}{\Gamma({\vec{x}, \vec{\xi}})}

\newcommand{\capgamma}{{\cap_{j=1}^{4}\gamma_{x_j, \xi_j}((0, \infty))}}

\newcommand{\tM}{{\widetilde{M}}}
\newcommand{\tg}{{\tilde{g}}}
\newcommand{\tw}{{\widetilde{w}}}
\newcommand{\tQ}{{\widetilde{Q}}}
\newcommand{\yr}{{y_|}}
\newcommand{\etar}{{\eta_|}}
\newcommand{\yu}{{y_{0}}}
\newcommand{\etau}{{\eta^0}}
\newcommand{\wfset}{{\text{\( \WF \)}}}
\newcommand{\EUqr}{{ \mathcal{E}_U^{\mathrm{reg}} (q) }}
\newcommand{\EUq}{{ \mathcal{E}_U (q) }}
\newcommand{\EWqr}{{ \mathcal{E}_W^{\mathrm{reg}} (q) }}
\newcommand{\EWq}{{ \mathcal{E}_W (q) }}
\newcommand{\hmupj}{\widehat{\mu}_{\partial}^{(j)}}
\newcommand{\hmupa}{\widehat{\mu}_{\partial}^{(1)}}

\newcommand{\hmuj}{\widehat{\mu}^{(j)}}
\newcommand{\hmua}{\widehat{\mu}^{(1)}}
\newcommand{\hmub}{\widehat{\mu}^{(2)}}

\newcommand{\bx}{x_|}
\newcommand{\bxi}{\xi_|}
\newcommand{\bTM}{{}^{b}\dot{T}^* M}

\newcommand{\intM}{M^{{o}}}
\newcommand{\Char}{\mathrm{Char}}
\newcommand{\gsetint}{\mathcal{G}^{\mathrm{int}}}
\newcommand{\compchar}{\dot{\Sigma}_g}

\newcommand{\epslam}{{\partial_{\epsilon_1}\partial_{\epsilon_2}\partial_{\epsilon_3}\partial_{\epsilon_4} \Lambda(f) |_{\epsilon_1 = \epsilon_2 = \epsilon_3 = \ep_4=0}}}
\newcommand{\epslamt}{{\partial_{\epsilon_1}\partial_{\epsilon_2}\partial_{\epsilon_3} \Lambda(f) |_{\epsilon_1 = \epsilon_2 = \epsilon_3=0}}}
\newcommand{\epslamone}{{\partial_{\epsilon_1} \Lambda(\ep_1 f_1) |_{\epsilon_1=0}}}
\newcommand{\epslamvec}{{ \partial_{\vec{\epsilon}} \Lambda(f) |_{\vec{\epsilon}=0} }}
\newcommand{\RN}[1]{%
    \textup{\uppercase\expandafter{\romannumeral#1}}%
}

\newcommand{\TMpm}{T_{\partial M, \pm} M}

\newcommand{\LcMpm}{L^{*}_{\partial M, \pm}M}
\newcommand{\LcMpo}{L^{*}_{\partial M, +}M}
\newcommand{\LcMng}{L^{*}_{\partial M, -}M}

\DeclareMathOperator{\WF}{WF}
\DeclareMathOperator{\supp}{supp}

\DeclareMathOperator{\ufour}{\mathcal{U}^{(4)}}

\title[Inverse BVP for wave equations with quadratic nonlinearities]{Inverse Boundary Value Problems for wave equations with quadratic nonlinearities} 
\author{Gunther Uhlmann and Yang Zhang}

\begin{document}
\maketitle

\begin{abstract}
We study inverse problems for the nonlinear wave equation 
$\square_g u + w(x,u, \nabla_g u) = 0$
in a Lorentzian manifold $(M,g)$ with 
boundary, where $\nabla_g u$ denotes the gradient and  $w(x,u, \xi)$ is smooth and quadratic in $\xi$. 
Under appropriate assumptions, 
we show that the conformal class of the Lorentzian metric $g$ 
can be recovered up to diffeomorphisms, from the knowledge of the Neumann-to-Dirichlet map. 
With some additional conditions, we can recover the metric itself up to diffeomorphisms. 
Moreover, we can recover the second and third quadratic forms in the Taylor expansion of $w(x,u, \xi)$ with respect to $u$ up to null forms. 

\end{abstract}

\input{file/intro}

\input{file/wellposedness}

\input{file/prelim}

\input{file/planewaves}

\input{file/assymptotic}

\input{file/operatorQ}

\input{file/newthree}

\input{file/newfour}

\input{file/newmetric}
\input{file/ND}
\input{file/lens}

\input{file/newrecovery}
\section*{Acknowledgment}
GU was partially supported by NSF, a Walker Professorship at UW and a
Si-Yuan Professorship at IAS, HKUST.

\begin{footnotesize}
    \bibliographystyle{plain}
    \bibliography{microlocal_analysis} 
\end{footnotesize}

\end{document}

%% file: file/intro.tex
\section{Introduction}
In recent years starting with \cite{Kurylev2014, Kurylev2014a}, 
there have been many works studying inverse problems for nonlinear hyperbolic equations
\cite{Barreto2021,Barreto2020, Chen2019,Chen2020,Hoop2019,Hoop2019a,Uhlig2020,Feizmohammadi2019,Hintz2020, Kurylev2014,Balehowsky2020,Lai2021,Lassas2017,Tzou2021,Uhlmann2020,Dutta2020,Hintz2021,ultra21}. 
For an overview of the recent progress see \cite{Uhlmann2021}.
One of the main features of these papers is the fact that the nonlinearities help. 
The interaction of distorted planes waves in the nonlinear case 
produces new waves that give information for the inverse problems.
Several results that are not known for the corresponding linear equations can be proved for the nonlinear equations.

Several of papers mentioned above have used the source-to-solution map as the data.
Inverse boundary value problems have been considered for nonlinear elastic equations in \cite{Uhlmann2019, Hoop2019}
and for semilinear equations in \cite{Hintz2020,Hintz2021}.
In this paper, we consider inverse boundary value problems for quadratic derivative nonlinear wave equations. 
One example is the wave map equation, 
see \cite[Chapter 6]{Tao2006} and \cite{Tataru2004}.
The corresponding inverse problem with the source-to-solution map as the data has been studied in \cite{Wang2019}.
In the following, we describe the inverse boundary value problems that we consider and present the main results.

Let $(M,g)$ be a $(1+3)$-dimensional globally hyperbolic Lorentzian manifold with boundary, where the metric has the signature $(-,+,+,+)$.
By \cite{Kurylev2018}, we can suppose $M = \mathbb{R} \times N$, where $N$ is a $3$-dimensional manifold with boundary $\partial N$.  
Assume the boundary $\partial M = \mathbb{R} \times \partial N$ is timelike and strictly null-convex as in \cite{Hintz2020}.
Here $\partial M$ is null-convex if for any null vector $v \in T_p \partial M$ one has
\begin{eqnarray}\label{def_nconvex}
\kappa(v,v) = g (\nabla_v \nu, v) \geq 0,
\end{eqnarray}
where we denote by $\nu$ the outward pointing unit normal vector field on $\partial M$. 
We say $\partial M$ is strictly null-convex if (\ref{def_nconvex}) holds with the strict inequality for nonzero null vector $v$.

We consider the quadratic derivative nonlinear wave equation 
\begin{align}\label{ihmNBC}
\begin{split}
\square_g u + w(x,u, \nabla_g u) = 0, \quad &\mbox{on } M,\\
\partial_\nu u(x) = f, \quad &\mbox{on } \partial M,\\
{u} = 0, \quad &\mbox{for } t <0,
\end{split}
\end{align}
where $w(x, u, \xi)$ is smooth in $x, u$ and quadratic in $\xi \in T_xM$. 
Here we use
$\nabla_g u$ to denote the gradient of $u$ and \[
\square_g = (-\det g)^{-\frac{1}{2}} \partial_j ((-\det g)^{\frac{1}{2}}g^{jk} \partial_k)
\] is the Laplace-Beltrami operator.
The boundary data $f$ 
is assumed to be compatible with the initial condition
and $\|f\|_{C^{6} ([0,T] \times \partial N)} $ is assumed to be small enough,
such that 
(\ref{ihmNBC}) is well-posed, see Section \ref{Sec_well}.
Then we can define 
\[
\Lambda f = u|_{\partial M}
\]
as the Neumann-to-Dirichlet (ND) map.
In this work, we study the inverse problem of recovering the quadratic form $w(x, u, \xi)$ and the Lorentzian metric $g$ from the ND map under proper assumptions. 
Similarly one can expect the same results assuming the knowledge of the Dirichlet-to-Neumann map (DN map).

\subsection{Main Results.}
We first introduce some definitions to state the main results. 
A smooth path $\mu:(a,b) \rightarrow M$ is timelike if $g(\dot{\mu}(s), \dot{\mu}(s)) <0 $ for any $s \in (a,b)$. 
It is causal if $g(\dot{\mu}(s), \dot{\mu}(s)) \leq 0 $ and $\dot{\mu}(s) \neq 0$ for any $s \in (a,b)$. 
For $p, q \in M$, we denote by $p < q$ (or $p \ll q$) if $p \neq q$ and there is a future pointing casual (or timelike) curve from $p$ to $q$. 
We denote by $p \leq q$ if either $p = q$ or $p<q$. 
The chronological future of $p$ is the set $I^+(p) = \{q \in M: \ p \ll q\}$ 
and the causal future of $p$ is the  set $J^+(p) = \{q \in M: \ p \leq q\}$.
Similarly we can define the chronological past $I^-(p)$ and the causal past $J^-(p)$.
For convenience, we use the notation 
$I(p,q)$ to denote the diamond set $I^+(p) \cap I^-(q)$.

We write the Taylor expansion of $w(x,u,\xi)$ in $u$ as
\begin{align}\label{def_taylorw}
w(x,u,\xi) = \mathcal{N}_0(x, \xi) + u \mathcal{N}_1 (x, \xi) + u^2 \mathcal{M}(x, \xi) + o(|u|^2 \cdot |\xi|^2),
\end{align}
where $\mathcal{N}_0, \mathcal{N}_1, \mathcal{M}$ are quadratic forms in $\xi$ and depending smoothly on $x$.
We sometimes omit the variable $x$ and write them as $\mathcal{N}_0(\xi, \eta)$, $\mathcal{N}_1(\xi, \eta)$, and  $\mathcal{M}(\xi, \eta)$ to emphasize they are quadratic. 
In \cite{Wang2019}, it is assumed that $\mathcal{N}_0, \mathcal{N}_1$ are null forms while $\mathcal{M}$ is not null.
Here we say a quadratic form $w(\xi, \eta)$ is null if $w(\xi, \xi) = 0$ for any $\xi$ satisfying $g(\xi, \xi) = 0$. 
By \cite[Lemma 2.1]{Wang2019}, null forms are linear combinations of the Lorentzian metric $g$ and anti-symmetric quadratic forms that vanish when $\xi = \eta$.

In this work, first we make the assumption on the quadratic form $w(x,u,\xi)$ that 
\begin{align}\label{assum_n0}
\mathcal{N}_0 \text{ is a null form and } \mathcal{N}_1, \mathcal{M} \text{ are symmetric}.
\end{align}
Then we have the following theorem which says that from the ND map $\Lambda$ we can recover $\mathcal{N}_1$ and $\mathcal{M}$ up to null forms. 
\begin{thm}\label{thm_recoverNM}
    Let $(M,g)$ be a globally hyperbolic Lorentzian manifold with timelike and 
    null-convex boundary. 
    Consider the quadratic derivative nonlinear wave equation
    \[
    \square_g u + w^{(j)}(x,u, \nabla_g u) = 0, \text{ for } j=1,2.
    \]
    Assume the quadratic form $ w^{(j)}(x,u, \xi)$ has the Taylor expansion in (\ref{def_taylorw})
    and satisfies the assumption (\ref{assum_n0}). 
    If the Neumann-to-Dirichlet map $\Lambda_j$ are equal for 
    all functions $f$ in a small neighborhood of the zero functions in $C^6([0,T] \times \partial N)$, 
    then there are smooth $c_1(x), c_2(x)$ such that
    \begin{align*}
    \mathcal{N}^{(2)}_1 (x, \xi) = \mathcal{N}^{(1)}_1 (x, \xi) + c_1(x) g(\xi, \xi), \quad 
    \mathcal{M}^{(2)} (x, \xi) = \mathcal{M}^{(1)} (x, \xi) + c_2(x) g(\xi, \xi),
    \end{align*}
    for $ x \in \mathbb{W}$ and $\xi \in T^*_x M$, where 
    \[
    \mathbb{W} = \bigcup_{y^-, y^+ \in (0,T) \times \partial N} I(y^-, y^+). 
    \]
\end{thm}

Next we make the assumption on 
the quadratic form $w(x,u,\xi)$ that
\begin{align}\label{assum_n0n1}
\mathcal{N}_0, \mathcal{N}_1 \text{ are null forms and } \mathcal{M} \text{ is not null.} 
\end{align}
With this assumption, we show that from the ND map restricted to a proper open set on the boundary, one can determine the conformal class of the metric at observable interior points up to diffeomorphisms. 
\begin{thm}\label{thm_recoverg} 
    Let $g_j$, $j=1,2$ be two Lorentzian metrics on a smooth manifold $ M = \mathbb{R} \times N $ such that $(M, g_j)$ are globally hyperbolic Lorentzian manifolds with timelike and strictly null-convex boundary. 
    Let $\hmupj: \ [0,1] \rightarrow (0,T) \times \partial N$ be two timelike paths 
    with 
    $y^{\pm,(j)} = \hmupj(s_\pm)$,  where $ 0 < s_- < s_+ <1$.
    Consider 
    the quadratic derivative nonlinear wave equations
    \[
    \square_{g_j}u + w^{(j)}(x,u, \nabla_{g_j} u) = 0, \quad j=1,2,
    \] 
    and the Neumann-to-Dirichlet map $\Lambda_j$.
    Assume the quadratic form $w^{(j)}(x,u, \xi)$ has the Taylor expansion in (\ref{def_taylorw}) and 
    satisfies the assumption (\ref{assum_n0n1}). 
    Let ${O}_j$ be connected open sets in $\partial N$ and let $V_{\partial, j} = (0, T) \times {O}_j$ 
    be small neighborhoods of $\hmupj((0,1))$ satisfying $(I(\hmupj(0), \hmupj(1)) \cap \partial M) \subset V_{\partial, j}$. 
    If there is a diffeomorphism $\Phi_\partial: V_{\partial, 1} \rightarrow V_{\partial, 2}$ 
    with $y^{\pm,(2)} = \Phi_\partial(y^{\pm,(1)})$ such that
\begin{align*} 
(\Phi_\partial^{-1})^* \circ \Lambda_1 \circ \Phi_\partial^* (f) = \Lambda_2(f)
\end{align*}
    for all functions $f$ in a small neighborhood of the zero function in $C^6(V_{\partial, 2})$, 
    then there exists a diffeomorphism $\Psi: I(y^{-,(1)}, y^{+,(1)}) \rightarrow I(y^{-,(2)}, y^{+,(2)}) $ such that 
    $\Psi^*g_2 = e^\gamma g_1$ in $I(y^{-,(1)}, y^{+,(1)})$
    with $\gamma$ smooth.
    
    In particular, one can choose ${O}_j= \partial N$, i.e. $V_{\partial, j} = (0, T )\times \partial N$.
    In this case, there exists a diffeomorphism $\Psi: \mathbb{W}_1 \rightarrow \mathbb{W}_2$ such that $\Psi^*g_2 = e^\gamma g_1$,
    where 
    \[
    \mathbb{W}_j = \bigcup_{y^{-,(j)}, y^{+,(j)} \in {V}_{\partial,j}} I(y^{-,(j)}, y^{+,(j)}), \text{ for } j = 1,2.
    \] 
    
    Additionally, if $\mathcal{M}^{(1)} =\mathcal{M}^{(2)}$ are independent of $x$ or $g_1, g_2$ are Ricci flat, then the diffeomorphism $\Psi$ is an isometry. 
\end{thm}  
    \begin{remark}
    The ND map are gauge invariant under a diffeomorphism of $M$ that fixes $\partial M$, see Lemma \ref{lm_gauge}.
    This kind of invariance happens if we have the DN map for the wave equations, see \cite[Lem 2.1]{Stefanov2018} (see also \cite{Hintz2021} for semilinear wave equations). 

    \end{remark}
    As we mentioned, the inverse problems of recovering the metric and the nonlinear terms were originated in \cite{Kurylev2018} for the semilinear wave equation
    $
    \square_g u(x) + a(x)u^2(x) = f(x)
    $
    in a manifold without boundary. 
    The authors use the nonlinear interactions of distorted planes waves to 
    produce point source like singularities in an observable set, and then 
    determine the conformal class of the metric $g$ in this set by observing the leading singularities of the linearized source-to-solution map.
    The recovery for semilinear wave equations with general derivative-free nonlinear terms are considered in \cite{Lassas2018}. 
Compared to \cite{Wang2019} where it was considered as data the source to
solution map for the same type of nonlinear wave equations,
the main difficulty that we need to handle when considering the inverse boundary
value problem is the presence of the boundary. 
Our analysis consists of analyzing the interaction of singularities for the linearized
equation, 
and we need to consider multiple reflections of waves as well as
gliding rays following the propagation of singularities along the
generalized bicharacteristic flow as proved in \cite{Melrose1978}. 
With the null-convex boundary, we can only consider the reflections of singularities. 
In \cite{Hintz2020}, the
multiple reflections were handled using the scattering control method of \cite{Caday2018},
under the additional assumption of no-conjugate points. 
We don't make such an assumption here.
    
    
    


The structure of the paper is as follows. 
In Section \ref{Sec_well}, we establish the well-posedness of the boundary value problems (\ref{ihmNBC}) for small boundary data. 
In Section \ref{Sec_prelim}, we present some preliminaries from Lorentzian geometry and microlocal analysis, construct the distorted planes waves, and derive the asymptotic expansion of the solution to (\ref{ihmNBC}). 
In Section \ref{Sec_threewaves} and \ref{Sec_fourwaves}, we analyze the singularities produced by the nonlinear interactions of three and four distorted plane waves.
Based on these results, 
we establish Proposition \ref{thm33} to describe when a covector over $\partial M$ can be observed 
from the leading singularities of the boundary measurements, and we prove Theorem \ref{thm_recoverNM}. 
In Section \ref{obs}, we show the determination of the earliest light observation set by using the ND map, see Proposition \ref{pp_recoverEOS}, under some assumptions of the extended manifold $(\tM, \tg)$. 
In Section \ref{sec_recovery}, we finish the proof of Theorem \ref{thm_recoverg} by showing the desired assumptions can be fulfilled. 

%% file: file/wellposedness.tex
\section{Local well-posedness}\label{Sec_well}
The well-posedness of the boundary value problem (\ref{ihmNBC}) with small boundary data $f$ can be established by the same arguments as in \cite{Hintz2020}. 
In the following let $T >0$ be fixed. Let $f \in C^{m+1} ([0,T] \times {\partial N})$ with $m \geq 5$ and $\|f\|_{C^{m+1} ([0,T] \times {\partial N})} \leq \epsilon_0$ with small positive number $\epsilon_0$ to be specified later. 
There exists a function $h \in C^{m+1} ([0,T] \times N)$ such that $\partial_\nu h|_{\partial M} = f$ and 
\[
\|h\|_{C^{m+1} ([0,T] \times N)} \leq\|f\|_{C^{m+1} ([0,T] \times {\partial N})} .
\]
Let $\tilde{u} = u -h $. 
If we rewrite the quadratic term as
\[
w(x,u, \nabla_g u) = q_{ij}(x,u) \nabla_g^i u \nabla_g^j u,
\] 
then we have 
\begin{equation}\label{hmNBC}
\begin{cases}
\square_g \tilde{u} = - \square_g h  - q_{ij}(x,\tilde{u} + h) \nabla^i_g (\tilde{u} + h) \nabla^j_g (\tilde{u} + h), &\mbox{on } M,\\
\partial_\nu \tilde{u} = 0, &\mbox{on } \partial M,\\
\tilde{u} = 0, &\mbox{for } t <0.
\end{cases}
\end{equation}
For $R>0$, we define $Z^m(R,T)$ as the set containing all functions $v$ such that
\[
v \in \bigcap_{k=0}^{m} W^{k, \infty}([0,T]; H^{m-k}(N)), 
\quad \|v\|^2_{Z^m} = \sup_{t \in [0, T]} \sum_{k=0}^m \|\partial_t^k v(t)\|^2_{H^{m-k}} \leq R^2.
\] 
In the following we abuse the notation $C$ to denote different constants that depends on $m, M, T$. 
We can show the following claim by Sobolev Embedding Theorem.
\begin{claim}\label{normineq}
Suppose $u \in Z^m(R, T)$.
Then $\|u\|_\zmm \leq \|u\|_\zm$ and $\nabla^j_g u \in Z^{m-1}(R, T)$,
$j = 1, \dots, 4$. Moreover, we have the following estimates. 
\begin{enumerate}[(1)]
    \item If $v \in Z^m(R', T)$, then $\|
    uv\|_\zm \leq C \|u\|_\zm \|v\|_\zm$.
    \item If $v \in \zmm(R', T)$, then $\|uv\|_\zmm \leq C \|u\|_\zm \|v\|_\zmm$.
    \item If $q(x,s) \in C^m(M \times \mathbb{R})$, 
    then $\|q(x, u)\|_\zm \leq C \|q\|_{C^m(M \times \mathbb{R})} (\sum_{l=0}^{m} \|u\|^l_\zm)$.
\end{enumerate}
\end{claim}
\begin{proof}
{
The first statement is straightforward by the definition of $Z^m(R, T)$.
We show the three estimates below.
To prove (1), for fixed $k$ we need to estimate 
\[
\| \partial_t^k  (uv) \|^2_{H^{m-k}} = 
\sum_{|\alpha| \leq m-k} \|\partial_t^k \partial_{x'}^\alpha (uv)\|^2_{L^2}
=\sum_{|\alpha| \leq m-k} 
\|\sum_{\substack{\alpha_1 + \alpha_2 = \alpha \\ k_1 + k_2 = k}}
(\partial_t^{k_1} \partial_{x'}^{\alpha_1} u)
(\partial_t^{k_2} \partial_{x'}^{\alpha_2} v)\|^2_{L^2}.
\]
Note that $\partial_t^{k_1} \partial_{x'}^{\alpha_1} u \in H^{m-k_1 - |\alpha_1|}$ and at most one of $k_j + |\alpha_j|$, $j = 1,2$ is no less than $m-2$, with the assumption $m \geq 5$. 
We can assume $k_1 + |\alpha_1| < m-2$. 
Then by Sobolev Embedding Theorem, we have $\partial_t^{k_1} \partial_{x'}^{\alpha_1} u \in C^{m-k_1 - |\alpha_1|-2}$ 
with 
\[
\|\partial_t^{k_1}  \partial_{x'}^{\alpha_1} u\|_{C^{m-k_1 - |\alpha_1|-2}} \leq C\|\partial_t^{k_1} \partial_{x'}^{\alpha_1} u\|_{H^{m-k_1 - |\alpha_1|}}
\leq C\|\partial_t^{k_1}  u\|_{H^{m-k_1 }}.
\]
This implies 
\[
\| \partial_t^k  (uv) \|^2_{H^{m-k}} 
\leq C \sum_{k_1 + k_2 =k} \|\partial_t^{k_1}  u\|^2_{H^{m-k_1 }} \|\partial_t^{k_2}  v\|^2_{H^{m-k_2 }}
\leq C \| u\|^2_{\zm} \|v\|^2_{\zm}
\]
and therefore we have (1). 
Combining (1) and the first statement $\|u\|_\zmm \leq \|u\|_\zm$, we have the estimate in (2).
For (3), we write $G(t,x') = q(x,u)$ and compute
\[
\partial_t^k \partial_{x'}^{\alpha}  G(t,x')  = \sum_{\substack{k_0 + \ldots + k_l = k\\ \alpha_0 + \ldots + \alpha_l = \alpha}} (\partial_t^{k_0}  \partial_{x'}^{\alpha_0}  \partial_{u}^{m-k_0 - |\alpha_0|} q)(\partial_t^{k_1}  \partial_{x'}^{\alpha_1} u) \ldots (\partial_t^{k_l}  \partial_{x'}^{\alpha_l} u).
\]
Similarly, for each fixed set of indexes $(k_0, \ldots, k_l, \alpha_0, \ldots, \alpha_l)$, at most one of $k_j + |\alpha_j|$ is no less than $m-2$. 
Then by Sobolev Embedding Theorem we have
\begin{align*}
\| (\partial_t^{k_0}  \partial_{x'}^{\alpha_0} \partial_{u}^{m-k_0 - |\alpha_0|} q)(\partial_t^{k_1}  \partial_{x'}^{\alpha_1} u) \ldots (\partial_t^{k_l}  \partial_{x'}^{\alpha_l} u) \|_{L^2}
&\leq C \|q\|_{C^m}\|\partial_t^{k_1}u\|_{H^{m-k_1}} \ldots \|\partial_t^{k_l}u\|_{H^{m-k_l}}\\
&\leq C \|q\|_{C^m}\|u\|_\zm^l.
\end{align*}
This leads to 
$\|q(x, u)\|_\zm \leq C \|q\|_{C^m(M \times \mathbb{R})} (\sum_{l=0}^{m} \|u\|^l_\zm)$. 
}
\end{proof}

For $v \in Z^m(\rho_0, T)$ with $\rho_0$ to be specified later, we consider the linearized problem 
\begin{equation*}
\begin{cases}
\square_g \tilde{u} = - \square_g h  - q_{ij}(x,v + h) \nabla^i_g (v + h) \nabla^j_g (v + h), &\mbox{on } M,\\
\partial_\nu \tilde{u} = 0, &\mbox{on } \partial M,\\
\tilde{u} = 0, &\mbox{for } t <0,
\end{cases}
\end{equation*}
and we define the solution operator $\mathcal{J}$ which maps $v$ to the solution $\tilde{u}$.
By Claim \ref{normineq}
we have 
\begin{align*}
&\| - \square_g h  - q_{ij}(x,v + h) \nabla^i_g (v + h) \nabla^j_g (v + h) \|_{Z^{m-1}} \\
\leq &  \| - \square_g h \|_\zm + C\|q_{ij}(x,v + h)\|_\zm \|\nabla^i_g (v + h)\|_\zmm \|\nabla^j_g (v + h) \|_\zmm )\\
\leq & C( \ep_0 + (1 + (\rho_0 + \ep_0) + \ldots +  (\rho_0 + \ep_0)^m) (\rho_0 + \ep_0)^2 ). 
\end{align*}
By \cite[Theorem 3.1]{Dafermos1985} 
the linearized problem has a unique solution 
\[
\tilde{u} \in \bigcap_{k=0}^{m} C^{k}([0,T]; H^{m-k}(N))
\] 
such that
\[
 \| \tilde{u}\|_{Z^m} \leq C(\ep_0 + (1 + (\rho_0 + \ep_0) + \ldots +  (\rho_0 + \ep_0)^m) (\rho_0 + \ep_0)^2 )e^{KT},
\]
where $C, K$ are positive constants. 
If we assume $\rho_0$ and $\epsilon_0$ are small enough, then the above inequality implies that
\[
\| \tilde{u}\|_{Z^m} \leq C(\ep_0 + (\rho_0 + \ep_0)^2 )e^{KT}.
\] 
For any $\rho_0$ satisfying $\rho_0 < 1/({2C e^{KT}})$, 
we can choose $\ep_0 = {\rho_0}/({8C e^{KT}}) $  such that 
\begin{equation}\label{eps}
C(\ep_0 + (\rho_0 + \ep_0)^2 )e^{KT} < \rho_0.
\end{equation}
In this case, we have  $\mathcal{J}$ maps $Z^m(\rho_0, T)$ to itself. 

In the following we show that $\mathcal{J}$ is a contraction map if $\rho_0$ is small enough. 
It follows that the boundary value problem (\ref{hmNBC}) has a unique solution $\tilde{u} \in Z^m(\rho_0, T)$ as a fixed point of $\mathcal{J}$.
Indeed, for $\tilde{u_j} = \mathcal{J}(v)_j$ with $v_j \in Z^m(\rho_0, T)$, we have $\tilde{u}_2 - \tilde{u}_1$ satisfies 
\begin{align*}
\square_g (\tilde{u}_2 - \tilde{u}_1) &= - (q_{ij}(x,v_2 + h) \nabla^i_g (v_2 + h) \nabla^j_g (v_2 + h) -q_{ij}(x,v_1 + h) \nabla^i_g (v_1 + h) \nabla^j_g (v_1 + h))\\
& = -(\int_0^1 \partial_u q_{ij}(x, h + v_1 + \tau (v_2 - v_1)) \diff \tau )(v_2 - v_1)\nabla^i_g (v_2 + h) \nabla^j_g (v_2 + h)\\
& -q_{ij}(x,v_1 + h)\nabla^i_g (v_2 - v_1) \nabla^j_g (v_2 + h) -q_{ij}(x,v_1 + h)\nabla^i_g (v_1 + h)\nabla^j_g (v_2 -v_1).\\
\end{align*}
We denote the right-hand side by $\mathcal{I}$ and by Claim \ref{normineq} we have the following estimates 
\begin{align*}
\|\mathcal{I}\|_\zmm 
&\leq C\|q_{ij}(x, u)\|_{C^m(M \times \mathbb{R})} 
(\|v_2 - v_1\|_\zm \|\nabla^i_g (v_2 + h)\|_\zmm \|\nabla^j_g (v_2 + h)\|_\zmm \\
& +  \|\nabla^i_g (v_2 - v_1)\|_\zmm \|\nabla^j_g (v_2 + h)\|_\zmm +\|\nabla^i_g (v_1 + h)\|_\zmm \|\nabla^j_g (v_2 -v_1)\|_\zmm)\\
&\leq C\|q_{ij}(x, u)\|_{C^m(M \times \mathbb{R})} 
\|v_2 - v_1\|_\zm ((\rho_0 + \ep_0)^2  + 2(\rho_0 + \ep_0))\\
&\leq C'  \|v_2 - v_1\|_\zm (\rho_0 + \ep_0),
\end{align*}
where last inequality holds if we choose $\rho_0, \ep_0$ small enough. 
By \cite[Theorem 3.1]{Dafermos1985} and (\ref{eps}) one has
\[
\| \tilde{u}_2 - \tilde{u}_1 \|_\zm \leq CC' \|v_2 - v_1\|_\zm (\rho_0 + \ep_0) e^{KT} < C'{C{ e^{KT} }} (1 + 1/(8Ce^{KT}))\rho_0 \|v_2 - v_1\|_\zm.
\]
This implies that if we choose $\rho \leq \frac{1}{C'Ce^{KT}(1 + 1/(8Ce^{KT}))} $ then 
\[
\| \mathcal{J}(v_2 -v_1)\|_\zm < \|v_2 - v_1\|_\zm
\]
shows that $\mathcal{J}$ is a contraction, which proves that there exists a unique solution $\tilde{u}$ to the problem (\ref{hmNBC}).
Furthermore, by \cite[Theorem 3.1]{Dafermos1985} this solution 
satisfies the estimates $\|\tilde{u}\|_\zm \leq 8C e^{KT} \ep_0.$
Therefore, we prove the following proposition. 
\begin{pp}
Let $f \in C^{m+1} ([0,T] \times N)$ with $m \geq 5$. 
Suppose $f = \partial_t f = 0$ at $t=0$. 
Then there exists small positive $\ep_0$ such that for any 
$\|f\|_{C^{m+1} ([0,T] \times N)} \leq \epsilon_0$, we can find a unique solution 
\[
u \in \bigcap_{k=0}^m C^k([0, T]; H_{m-k})
\]
to the boundary value problem (\ref{ihmNBC}), which satisfies the estimate
\[
\|{u}\|_\zm \leq C \|f\|_{C^{m+1}([0,T]\times N)}
\]
for some $C>0$ independent of $f$. 
\end{pp}


%% file: file/prelim.tex
\section{Preliminaries and the asymptotic expansion}\label{Sec_prelim}
Let $(M,g)$ be a globally hyperbolic Lorentzian manifold with timelike and 
null-convex boundary. 
As in \cite{Hintz2020}, we extend the metric $g$ smoothly to a slightly larger open manifold $\tM  = \mathbb{R}_t \times \widetilde{N}$, 
such that $N$ is contained in the interior of $\widetilde{N}$.
Suppose we are given an open set $V$, the observation set,  in $[0,T] \times \widetilde{N} \setminus N$.
For more details of this extension and the construction of $V$, see Section \ref{sec_recovery}.

In this section, we recall some preliminaries in \cite{Kurylev2018}.
Then we construct the distorted planes waves and carry out the asymptotic analysis of the solution to (\ref{ihmNBC}).
In the last subsection, we recall the propagation of singularities for the wave equation on a manifold with boundary and use these results to understand the singularities of the nonlinear interaction of waves in Section \ref{Sec_threewaves} and Section \ref{Sec_fourwaves}.

\subsection{Notations}
For $\eta \in T_p^*\tM$, the corresponding vector of $\eta$  is denoted by $ \eta^\# \in T_p \tM$.
The corresponding covector of a vector $\xi \in T_p \tM$ is denoted by $ \xi^b \in T^*_p \tM$.
We denote by 
\[
L_p \tM = \{\zeta \in T_p \tM \setminus 0: \  g(\zeta, \zeta) = 0\}
\]
the set of light-like vectors at $p \in \tM$ and, similarly by $L^*_p \tM$ the set of light-like covectors.
The sets of future (or past) light-like vectors are denoted by $L^+_p \tM$ (or $L^-_p \tM$), and those of future (or past) light-like covectors are denoted by $L^{*,+}_p \tM$ (or $L^{*,-}_p \tM$).

The null bicharacteristic set $\Char(\square_g)$ is the set $p^{-1}(0) \subset T^*M$, where 
$
p(x, \zeta) =g^{ij}\zeta_i \zeta_j
$
is the principal symbol of $\square_g$. 
It is also the set of light-like covectors with respect to $g$.  
We denote by $\Theta_{x, \zeta}$ the null bicharacteristic of $\square_g$ that contains $(x, \zeta) \in L^*\tM$. 
Then a covector $(y, \eta) \in \Theta_{x, \zeta}$ if and only if there is a light-like geodesic $\gamma_{x, \zeta^\#}$ such that 
\[
(y, \eta) = (\gamma_{x, \zeta^\#}(t), (\dot{\gamma}_{x, \zeta^\#}(t))^b), \ \text{ for } t \in \mathbb{R}.
\]
We denote by $\Theta^+_{x, \zeta}$ the set of such covector $(y, \eta)$ with $t>0$ above, i.e. if $(y, \eta)$ lies in the forward bicharacteristic starting from $(x,\zeta)$. 
The future directed light-cone emanating from $p$ is defined as 
\[
\mathcal{L}^+(p) = \{\gamma_{p, \zeta} (t) \in \tM: \zeta \in L_p^+ \tM, t \geq 0\}.
\]
The time separation function $\tau(x,y) \in [0, \infty)$ between two points $x < y$ in  $\tM$
is the supremum of the lengths \[
L(\alpha) =  \int_0^1 \sqrt{-g(\dot{\alpha}(s), \dot{\alpha}(s))} ds
\] of 
the piecewise smooth causal paths $\alpha: [0,1] \rightarrow M$ from $x$ to $y$. 
If $x<y$ is not true, we define $\tau(x,y) = 0$. 
Note that $\tau(x,y)$ satisfies the reverse triangle inequality
\[
\tau(x,y) +\tau(y,z) \leq \tau(x,z), \text{ where } x \leq y \leq z.
\]
For $(x,\xi) \in L^+M$, recall the cut locus function
\[
\rho(x,\xi) = \sup \{ s\in [0, \mathcal{T}(x,\xi)]:\ \tau(x, \gamma_{x,\xi}(s)) = 0 \},
\]
where $\mathcal{T}(x,\xi)$ is the maximal time such that $\gamma_{x,\xi}(s)$ is defined.

The light observation set of a point $q \in M$ in a set $W$ is 
\begin{align*}
\mathcal{P}_W (q) = \mathcal{L}^+(q) \cap W, 
\end{align*}
and the earliest light observation set $\EWq$ in a set $W$ of $q \in M$  is
\begin{align}\label{def_eob}
\EWq = & \{x \in \mathcal{P}_W (q): \text{ there are no }y \in \mathcal{P}_W (q) 
\text{ and a future-pointing }  \\
& \text{ time-like path } \alpha:[0,1] \rightarrow W \text{ with } \alpha(0) = y \text{ and } \alpha(1) = x\}. \nonumber
\end{align}
The direction set of $q$ in a set $W$ is 
\begin{align*}
\mathcal{C}_W(q) = \{
(y, \eta) \in L^+W: \ &y = \gamma_{q,\xi}(t), \eta = \dot{\gamma}_{q,\xi}(t), \\
& \text{for some } \xi \in L^+_q M, \ 0 \leq t \leq \rho(x,\xi)\},
\end{align*}
and the regular direction set of $q$ in a set $W$ is 
\begin{align*}
\mathcal{C}^{\text{reg}}_W(q) = \{
(y, \eta) \in L^+W: \ &y = \gamma_{q,\xi}(t), \eta = \dot{\gamma}_{q,\xi}(t), \\
& \text{for some } \xi \in L^+_q M, \ 0 < t < \rho(x,\xi)\}.
\end{align*}
It is shown in \cite{Kurylev2018} that $\EWq = \pi(\mathcal{C}_W(q))$ and 
one can define the regular earliest light observation set of $q$ as $\EWqr = \pi(\mathcal{C}^{\text{reg}}_W(q))$.
The lower semicontinuity of $\rho(x,\xi)$ implies that $\EWqr \subset W$ is a smooth submanifold of dimension $3$. 

Suppose $\Lambda$ is a conic Lagrangian submanifold in $T^*\tM$ away from the zero section. 
We denote by $I^\mu(\Lambda)$ the set of Lagrangian distributions in $\tM$ associated with $\Lambda$ of order $\mu$. 
In local coordinates, a Lagrangian distribution can be written as an oscillatory integral and we regard its principal symbol, 
which is invariantly defined on $\Lambda$ with values in the half density bundle tensored with the Maslov bundle, as a function in the cotangent bundle. 

If $\Lambda$ is a conormal bundle of a submanifold $K$ of $\tM$, i.e. $\Lambda = N^*K$, then such distributions are also called conormal distributions. 
The space of distributions in $\tM$ associated with two cleanly intersecting conic Lagrangian manifolds $\Lambda_0, \Lambda_1 \subset T^*\tM \setminus 0$ is denoted by $I^{p,l}(\Lambda_0, \Lambda_1)$. 
If $u \in I^{p,l}(\Lambda_0, \Lambda_1)$, then one has $\wfset{(u)} \subset \Lambda_0 \cup \Lambda_1$ and 
\[
 u \in I^{p+l}(\Lambda_0 \setminus \Lambda_1), \quad  u \in I^{p}(\Lambda_1 \setminus \Lambda_0)
\]
away from their intersection $\Lambda_0 \cap \Lambda_1$. The principal symbol of $u$ on $\Lambda_0$  and  $\Lambda_1$ can be defined accordingly and they satisfy some compatible conditions on the intersection. 

For more detailed introduction to Lagrangian distributions and paired Lagrangian distributions, see \cite[Section 3.2]{Kurylev2018} and \cite[Section 2.2]{Lassas2018}. 
The main reference are \cite{MR2304165, Hoermander2009} for conormal and Lagrangian distributions and
\cite{Hoop2015,Greenleaf1990,Greenleaf1993,Melrose1979} for paired Lagrangian distributions. 

In particular, the causal inverse $\tQ_g$ of the wave operator $\square_g$ in $\tM$ is a paired Lagrangian distribution by \cite{Melrose1979} (see also \cite{Hoop2015, Baer2007, Greenleaf1993}).
Here $\tQ_g$ is the solution operator in the microlocal sense to the linear wave equation
\begin{align*}
\square_g v &= f, \quad \text{ on } \tM,\\
v & = 0, \quad \text{ on } \tM \setminus J^+(\supp(f)).
\end{align*}
If 
additionally we assume $\Lambda$ intersects $\Char(\square_g)$ transversally such that its intersection with each bicharacteristics has finite many times, then 
\[
\tQ_g: I^\mu(\Lambda) \rightarrow I^{p,l}(\Lambda, \Lambda^g),
\]
where $ \Lambda^g$ is the flow out of $\Lambda \cap \Char(\square_g)$ under the Hamiltonian flow of $\square_g$, according to \cite[Proposition 2.1]{Greenleaf1993}, 
see also \cite[Proposition 2.1]{Lassas2018} for more details.

%% file: file/planewaves.tex
\subsection{Distorted Plane Waves}\label{subsec_plane}
We briefly review the distorted plane waves defined in \cite{Kurylev2018} (see also \cite{Lassas2018, Wang2019, Hintz2020}).
Roughly speaking, before the first conjugate point they are conormal distributions propagating along the fixed null geodesic.

Let $g^+$ be a Riemannian metric on $\tM$.
For $x_0 \in \tM \setminus M$, $\xi_0 \in L^+_{x_0} \tM$, and a small parameter $s_0 >0$, 
we define 
\begin{align*}
\mathcal{W}({x_0, \xi_0, s_0}) &= \{\eta \in L^+_{x_0} \tM: \|\eta - \xi_0\|_{g^+} < s_0\}
\end{align*}
as a neighborhood of $\xi_0$ at the point $x_0$ and 
\begin{align*}
K({x_0, \xi_0, s_0}) &= \{\gamma_{x_0, \eta}(s) \in \tM: \eta \in \mathcal{W}({x_0, \xi_0, s_0}), s\in (0, \infty) \},
\end{align*}
as the subset of the light cone emanating from $x_0$ by light-like vectors in $\mathcal{W}({x_0, \xi_0, s_0})$. 
As $s_0$ goes to zero, the surface $K({x_0, \xi_0, s_0})$ tends to the geodesic $\gamma_{x_0, \xi_0}((0, \infty) )$. 
We define
\begin{align*}
\Lambda({x_0, \xi_0, s_0})
 = &\{(\gamma_{x_0, \eta}(s), r\dot{\gamma}_{x_0, \eta}(s)^b )\in \tM: \\
 & \quad \quad \eta \in \mathcal{W}({x_0, \xi_0, s_0}), s\in (0, \infty), r \in \mathbb{R}\setminus \{0 \} \}
\end{align*}
as the flow out from $\{(x_0, r \eta^b )\in T^*M: \eta \in \mathcal{W}({x_0, \xi_0, s_0}) \}$ by the Hamiltonian vector field of $\square_g$ in the future direction. 
Note that $\Lambda({x_0, \xi_0, s_0})$ is the conormal bundle of $K({x_0, \xi_0, s_0})$ near $\gamma_{x_0, \xi_0}$ before the first conjugate point of $x_0$. 

Now let 
\[
x_j \in V \text{ with } (x_j, \xi_j) \in L^+\tM, \quad j=1,2,3,4
\] 
be four lightlike covectors 
and we denote this quadruplet by $(\vec{x}, \vec{\xi})$.
Note the null geodesic $\gamma_{x_j, \xi_j}(t)$ starting from $x_j \in V$ could never intersect $M$ or could enter $M$ more than once. 
Thus, we define
\begin{align}\label{def_bpep}
t_j^0  = \inf\{s > 0 : \  \gamma_{x_j, \xi_j}(s) \in M \}, \quad t_j^b  = \inf\{s > t_j^0 : \  \gamma_{x_j, \xi_j}(s) \in \tM \setminus M \}
\end{align}
as the first time when it enters $M$ and 
the first time when it leaves $M$ from inside, 
if such limits exist. 
In the following, we focus on  $(\vec{x}, \vec{\xi})$ 
that satisfies  
\begin{enumerate}[\hspace{0.5cm}(\ref{assp_causalindep}-a)] \label{assp_causalindep}
    \item 
    each $\gamma_{x_j, \xi_j}(t)$ hits $\partial M$ transversally at $t_j^0$,
    \item the intersection points of the four geodesics on $\partial M$ at $t_j^0$ are causally independent; more strictly we require
    \begin{align*}
    \gamma_{x_j, \xi_j}(t_j^0) \notin J^+(x_k), \quad 1 \leq j \neq k \leq 4.
    \end{align*}
\end{enumerate}
We emphasize that these two assumptions are open conditions.
Given $(\vec{x}, \vec{\xi})$ satisfying them, we can find small perturbation $(\vec{x}', \vec{\xi}')$ such that they are true for $(\vec{x}', \vec{\xi}')$. 
Roughly speaking, the second assumption helps us to avoid the complications caused by the interaction of the distorted plane waves at boundary points, see Lemma \ref{lm_intM} and \ref{lm_w1}. 
On the other hand, by \cite[Propositon 2.4]{Hintz2017}, 
if there a null geodesic $\gamma_{x_j, \xi_j}$ passing a point $q \in \intM$,
then it 
must hits $\partial M$ transversally at some point before $q$, 
since $\partial M$ is (strictly) null-convex. 

\begin{figure}[h]
    \centering
    \includegraphics[height=0.25\textwidth]{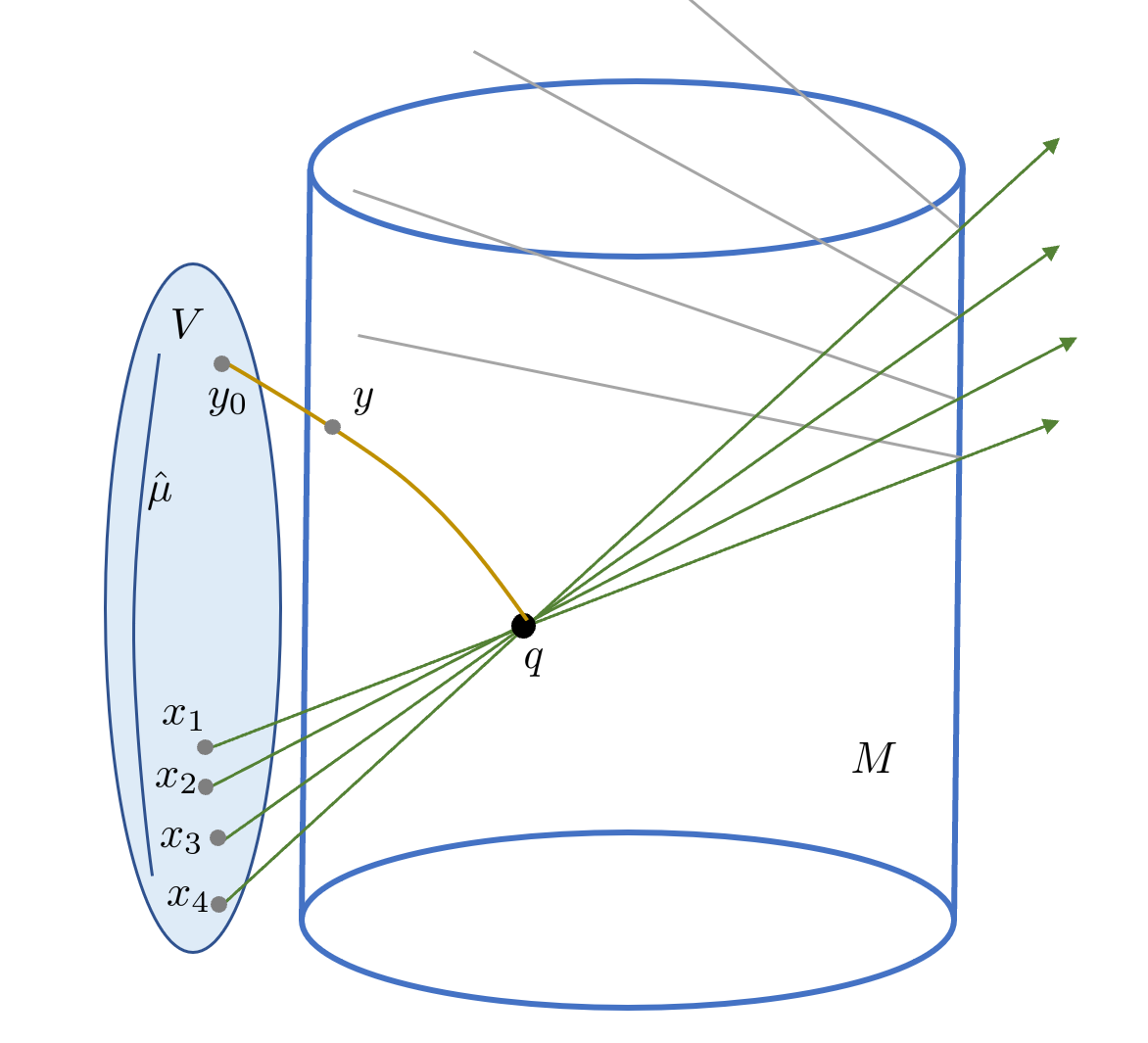}
    \caption{The interaction of four distorted plane waves at $q$. Distorted plane waves satisfying (\ref{assp_causalindep}) are sent from the observation set $V$, the light blue region. The interaction produces new singularities at $q$ propagating to $\yu$ and these singularities can be observed at $y \in \partial M$. }
    \label{fig_setting}
\end{figure}

Next, according to \cite[Lemma 3.1]{Kurylev2018}, 
we can construct distributions 
\[
u_j \in I^\mu(\Lambda(x_j, \xi_j, s_0)) \text{ satisfying } \square_g u_j \in C^\infty(M)
\]
with nonzero principal symbol along $(\gamma_{x_j, \xi_j}(s), \dot{\gamma}_{x_j, \xi_j}(s)^b )$. 
Note that $u_j \in \mathcal{D}'(\tM)$ has no singularities conormal to $\partial M$. 
Thus its restriction to the submanifold $\partial M$ is a well-defined, see \cite[Corollary 8.2.7]{Hoermander2003}.
Let $f_j = \partial_\nu u_j|_{\partial M}$ and $v_j$ solve the boundary value problem
\begin{equation}\label{eq_v1}
\begin{aligned}
\square_g v_j &= 0, & \  & \mbox{on } M ,\\
\partial_\nu v_j &= f_j, & \ &\mbox{on } \partial M,\\
v_j &= 0,  & \  &\mbox{for } t <0.
\end{aligned}
\end{equation}
It follows that $v_j  = u_j \mod C^\infty(M)$. 
We consider the boundary value problem (\ref{ihmNBC}) for the semilinear wave equation with the Neumann boundary condition 
$
f = \sum_{j=1}^4 \ep_j f_j, 
$
and write the solution $u$ to (\ref{ihmNBC}) as an asymptotic expansion with respect to $v_j$ in Section \ref{Sec_aexpan}. 


For convenience, we sometimes denote  $\gamma_{x_j, \xi_j}$ by $\gamma_j$, $K({x_j, \xi_j, s_0})$  by $K_j$, and $\Lambda({x_j, \xi_j, s_0})$ by $\Lambda_j$ if there are no confusion. 
To analyze the wave front set of the interactions of distorted plane waves, we introduce the notations
\[
\Lambda_{ij} = N^*(K_i \cap K_j), \quad \Lambda_{ijk} =  N^*(K_i \cap K_j \cap K_k), \quad \Lambda_q = T^*_q M \setminus 0,
\] 
if such $q$ as the intersection point in $\capgamma$ exists. 
We define
\[
\Lambda^{(1)} = \cup_{j=1}^4 \Lambda_j, \quad \Lambda^{(2)} = \cup_{i<j} \Lambda_{ij}, \quad 
\Lambda^{(3)} =  \cup_{i<j<k} \Lambda_{ijk}.
\]
Then we denote the flow out of $\Lambda^{(3)} \cap \Char(\square_g)$ under the null bicharacteristics of $\square_g$ in $T^*\tM$
by  
\[
\Lambda^{(3), g} = \{(z, \zeta) \in T^*\tM; \ \exists \ (y, \eta) \in \Lambda^{(3)}\cap \Char(\square_g) \text{ such that } (z, \zeta) \in \Theta_{y, \eta} \}.
\] 
According to \cite{Kurylev2018}, the Hausdorff dimension of $ \pi(\Lambda^{(3), g})$ is at most $2$, where $\pi: T^*\tM \rightarrow \tM$ is the natural projection.  
The flow out of $\Lambda^{(3)} \cap \Char(\square_g)$ under the broken bicharacteristic arcs of $\square_g$ in $T^*M$ is denoted by 
\[
\Lambda^{(3), b} = \{(z, \zeta) \in T^*M; \ \exists \ (y, \eta) \in \Lambda^{(3)} \text{ such that } (z, \zeta) \in \Theta^b_{y, \eta}
\},
\]
see Section \ref{subsec_Qg} for the broken characteristics and the definition of $\Theta^b_{y, \eta}$. 
Basically, we can regard $\pi(\Lambda^{(3), b})$ as a union of $\pi(\Lambda^{(3), g})$ and its reflection part, which implies the Hausdorff dimension of is $\pi(\Lambda^{(3), b})$ at most $2$ as well. 
Moreover, we define
\begin{align}\label{def_Lambda3bp}
\Lambda^{(3), b, +} = \{(z, \zeta) \in T^*\tM; \ \exists \ (y, \eta) \in \Lambda^{(3),b} \text{ such that } (z, \zeta) \in \Theta^+_{y, \eta}
\}
\end{align}
by extending the broken bicharacteristic arcs to $T^*\tM$ in the forward direction. 
Let 
\begin{align*}
\Gamma({\vec{x}, \vec{\xi}}, s_0) =  \Lambda^{(1)} \cup \Lambda^{(2)} \cup \Lambda^{(3)} \cup \Lambda^{(3),b, +},
\end{align*}
which depends on the parameter $s_0$ by definition. 
Then we define
\begin{align}\label{def_Gamma}
\unionGamma = \bigcap_{s_0>0}\Gamma({\vec{x}, \vec{\xi}}, s_0)
\end{align}
as the set containing all possible singularities caused by the Hamiltonian flow 
and the interactions of at most three distorted plane waves. 
Its projection $\pi(\unionGamma )$ on $\tM$  is relatively small (with Hausdorff dimension at most $2$) compared to the singularities caused by the interaction of four waves. 

As in \cite{Kurylev2018}, to deal with the complications caused by the cut points, we consider the interactions in the set
\begin{align}\label{def_nxxi}
\nxxi  = \tM \setminus \bigcup_{j=1}^4 J^+(\gamma_{x_j, \xi_j}(\rho(x_j, \xi_j))),
\end{align}
which is the complement of the causal future of the first cut points. 
To deal with the complications caused by the reflection part, for $j = 1, \ldots, 4$, we define
\begin{align}\label{def_ntxxi}
\ntxxi  = \tM \setminus \bigcup_{j=1}^4 J^+(\gamma_{x_j, \xi_j}(t_j^b)),
\end{align}
as the complement of the causal future of the point at $t=t_j^b$, see (\ref{def_bpep}), where $\gamma_{x_j, \xi_j}$ leaves $M$ from the inside for the first time. 
Note that in $\ntxxi \cap \nxxi$, each null geodesic $\gamma_{x_j, \xi_j}$ enters $M$ at most once and any two of them intersect at most once.
Moreover, in this open set we can show that all intersections happen in $\intM$ 
when $s_0$ is small enough.
Here $\intM$ is the interior of $M$.
\begin{lm}\label{lm_intM}
Suppose $(p, \zeta) \in \Lambda^{(2)} \cup \Lambda^{(3)}$ with arbitrarily small $s_0$.
If  $p \in \ntxxi$,  then we must have $p \in \intM$. 
\end{lm}
\begin{proof}
We prove by contradiction. 
Assume there is $p \in \tM \setminus \intM$ 
such that $(p, \zeta) \in \Lambda^{(2)} \cup \Lambda^{(3)}$. 
Then we can find $K_i, K_j$ such that $p$ is in their intersection. 
If $p \notin \gamma_{x_i, \xi_i}(\mathbb{R}_+) \cap \gamma_{x_j, \xi_j}(\mathbb{R}_+) $, then by choosing small enough $s_0$ one has $p \notin K_i \cap K_j$. 
This implies $p = \gamma_{x_i, \xi_i}(s_i) =  \gamma_{x_j, \xi_j}(s_j)$ for some $s_i, s_j>0$. 
In $\ntxxi$, we must have 
$s_i \leq t_i^0$, where recall $t_i^0$ is the earliest time when $\gamma_{x_i, \xi_i}$ hits the boundary.
Since $p \in \gamma_{x_j, \xi_j}$, we have
\[
x_j \leq p  = \gamma_{x_i, \xi_i}(s_i)  \leq \gamma_{x_i, \xi_i}(t_i^0),
\]
which implies that $\gamma_{x_i, \xi_i}(t_i^0)$ is contained in the casual future of $(x_j, \xi_j)$.
This contradicts with the assumption (\ref{assp_causalindep}). 
%
\end{proof}
Recall $V$ is the given observation set in $\tM \setminus M$. 
The following lemma shows a covector $(y, \eta)$ with $y \in \partial M$ belongs to $\nxxi$ (or $ \ntxxi$), if this is true for a covector in $V$ given by the forward null bicharacteristics. 
\begin{lm}\label{connecty}
    Let $y \in \pmt$ and $\eta \in L_y^{*,+} \tM$ such that $\yu = \gamma_{y, \eta^\#}(t_0), \etau = (\dot{\gamma}_{y, \eta^\#}(t_0))^b$ with $\yu \in V$ for some $t_0 > 0$. 
    We have following statements.  
    \begin{enumerate}[(1)]
        \item If $\yu \in \nxxi$, then $y \in \nxxi$.
        \item If $\yu \in \ntxxi$, then $y \in \ntxxi$.
        \item Suppose $\yu \in \ntxxi$. If $(\yu, \etau) \notin \unionGamma$, then $(y, \eta) \notin \unionGamma$.
    \end{enumerate}
    \begin{proof}
        We prove the first two statements by contradiction arguments.  
        For (1), assume 
        $y \notin \nxxi$, there exists some $x_j$ such that $y \in J^+(\gamma_{x_j, \xi_j}(\rho(x_j, \xi_j)))$. 
        Since 
        $\eta^\# \in L_y^+ \tM$ and $\yu = \gamma_{y, \eta^\#}(t_0)$, we have $y_0 \in J^+(y)$. 
        Notice the casual relation $\leq$ is transitive. It follows that $\yu \in J^+(\gamma_{x_j, \xi_j}(\rho(x_j, \xi_j)))$, which contradicts with the assumption  $\yu \in \nxxi$. 
        The same arguments proves (2). 
        
        For (3), we show $(y, \eta) \in \unionGamma$ implies that $(\yu, \etau) \in \unionGamma$.
        Indeed, with $y \in \partial M$,  by Lemma \ref{lm_intM} one has $(y,\eta) \notin \Lambda^{(3)}$.
        If $(y, \eta) \in \unionGamma$, then we have 
        $(y, \eta)$ belongs to either  $\Lambda^{(1)}$ or $\Lambda^{(3),b,+}$, 
        To show $(y, \eta) \in \Lambda^{(1)}$ implies that $(\yu,\etau) \in \Lambda^{(1)}$ for small enough $s_0$, we can extend $\gamma_{x_j, \xi_j}$ if needed. 
        Then note that $(y,\eta) \in T^*M$ belongs to $\Lambda^{(3), b, +}$ exactly when it is in $\Lambda^{(3), b}$. 
        By its definition in (\ref{def_Lambda3bp}), we have 
        $(\yu, \etau)\in \Theta_{y, \eta} \subset \Lambda^{(3), b,+}$. 
    \end{proof}
\end{lm}
Furthermore, we prove the following lemma that describes the connection between the set $\ntxxi$ and the earliest observation set. 
\begin{lm}\label{lm_ntxxi}
    Let $(\capgamma) \cap \nxxi = \{ q \}$ and let $U$ be an open set in $\tM \setminus M$ containing $\yu$. 
    Suppose $y \in \pmt$ and $\eta^\# \in L_y^+ \tM$ such that $\yu = \gamma_{y, \eta^\#}(t_0), \etau = (\dot{\gamma}_{y, \eta^\#}(t_0))^b$ for some $t_0 > 0$. 
    Then the following statements are true.
    \begin{enumerate}[(1)]
        \item Assume $\yu \notin \uniongamma$. If $\yu \notin \ntxxi$, then $\yu \notin \mathcal{E}_U(q)$. 
        \item We have $J^+(\gamma_j(t_j^b)) \subset J^+(q)$, for $j = 1,2,3,4$. 
    \end{enumerate}
\end{lm}
\begin{proof}   
    For (1), if $\yu \notin \ntxxi$, 
    then there exists $1 \leq j \leq 4$ such that $\yu \in J^+(\gamma_{x_j, \xi_j}(t_j^b))$. 
    For convenience, we denote $\gamma_{x_j, \xi_j}(t_j^b)$ by $p_b$ in the following.
    With $\yu > p_b$, there are two possibilities, either $\tau(p_b, \yu) > 0$ or $\tau(p_b, \yu)  = 0$. 
    In the first case, one has
    \[
    \tau(q, \yu) \geq \tau(q,p_b) + \tau(p_b, \yu) > 0.
    \]
    By \cite[Lemma 14.19]{Neill1983}, there is a time-like geodesic connecting $q$ and $\yu$. 
    The definition (\ref{def_eob}) of the earliest observation set implies that $\yu \notin \mathcal{E}_U(q)$. 
    In the second case, there is a future pointing light-like geodesic $\gamma_{p_b, \zeta}(t)$ from $p_b$ to $\yu$. 
    Notice $\gamma_{x_j, \xi_j}(t)$ is a future pointing light-like geodesic from $q$ to $p_b$. 
    We would like to use the short-cut argument in \cite[Section 2.1]{Kurylev2018} to show that $\tau(q, y)>0$. 
    It suffices to check that $\dot{\gamma}_{x_j, \xi_j}(t_j^b)$ is not proportional to $\dot{\gamma}_{p_b, \zeta}(0)$ at $p_b \in \partial M$, and therefore we have a union of two geodesics that cannot be one light-like geodesics by reparameterization. 
    This is true with the assumption $\yu \notin \uniongamma$. 
    It follows that
    $
    \tau(q, \yu)> 0
    $
    and therefore $\yu \notin \mathcal{E}_U(q)$. 
    \noindent
    For (2), it is a result of the causal relation $\gamma_j(t_j^b) > q$. 
\end{proof}


%
%
%

%% file: file/assymptotic.tex
\subsection{Asymptotic expansion}\label{Sec_aexpan}
Recall the Taylor expansion of $w(x,u,\xi)$ in (\ref{def_taylorw}). 
For the moment, we do not assume $\mathcal{N}_0, \mathcal{N}_1$ are null forms and leave this to be assumed later. 

Let $f = \sum_{j = 1}^J \epsilon_j f_j$, where $J = 3$ or $4$. 
The small boundary data $f_j$ are properly chosen as before for $j = 1, \dots, J$. 
Let $v_j$ solve the boundary value problem (\ref{eq_v1})
and we write $w  = Q_g(h)$ if $w$ solves the boundary value problem
\begin{equation}\label{bvp_qg}
\begin{aligned}
\square_g w &= h, & \  & \mbox{on } M ,\\
\partial_\nu w &= 0 , & \  & \mbox{on } \partial M,\\
w &= 0, & \  & \mbox{for } t <0.
\end{aligned}
\end{equation}
Let $v = \sum_{j=1}^J \ep_j v_j$ and we have
\begin{align}\label{expand_u}
u &= v - Q_g{[w(x,u(x), \nabla_g u(x))]}\nonumber \\
& = v - 
\underbrace{ Q_g[\mathcal{N}_0(\nabla_g u, \nabla_g u) + u \mathcal{N}_1 (\nabla_g u, \nabla_g u) + u^2 \mathcal{M}(\nabla_g u, \nabla_g u) + o(|u|^2 \cdot |\xi|^2)]}_{\text{denote it by } A \coloneqq A_2 + A_3 + A_4 + \dots},
\end{align}
where we rewrite the term $A$ by the order of $\epsilon$-terms, such that $A_2$ denotes terms with $\epsilon_i \epsilon_j$, 
$A_3$ denotes terms with $\epsilon_i \epsilon_j \epsilon_k$, and $A_4$ denotes terms with $\epsilon_i \epsilon_j \epsilon_k \epsilon_l$, for $1 \leq i,j,k,l \leq J$. 
Notice the term $v$ appears $j$ times in each $A_j$, for $j = 2, 3, 4$. 
Therefore, we introduce the notation $A_2^{ij}$ to denote the result if we replace $v$ by $v_i, v_j$ in $A_2$ in order, and similarly the notations
$A_3^{ijk}$, $A_4^{ijkl}$
such that 
\[
A_2 = \sum_{i,j} \ep_i \ep_j A_2^{ij}, \quad 
A_3 = \sum_{i,j, k} \ep_i\ep_j\ep_k  A_3^{ijk},\quad  
A_4 =\sum_{i,j, k,l} \ep_i\ep_j\ep_k \ep_l A_4^{ijkl}.
\] 
By (\ref{expand_u}), we have
\begin{align}\label{eq_A}
\begin{split}
	&A_2^{ij} = Q_g[\mathcal{N}_0(\nabla_g v_i, \nabla_g v_j)],\\
	&A_3^{ijk} = (Q_g[\mathcal{N}_0(\nabla_g v_i, \nabla_g A_2^{jk})] + Q_g[v_i \mathcal{N}_1(\nabla_g v_j, \nabla_g v_k)])\\
	&A_4^{ijkl} = (Q_g[\mathcal{N}_0(\nabla_g v_i, \nabla_g A_3^{jkl})] + Q_g[\mathcal{N}_0(\nabla_g A_2^{ij}, \nabla_g A_2^{kl})]\\
	&+ 2 Q_g[v_i \mathcal{N}_1(\nabla_g v_j, \nabla_g A_2^{kl})] 
	-  Q_g[A_2^{ij} \mathcal{N}_1(\nabla_g v_k, \nabla_g v_l)]
	+  Q_g[v_i v_j \mathcal{M}(\nabla_g v_k, \nabla_g v_l)]).
    \end{split}
\end{align}
In Section \ref{Sec_threewaves} and \ref{Sec_fourwaves},
we analyze the interactions of nonlinear waves microlocally by analyzing each terms above.

%% file: file/operatorQ.tex
\subsection{The solution operator $Q_g$}\label{subsec_Qg}
In this subsection, we study how the singularities propagate after applying $Q_g$, i.e. the solution operator to the boundary value problem (\ref{bvp_qg}).
We denote the outward (+) and inward (-) pointing tangent bundles by 
\begin{equation}\label{def_tbundle}
\TMpm = \{(x, v) \in \partial TM: \ \pm g(v, n)>0 \},
\end{equation}
where $n$ is the outward pointing unit normal of $\partial M$.
For convenience, we also introduce the notation 
\begin{equation}\label{def_Lbundle}
\LcMpm =\{(z, \zeta)\in L^* M \text{ such that } (z, \zeta^\#)\in \TMpm \}
\end{equation}
to denote the lightlike covectors that are outward or inward pointing on the boundary.

First we recall some definitions and notations in 
\cite{Melrose1978,Melrose1982,Vasy2008,MR2304165,melrose1992atiyah} (see also \cite{LEBEAU1997,Taylor1975}). For a smooth manifold $M$ with boundary $\partial M$, 
let $\bTM$ be the compressed cotangent bundle, see \cite[Lemma 2.3]{melrose1992atiyah}. 
There is a natural map \[
\pi_b: T^*M \rightarrow \bTM
\] 
satisfying that 
it is the identity map in $T^*\intM$ and for $p \in \partial M$ it has the kernel $N^*_p(\partial M)$ and the range identified with $T^*_p(\partial M)$. 
Suppose locally $M$ is given by $(\bx, x^n)$ with $x^n \geq 0$ and the one forms are given by $ \sum \bxi \diff \bx + \xi_n \diff x^n$. 
Then in local coordinates, $\pi_b$ takes the from 
\[
\pi_b(\bx,x^n,\bxi,\xi_n)  = (\bx,x^n,  \bxi, x^n\xi_n).
\]
The compressed cotangent bundle $\bTM$ can be decomposed into the elliptic, glancing, and hyperbolic sets by 
\begin{align*}
&\mathcal{E} = \{\mu \in \bTM \setminus 0,\  \pi_b^{-1}(\mu) \cap \Char(\square_g) = \emptyset \},\\
&\mathcal{G} = \{\mu \in \bTM \setminus 0,\  \text{Card}(\pi_b^{-1}(\mu) \cap \Char(\square_g)) = 1 \},\\
&\mathcal{H} = \{\mu \in \bTM \setminus 0,\  \text{Card}(\pi_b^{-1}(\mu) \cap \Char(\square_g)) = 2 \}.
\end{align*}
Recall $\Char(\square_g)$ is the null bicharacteristic set of $\square_g$ and 
we define its image 
\[
\compchar = \pi_b{(\Char(\square_g))}
\] as the compressed bicharacteristic set. 
Obviously the set $\compchar$ is a union of the glancing set $\mathcal{G}$ and the hyperbolic set $\mathcal{H}$.
If $\mu \in T^* \intM$ and $\mu \in \Char{(\square_g)}$, then $\mu$ is in $\mathcal{G}$.
Let $\gsetint$ be the subset of $\mathcal{G}$ containing such $\mu$.    
Note  $\gsetint$ can be identified with $ T^* \intM$. 
If $\mu \in T^* \intM$ and $\mu \notin \Char({\square_g)}$, then $\mu$ is in the elliptic set $\mathcal{E}$. 

The glancing set on the boundary $\mathcal{G} \setminus \gsetint$ is the set of all points $\mu \in \compchar$ such that $\pi_b^{-1}(\mu) \in T^*(\partial M)$ and the Hamilton vector field $H_p$ of $\square_g$ is tangent to $\partial M$ at $\pi_b^{-1}(\mu) $. 
We define $\mathcal{G}^k$ as the  subset where $H_p$ is tangent to $\partial M$ with the order less than $k$, for $k \geq 2$, see \cite[(3.2)]{Melrose1978} and \cite[Definiton 24.3.2]{MR2304165}. 
Note that $\mathcal{G} = \gsetint \cup \mathcal{G}^2$. 
In particular, we denote by $\mathcal{G}^\infty$ the subset with infinite order of bicharacteristic tangency. 
The subset $\mathcal{G}^2 \setminus \mathcal{G}^3$ is the union of the diffractive part $\mathcal{G}_d$ and the gliding part $\mathcal{G}_g$ depending on whether $H_p^2 x^n > 0 $ or $H_p^2 x^n < 0 $. 

Next we are ready to define the generalized broken bicharacteristic of $\square_g$, 
see \cite{Melrose1978, Melrose1982,MR2304165} and \cite[Definition 1.1]{Vasy2008}.
\begin{df}[{\cite[Definition 24.3.7]{MR2304165}}]\label{def_gb}
    Let $I \subset \mathbb{R}$ be an open interval and $B \subset I$ is a discrete subset. 
    A generalized broken bicharacteristic arc of $\square_g$ is a map $\nu: I \rightarrow \pi_b^{-1}(\mathcal{G} \cup \mathcal{H})$ satisfying the following properties:
    \begin{enumerate}[(1)]
        \item $\nu(t)$ is differentiable and $\nu'(t) = H_p(\nu(t))$, if $\nu(t) \in T^*\intM$ 
        \item $\nu(t)$ is differentiable and $\nu'(t) = H_p^G(\nu(t))$ , if $\nu(t) \in \pi_b^{-1}( \mathcal{G}^2 \setminus \mathcal{G}_d)$, see \cite[Definition 24.3.6]{MR2304165} for the vector field $H_p^G(\nu(t))$,
        \item every $t \in B$ is isolated, and $\nu(s) \in  T^*\intM$ if $|s-t|$ is small enough and $s \neq t$.
        The limits $\nu(t\pm 0)$ exist and are different points in the same fiber of $\partial T^*M$. 
    \end{enumerate}
The continuous curve $\dot{\nu}$ obtained by mapping $\nu$ into $\compchar$ by $\pi_b$  is called a generalized broken bicharacteristic. 
    \end{df}  

If $\nu$ is a generalized bicharacteristic arc contained in $\gsetint \cup \mathcal{H}$, then we call it a broken bicharacteristic arc, see \cite[Definiton 24.2.2]{MR2304165}, which is roughly speaking the union of bicharacteristics of $\square_g$ over $\intM$ with reflection points on the boundary. 
By the definition, a broken bicharacteristic arc arrives transversally to $\partial M$ at $\nu(t-0)$ and then leaves the boundary transversally from the reflected point $\nu(t+0)$, with the same projection in $T^*\partial M \setminus 0$. 
The image of such $\nu$ under $\pi_b$ is called a broken bicharacteristic.   

In this paper, with the assumption that $\partial M$ is null-convex, 
the generalized broken bicharacteristics we consider in Section \ref{Sec_threewaves} and \ref{Sec_fourwaves} are contained in $\gsetint \cup \mathcal{H}$, and therefore are all broken bicharacteristics. 
We denote
the broken bicharacteristic arc of $\square_g$ that contains the covector $(y, \eta) \in L^*M$ by $\Theta^b_{y, \eta}$. 
According to \cite[Corollary 24.3.10]{MR2304165}, the arc $\Theta^b_{y, \eta}$ is unique for each 
$\pi_b(y, \eta) \notin \mathcal{G}^\infty$. 
The following lemma shows that with proper assumptions on the boundary, a generalized bicharacteristics passing a point in $\gsetint \cup \mathcal{H}$ is always contained in $\gsetint \cup \mathcal{H}$, i.e., is always a broken bicharacteristic. 

\begin{lm}\label{lm_gh}
    Let $(M,g)$ be a Lorentzian manifold with timelike 
    boundary $\partial M$. 
    Suppose $\partial M$ is null-convex, see (\ref{def_nconvex}). 
    If $\dot{\nu}$ is a future pointing generalized bicharacteristic with $\dot{\nu}(0) \in \gsetint \cup \mathcal{H}$, then $\dot{\nu} \subset \gsetint \cup \mathcal{H}$ and therefore it is a future pointing broken characteristic. 
\end{lm}

\begin{proof}
    If $\dot{\nu}(0) \in \mathcal{H}$, then there are two different points in $\pi_b^{-1}(\dot{\nu})(0) \cap \Char({\square_g})$. 
    In this case suppose $\nu^\pm \in T^\pm_{\partial M}  M$ such that $\pi_b(\nu^\pm) = \dot{\nu}(0)$.
    By Definition \ref{def_gb}, the generalized bicharacteristic arc  $\nu $ of $\dot{\nu}$ has $\nu(0-) = \nu^+$ and $\nu(0+) = \nu^-$. 
    It leaves $\partial M$ transversally after $t >0$ and one has $\dot{\nu}(t) \in \gsetint$ for small enough positive $t$. 
    
    Now we consider the second case where $\dot{\nu}(0) \in \gsetint$. 
    The generalized bicharacteristic arc  $\nu $ of $\dot{\nu}$ has $\nu(0) \in T^* \intM$ and therefore near $\nu(0)$ it is the null bicharacteristic of $\square_g$. 
    Let $\pi: T^*M \rightarrow M$ be the natural projection. 
    Then $\pi(\nu)(t)$ near $t = 0$ is a light-like geodesic in $\intM$. 
    If $\pi(\nu)(t)$ hits the boundary at $\pi(\nu)(t_0) = p_0$, 
    then by \cite[Proposition 2.4]{Hintz2017} the null-convex boundary implies that the intersection is transversal with the tangent vector $v^+$ such that $(p_0, v^+) \in T^+_{\partial M} M$. 
    Let $v^- = v^+  - 2 g(v^+, n) n$ and it follows that
    $
    g(v^-, n) < 0, \  g(v^-, v^-) = 0
    $
    and $\pi_b(v^+) = \pi_b(v^-)$.
    Thus, $\dot{\nu}(t_0) \in \mathcal{H}$ and we arrive at the first case. 
    
    Combining the two cases, we have $\dot{\nu}$ is always contained in $\gsetint \cup \mathcal{H}$. 
    \end{proof}
Let $\dot{\mathcal{D}}'(M)$ be the set of distributions supported in $M$. 
The boundary wave front set $\wfset_b(u)$ of a distribution $u \in \dot{\mathcal{D}'}(M)$ is defined by 
\[
\wfset_b(u)  = \bigcap \pi_b(\Char (B)),
\]
where the intersection takes for any properly supported $B \in \Psi^0_b(M)$ and $Bu \in \dot{\mathcal{A}}(M)$.  
Here $\Psi^0_b(M)$ is the set of b-pseudodifferential operators on $M$ of order zero, for more details see \cite[Definition 18.3.18]{MR2304165}. 
The set $\dot{\mathcal{A}}(M) \equiv \bigcup_m I^m(M, \partial M)$ contains distributions that are smooth in $\intM$ and have tangential smoothness at $\partial M$, where $I^m(M, \partial M)$ denotes the class of conormal distributions on $\partial M$ of order $m$. 
Note that away from the boundary it coincides with the usual definition of wave from set, i.e. $\wfset_b(u)|_{\intM} = \wfset(u|_{\intM})$.  
We present the result about the singularities of solutions to the boundary value problem (\ref{bvp_qg}) in the following proposition according to \cite{Melrose1978, Melrose1982}, \cite[Theorem 8.7]{Vasy2008}, and  \cite{MR2304165}. 

\begin{pp} \label{pp_wfb}
    Let $h \in \dot{\mathcal{D}'}(M)$ and $w = Q_g(h)$ be the solution to the boundary value problem (\ref{bvp_qg}) in $M$. 
    Then \[
    \wfset_b(w) \setminus \wfset_b(h) \subset \compchar
    \] is a union of maximally extended generalized characteristics of $\square_g$. 
    
 \end{pp}
Combining this proposition and Lemma \ref{lm_gh}, we have the following corollary. 
\begin{corollary}
In particular, if $\partial M$ is timelike and null-convex 
with 
$\wfset_b(h)$ contained in $\gsetint \cup \mathcal{H}$, then 
\[
\wfset_b(w) \subset \gsetint \cup \mathcal{H}
\]
and is a union of broken bicharacteristics. 
\end{corollary}

%% file: file/newthree.tex
\section{The nonlinear interaction of three waves}\label{Sec_threewaves}
In this section, we consider the interaction of three distorted plane waves. 
Consider three vectors $(x_j, \xi_j)$ satisfying the assumption (\ref{assp_causalindep}) for $j=1,2,3$ (see also the construction in \cite[Section  3.2]{Hintz2020}). 
Let $K(x_j, \xi_j, s_0)$ and $\Lambda(x_j, \xi_j, s_0)$ be defined as in Section \ref{subsec_plane}. 
We assume the submanifolds $K_j \equiv K_j(x_j, \xi_j, s_0)$ intersect 3-transversally 
as in \cite{Kurylev2018,Lassas2018}, i.e.
\begin{itemize}
    \item[(1)] $K_i$ and $K_j$ intersect transversally at a codimension $2$ submanifold $K_{ij}$, for $1 \leq i < j \leq 3$;
    \item[(2)] $K_1, K_2, K_3$ intersect transversally at a codimension $3$ submanifold $K_{123}$. 
\end{itemize}
By Lemma \ref{lm_intM}, the causal independence assumption in (\ref{assp_causalindep}) implies that in $\nxxi \cap \ntxxi$,
with $s_0$ small enough, one has $K_{ij}, K_{ijk} \subset \intM$ and therefore
$
\pi_b(\Lambda_{ij}), \pi_b(\Lambda_{ijk}) \subset \gsetint
$
for $ 1 \leq i < j \leq 3$. 

We consider three distorted waves $u_j$ associated with $(\vec{x}, \vec{\xi})$ such that
\[
u_j \in I^{\mu}(\Lambda(x_j, \xi_j, s_0)), \quad j = 1,2,3,
\] 
solves the linearized wave problem in $\tM$, i.e.  $\square_g u_j \in C^\infty(M)$, with the principal symbol nonvanishing along $\gamma_{x_j, \xi_j}(\mathbb{R}_+)$. 
Let $f_j = \partial_\nu u_j|_{\partial M}$ and we use the Neumann data $f = \sum_{j = 1}^4 \ep_j f_j$ for the semilinear boundary value problem (\ref{ihmNBC}).
Consider $v_j$ solving (\ref{eq_v1}) and it follows that $v_j = u_j\mod C^\infty(M)$. 
We define
\[
\mathcal{U}^{(3)} = \partial_{\epsilon_1}\partial_{\epsilon_2}\partial_{\epsilon_3} u |_{\epsilon_1 = \epsilon_2 = \epsilon_3=0},
\]
and combine (\ref{expand_u}), (\ref{eq_A}) to have
\begin{align*}
\mathcal{U}^{(3)} 
&= - \sum_{(i,j,k) \in \Sigma(3)} A_3^{ijk} \\
&= - \sum_{(i,j,k) \in \Sigma(3)} Q_g[\mathcal{N}_0(\nabla_g v_i, \nabla_g Q_g[\mathcal{N}_0(\nabla_g v_i, \nabla_g v_j)])] + Q_g[v_i \mathcal{N}_1(\nabla_g v_j, \nabla_g v_k)].
\end{align*}
We emphasize that $v_j \in \dot{\mathcal{D}}'(M)$ and  
we can always identify it as en element of $\mathcal{D}'(\tM)$.
Particularly in $\nxxi \cap \ntxxi$, it can be identified as a conormal distribution in $\tM$. 
Thus, we can use the calculus of conormal distributions to analyze their multiplication, for example, $v_iv_j$ with $i \neq j$.
Note $v_iv_j$ has the wave front set away from $N^*\partial M$, since we have $K_{ij} \subset \intM$.
Then we treat $v_iv_j$ as a distribution supported in $M$ and we use Proposition \ref{pp_wfb} and its corollary to analyze the singularities of $Q_g(v_iv_j)$. 
We repeat this type of arguments to analyze the singularities 
and compute the principal symbol of $\mathcal{U}^{(3)}$ in the following.


In addition, we would like to relate the singularities of the ND map with those of $\mathcal{U}^{(3)}$ restricted on the boundary. 
Indeed, with $\mathcal{U}^{(3)}$  defined above we have  
\[
\epslamt  = \mathcal{U}^{(3)}|_{\partial M}. 
\]
Let $\mathcal{R}(\mathcal{U}^{(3)})$ be the trace of $\mathcal{U}^{(3)}$ on $\partial M$. 
Notice for any timelike covector $(y_|, \eta_|) \in T^* \partial M \setminus 0$, there is exactly one outward pointing lightlike covector $(y, \eta^+)$ and one inward pointing lightlike covector$(y, \eta^-)$ satisfying $y_| = y,\  \eta_| = \eta^\pm|_{T^*_{y} \partial M}$. 
It is shown in \cite{Hintz2020} that the trace operator $\mathcal{R}$ is an FIO of order ${1}/{4}$ with 
a nonzero principal symbol at such $(y_|, \eta_|, y, \eta^+)$ or $(y_|, \eta_|, y, \eta^-)$ .  

\subsection{The analysis of $A_2^{ij}$} 
First we analyze the singularities of the multiplication of two conormal distributions under the solution operator $Q_g$, i.e. 
 \[
A_2^{ij} = Q_g[\mathcal{N}_0(\nabla_g v_i, \nabla_g v_j)],
\] 
for $1 \leq i < j \leq 3$. 
By \cite[Lemma 3.3]{Lassas2018} and \cite{Wang2019}, one has 
\[
\mathcal{N}_0(\nabla_g v_i, \nabla_g v_j) \in I^{\mu+1, \mu+2}(\Lambda_{ij}, \Lambda_i) + I^{\mu+1, \mu+2}(\Lambda_{ij}, \Lambda_j),
\]
and away from 
$\Lambda_i$ and $\Lambda_j$, it has the principal symbol \[
(2\pi)^{-1}\mathcal{N}_0(\zeta^{(i)}, \zeta^{(j)})\sigma^{(p)}(v_i) (q, \zeta^{(i)}) \sigma^{(p)}(v_j) (q, \zeta^{(j)})
\] at the covector $(q, \zeta^{(i)}+ \zeta^{(j)}) \in \Lambda_{ij}$.
Note that $\mathcal{N}_0(\nabla_g v_i, \nabla_g v_j)$ is also a distribution supported in $M$.
Its boundary wave front set is contained in $\pi_b(\Lambda_i \cup \Lambda_j\cup {\Lambda_{ij}})$ and thus 
as a subset of $\mathcal{G}^{\text{int}} \cup \mathcal{H}$.
Then by Proposition \ref{pp_wfb} and its corollary, the set $\wfset_b(A_2^{ij})$ is contained in the union of $\pi_b(\Lambda_i \cup \Lambda_j\cup {\Lambda_{ij}})$ and their flow out under the broken bicharacteristics.
We notice that away from $J^+(\gamma_{x_i, \xi_i}(t_i^b)) $ and $ J^+(\gamma_{x_i, \xi_i}(t_j^b)) $, there are no new singularities produced by the flow out. 

More explicitly, in $\nxxi \cap \ntxxi$ we can identify $Q_g$ by $\tQ_g$, the causal inverse of $\square_g$ on $\tM$, to have
\[
A_2^{ij} \in I^{\mu, \mu+1}(\Lambda_{ij}, \Lambda_i) + I^{\mu, \mu+1}(\Lambda_{ij}, \Lambda_j),
\]
by \cite[Lemmma 3.4]{Lassas2018}. 
Here we regard the restriction of $A_2^{ij}$ to $\nxxi \cap \ntxxi$ as a distribution in  $ \tM$.
Additionally, 
if $(q, \zeta) \in \Lambda_{ij}$ is away from $\Lambda_i$ and $\Lambda_j$, then we have the principal symbol 
\begin{align*}
\sigma^{(p)}(A_2^{ij}) (q, \zeta) &= (2\pi)^{-1} |\zeta^{(i)} + \zeta^{(j)}|_{g^*}^{-2} \mathcal{N}_0(\zeta^{(i)}, \zeta^{(j)})\sigma^{(p)}(v_i) (q, \zeta^{(i)}) \sigma^{(p)}(v_j) (q, \zeta^{(j)}),
\end{align*}
with $\zeta = \zeta^{(i)} + \zeta^{(j)}$.

\subsection{The analysis of $A_3^{ijk}$}
The following proposition describes the singularities of $A_3^{ijk}$ produced by three waves interaction, see (\ref{eq_A}) for the explicit form of $A_3^{ijk}$, for $1 \leq i,j,k \leq 3$. 

\begin{pp}\label{pp_Aijk}
    Suppose the three submanifolds $K_i, K_j, K_k$ intersect 3-transversally at $K_{ijk}$. 
    Then in $\nxxi$ and $\ntxxi$ (for the definition see (\ref{def_nxxi},\ref{def_ntxxi})), we have the following statements. 
    \begin{enumerate}[(a)]
        \item  There is a decomposition $A_3^{ijk} = \tw_0 + \tw_1 + \tw_2 + \tw_3$ with 
        \begin{align}\label{eq_Aijk}
        \begin{split}
        &\tw_0 \in I^{3\mu + 3}(\Lambda_{ijk}), 
        \quad \wfset_b(\tw_1) \subset \pi_b((\Lambda_{ijk}^{g}(\epsilon) \cap \Lambda_{ijk}) \cup \Lambda_{ijk}^{b}),\\
        &\wfset_b(\tw_2) \subset \pi_b(\Lambda^{(1)} \cup (\Lambda^{(1)}(\epsilon) \cap \Lambda_{ijk}) \cup \Lambda_{ijk}^{b}),
        \quad  \wfset(\tw_3)\subset \Lambda^{(1)} \cup \Lambda^{(2)}.
        \end{split}
        \end{align}
        In particular, for $(q, \zeta) \in \Lambda_{ijk}$ the leading term $\tw_0$ has the principal symbol (\ref{eq_tw0ps}).
        \item Let $(y, \eta) \in \partial T^*M$ be a covector
        lying along the forward null-bicharacteristic starting at $(q, \zeta) \in \Lambda_{ijk}$. 
        Suppose $(y, \eta)$ is away from $\Lambda^{(1)}$. 
        Then the principal symbol of $\epslamt$ at $(y_|, \eta_|)$ is given in (\ref{eq_lambdaps}),
        where $(y_|, \eta_|)$ is the projection of $(y, \eta)$ on the boundary. 
    \end{enumerate}

\end{pp}
\begin{proof}
By (\ref{eq_A}), we have $A_3^{ijk}  = B_3^{ijk} + C_3^{ijk}$, where
\[
B_3^{ijk} = Q_g[\mathcal{N}_0(\nabla_g v_i, \nabla_g A_2^{jk})], \quad 
C_3^{ijk} = Q_g[v_i \mathcal{N}_1(\nabla_g v_j, \nabla_g v_k)]. 
\]
To analyze $B_3^{ijk}$, since we are away from $J^+(\gamma_{x_i, \xi_i}(t_i^b)) $ and $ J^+(\gamma_{x_i, \xi_i}(t_j^b))$,
first one can write
$\mathcal{N}_0(\nabla_g v_i, \nabla_g A_2^{jk}) = w_0 + w_1 + w_2$
 using \cite[Lemma 3.6]{Lassas2018}, where
\begin{align}\label{eq_Bijk}
\begin{split}
\ & w_0 \in I^{3\mu + 5}(\Lambda_{ijk}), \quad \wfset(w_2) \subset \Lambda^{(1)} \cup \Lambda^{(2)},\\
\ &  \wfset(w_1) \subset \Lambda^{(1)} \cup (\Lambda^{(1)}(\epsilon) \cap \Lambda_{ijk}).
\end{split}
\end{align}
The leading term $w_0$ has principal symbol 
\begin{align*}
\sigma^{(p)}(w_0) (q, \zeta) = (2\pi)^{-2}
\frac{\mathcal{N}_0(\zeta^{(i)}, \zeta^{(j)} + \zeta^{(k)})}{|\zeta^{(j)} + \zeta^{(k)}|^2_{g^*}} \mathcal{N}_0(\zeta^{(j)}, \zeta^{(k)})
\prod_{m= i,j,k}
\sigma^{(p)}(v_m) (q, \zeta^{(m)}),
\end{align*}
where $\zeta = \zeta^{(i)}+ \zeta^{(j)}+ \zeta^{(k)}$. 

Next, we apply $Q_g$ to $w_0, w_1, w_2$ respectively. 
Since each $\wfset_b(w_k) \subset \gsetint \cup \mathcal{H}$, then by Proposition \ref{pp_wfb} and its corollary, the set $\wfset_b(Q_g(w_k))$ is a union of  $\wfset_b(w_k)$ and its flow out under the broken bicharacteristics of $\square_g$. 
For $w_0$, it follows that
\[
\wfset_b(Q_g [w_0] )\subset\pi_b(\Lambda_{ijk} \cup \Lambda_{ijk}^{b}). 
\]
To find its principal symbol, 
we decompose it as
$ Q_g[w_{0}] = u^{\text{inc}} + u^{\text{ref}}$,
where 
$u^{\text{inc}} $ is the incident wave before the reflection on the boundary and $u
^{\text{ref}}$ is the reflected one. 
Recall $\tQ_g$ is the causal inverse of $\square_g$ on $\tM$.
Then $ u^{\text{inc}}  = \tQ_g[w_{0}]$ and by \cite[Proposition 2.1]{Greenleaf1993} we have
\begin{align}\label{r_qw0}
u^{\text{inc}} =\tQ_g[w_{0}] \in I^{3 \mu + \frac{7}{2}, -\frac{1}{2}}  ( \Lambda_{ijk} ,\Lambda_{ijk}^{g} ).
\end{align}
The principal symbol of $u^{\text{inc}}$ at the covector $(q, \zeta) \in \Lambda_{ijk}$ equals to
\[
\sigma^{(p)}(u^{\text{inc}})(q, \zeta) = 
(2\pi)^{-2}|\zeta^{(i)} + \zeta^{(j)} + \zeta^{(k)}|^{-2}_{g^*}
\sigma^{(p)}(w_0) (q, \zeta),
\]
where $\zeta = \zeta^{(i)} + \zeta^{(j)} + \zeta^{(k)}$. 
This is also the principal symbol of $ Q_g[w_{0}]$ at $(q, \zeta)$. 

If $(y, \eta) \in \partial T^*M$ lies along the forward null-bicharacteristic starting at $(q, \zeta) \in \Lambda_{ijk}$, then $(y,\eta)$ belongs to $\Lambda_{ijk}^b$ and is the first point there to hit the boundary.
The transversal intersection implies $\zeta$ has a unique decomposition $\zeta = \sum_{j=1}^3 \zj$ with $\zeta^{(j)} \in \Lambda(x_j, \zeta_j, s_0)$. 
Suppose $(y, \eta)$ is away from $\Lambda^{(1)} \cup \Lambda^{(2)} \cup \Lambda^{(3)}$.
The principal symbol of $u^{\text{inc}} $ at $(y, \eta)$ is
\begin{align*}
\sigma^{(p)}(u^{\text{inc}})(y, \eta)  = (2\pi)^{-2} \sigma^{(p)}(\tQ_g)(y, \eta, q, \zeta)   
\sigma^{(p)}(w_{0})(q, \zeta).
\end{align*}
The reflected wave $u^{\text{ref}}$ satisfies 
$\square_g u^{\text{ref}}= 0$ with $\partial_\nu  (u^{\text{inc}} +   u^{ \text{ref}})|_{\partial M} = 0$ near $(y, \eta)$ to satisfy the Neumann boundary condition. 
Following the analysis in \cite{Stefanov2018} and \cite{Hintz2020}, 
in a small conic neighborhood of $(y, \eta)$ we write
\[
u^{\bullet} = (2 \pi)^{-3} \int e^{i \phi^\bullet(x, \theta)} a^\bullet (x, \theta) \diff \theta, 
\]  
for $\bullet = \text{inc, ref}$, with suitable amplitude $a^\bullet (x, \theta)$ and the phase functions $\phi^\bullet(x, \theta)$ satisfying the eikonal equation 
and $\partial_\nu \phi^\text{inc} = - \partial_\nu \phi^\text{ref}$.
The Neumann boundary condition implies $\partial_\nu \phi^\text{inc} a^{\text{inc}} = - \partial_\nu \phi^\text{ref}a^{\text{ref}}$ and therefore $a^{\text{inc}} = a^{\text{ref}}$ on $\partial M$. 
This implies  
the principal symbols of the restrictions of $ u^{\text{inc}} $ and  $ u^{\text{ref}}$ to $\partial M$ agree. 
Consider the restriction $\mathcal{R}(Q_g[w_{0}])$ of $Q_g[w_{0}]$ on the boundary. 
The analysis above shows that
\begin{align}\label{eq_rqw0}
&\sigma^{(p)}(\mathcal{R}(Q_g[w_{0}]))(y_|, \eta_|) = 2
\sigma^{(p)} (\mathcal{R})(y_|, \eta_|, y, \eta)
\sigma^{(p)}(u^{\text{inc}})(q, \zeta) \\
=&2(2\pi)^{-2} 
\sigma^{(p)} (\mathcal{R})(y_|, \eta_|, y, \eta)
\sigma^{(p)}(\tQ_g)(y, \eta, q, \zeta)   
\frac{\mathcal{N}_0(\zeta^{(i)}, \zeta^{(j)} + \zeta^{(k)})}{|\zeta^{(j)} + \zeta^{(k)}|^2_{g^*}} \mathcal{N}_0(\zeta^{(j)}, \zeta^{(k)}) \nonumber \\
& \times \prod_{m= i,j,k} 
\sigma^{(p)}(v_m) (q, \zeta^{(m)}), 
\nonumber 
\end{align}
where $(y_|, \eta_|)$ is the projection of the lightlike covector $(y, \eta)$ on the boundary. 

For $Q_g[w_1]$, 
first we identify $w_1$ as an element in $\dot{D}'(M)$.  
Its boundary wave front set in $M$ is a subset of $\pi_b(\Lambda^{(1)}) \cup \pi_b(\Lambda^{(1)}(\epsilon) \cap \Lambda_{ijk})$. 
Then $\wfset_b(Q_g[w_1]) $ is contained in the union of this set and its flow out under broken bicharacteristics, i.e. $\pi_b(\Lambda^{(1)} \cup (\Lambda^{(1)}(\epsilon) \cap \Lambda_{ijk}) \cup \Lambda_{ijk}^{b})$,
away from $\bigcup_{j=1}^4 J^+(\gamma_{x_j, \xi_j}(t_j^b))$. 
When $\ep$ goes to zero, the boundary wave front set tends to $\pi_b(\Lambda^{(1)} \cup \Lambda_{ijk}^{b})$.

For $Q_g[w_2]$, the boundary wave front set of $w_2$ in $M$ is a subset of $\pi_b(\Lambda^{(1)} \cup \Lambda^{(2)})$. 
Then $\wfset_b(Q_g(w_2)) $ is contained in the union of this set and its flow out under the broken bicharacteristics, i.e. 
$\pi_b(\Lambda^{(1)} \cup \Lambda^{(2)})$, away from $\bigcup_{j=1}^4 J^+(\gamma_{x_j, \xi_j}(t_j^b))$.

Thus, we write $B_3^{ijk} =  \tw_0 + \tw_1 + \tw_2 + \tw_3$ with 
\begin{align*}
\begin{split}
&\tw_0 \in I^{3\mu + 3}(\Lambda_{ijk}), 
\quad   \wfset_b(\tw_1) \subset \pi_b((\Lambda_{ijk}^{g}(\epsilon) \cap \Lambda_{ijk}) \cup \Lambda_{ijk}^{b}),\\
&\wfset_b(\tw_2) \subset \pi_b(\Lambda^{(1)} \cup (\Lambda^{(1)}(\epsilon) \cap \Lambda_{ijk}) \cup \Lambda_{ijk}^{b}),
\quad  \wfset(\tw_3)\subset \Lambda^{(1)} \cup \Lambda^{(2)},
\end{split}
\end{align*}
where the first two terms $\tw_0, \tw_1$ come from applying \cite[Lemma 3.9]{Lassas2018} to (\ref{r_qw0}).
Indeed, \cite[Lemma 3.9]{Lassas2018} implies that we can decompose (\ref{r_qw0}) as a conormal distribution supported in $\Lambda_{ijk}$ and a distribution microlocally supported in $\Lambda_{ijk}^g(\epsilon)$.
Since $u^{\text{inc}}$ in (\ref{r_qw0}) has wave front set contained in $\Lambda_{ijk} \cup \Lambda_{ijk}^g$, the second distribution in the decomposition must be microlocally supported in $(\Lambda_{ijk}^{g}(\epsilon) \cap \Lambda_{ijk}) \cup \Lambda_{ijk}^{g}$. 
Then we include the wave front set after $\Lambda_{ijk}^{g}$ hits the boundary for the first time to have $\wfset_b(\tw_1)$ in (\ref{eq_Aijk}).

To analyze $C_3^{ijk}$, we write it as $Q_g[h]$. 
Note that $h$ has a decomposition with the same pattern as that of $B_3^{ijk}$ in (\ref{eq_Bijk}). 
The leading term is a conormal distribution in $I^{3\mu + 5}(\Lambda_{ijk})$, with principal symbol 
\begin{align}\label{eq_Cijk}
(2\pi)^{-2}\mathcal{N}_1(\zeta^{(j)}, \zeta^{(k)})
\prod_{m= i,j,k}
\sigma^{(p)}(v_m) (q, \zeta^{(m)})
\end{align}
for $(q, \zeta^{(i)} + \zeta^{(j)} + \zeta^{(k)}) \in \Lambda_{ijk}$. 
Then we apply $Q_g$ and follow the same analysis as above. 

Combining the analysis of $B_3^{ijk}$ and $C_3^{ijk}$, we conclude that we can write 
$A_3^{ijk} = \tw_0 + \tw_1 + \tw_2 + \tw_3$ with $\tw_j$ satisfying (\ref{eq_Aijk}) for $j = 0,1,2,3$. 
In particular, 
the leading term $\tw_0$ has principal symbol given by 
\begin{align}\label{eq_tw0ps}
\sigma^{(p)}(\tw_0)(q, \zeta)  = (2\pi)^{-2}
\frac{\mathcal{A}_3(\zeta^{(i)}, \zeta^{(j)}, \zeta^{(k)})}{|\zeta^{(i)} + \zeta^{(j)} + \zeta^{(k)}|^{2}_{g^*}}
\prod_{m= i,j,k}\sigma^{(p)}(v_m) (q, \zeta^{(m)}),
\end{align}
where
\begin{align*}
\mathcal{A}_3(\zeta^{(i)}, \zeta^{(j)}, \zeta^{(k)})  
=
\frac{\mathcal{N}_0(\zeta^{(i)}, \zeta^{(j)} + \zeta^{(k)})}{|\zeta^{(j)} + \zeta^{(k)}|^2_{g^*}} \mathcal{N}_0(\zeta^{(j)}, \zeta^{(k)}) + \mathcal{N}_1(\zeta^{(j)}, \zeta^{(k)}).
\end{align*}

Suppose $(y, \eta) \in \partial T^*M$
lies along the forward null-bicharacteristic starting at $(q, \zeta) \in \Lambda_{ijk}$ and is 
away from $\Lambda^{(1)} \cup \Lambda^{(2)} \cup \Lambda^{(3)}$.
From the analysis above, we combine (\ref{eq_rqw0}) and (\ref{eq_Cijk}) to have 
\begin{align} \label{eq_lambdaps}
&\sigma^{(p)}(\epslamt)(y_|, \eta_|)
= \sum_{(i,j,k) \in \Sigma(3)} \sigma^{(p)}(\mathcal{R}(B_3^{ijk} + C_3^{ijk})) \\
=&  2\sigma^{(p)} (\mathcal{R})(y_|, \eta_|, y, \eta)\sigma^{(p)}(\tQ_g)(y, \eta, q, \zeta) 
\mathcal{Q}(\zeta^{(1)}, \zeta^{(2)}, \zeta^{(3)}) 
 \prod_{m=1}^3\sigma^{(p)}(v_m) (q, \zeta^{(m)}) \nonumber, 
\end{align}
where $(y_|, \eta_|)$ is the projection of $(y, \eta)$ on the boundary and 
\begin{align*}
\mathcal{Q}(\zeta^{(1)}, \zeta^{(2)}, \zeta^{(3)}) &=  \sum_{(i,j,k) \in \Sigma(3)}  \mathcal{A}_3(\zeta^{(i)}, \zeta^{(j)}, \zeta^{(k)}) \\
&= \sum_{(i,j,k) \in \Sigma(3)}
\frac{\mathcal{N}_0(\zeta^{(i)}, \zeta^{(j)} + \zeta^{(k)})}{|\zeta^{(j)} + \zeta^{(k)}|^2_{g^*}} \mathcal{N}_0(\zeta^{(j)}, \zeta^{(k)}) + \mathcal{N}_1(\zeta^{(j)}, \zeta^{(k)}).
\end{align*}
\end{proof}
\subsection{The recovery of $\mathcal{N}_1(x, \xi)$}
Now we use the nonlinear interaction of three waves to recover the symmetric part of the second quadratic form $\mathcal{N}_1 (x, \xi)$ in the Taylor expansion (\ref{def_taylorw}), when the first one  $\mathcal{N}_0 (x, \xi)$ is null.
From the analysis above, the ND map determines the quantity 
$\mathcal{Q}$ at any $(\zeta^{(1)}, \zeta^{(2)}, \zeta^{(3)})$ in the set
\begin{align*}
\mathcal{Y}=\{(\zeta^{(1)}, \zeta^{(2)}, \zeta^{(3)}) \in (L^*_q M \setminus 0)^3: \ \zj \text{ are linearly independent and } \sum_{j=1}^3 \zj \text{ light-like} \}. 
\end{align*}
In the following, recall the assumption (\ref{assum_n0}) on the quadratic form $w(x,u,\xi)$ that $\mathcal{N}_0$ is null and $\mathcal{N}_1$ is not null. 
By \cite[Lemma 2.1]{Wang2019}, we can assume $\mathcal{N}_0(x, \xi) =C_0(x) g(\xi,\xi)$ in (\ref{def_taylorw}), since only the symmetric part matters there. 
Then a direct computation 
shows that 
\begin{align*}
\sum_{(i,j,k) \in \Sigma(3)}  
\frac{\mathcal{N}_0(\zeta^{(i)}, \zeta^{(j)} + \zeta^{(k)})}{|\zeta^{(j)} + \zeta^{(k)}|^2_{g^*}} \mathcal{N}_0(\zeta^{(j)}, \zeta^{(k)})
=2 C_0^2(x)g(\zeta, \zeta)  = 0
\end{align*}
for $\zeta = \sum_{j=1}^3 \zj$ lightlike. 
Thus, on the set $\mathcal{Y}$ one has
\[
\mathcal{Q}(\zeta^{(1)}, \zeta^{(2)}, \zeta^{(3)}) = \sum_{(i,j,k) \in \Sigma(3)}   \mathcal{N}_1(\zeta^{(j)}, \zeta^{(k)}) = \mathcal{N}_1(\zeta, \zeta) - \sum_{j=1}^3 \mathcal{N}_1(\zj, \zj). 
\]
\begin{lm}\label{lm_recoverN1}
Suppose $\mathcal{N}^{(1)}_1 (x, \xi)$ and $\mathcal{N}^{(2)}_1 (x, \xi)$ are two symmetric quadratic form such that 
\[
 \mathcal{N}^{(1)}_1(\zeta, \zeta) - \sum_{j=1}^3 \mathcal{N}^{(1)}_1(\zj, \zj) = \mathcal{N}^{(2)}_1(\zeta, \zeta) - \sum_{j=1}^3 \mathcal{N}^{(2)}_1(\zj, \zj)
\]
for any 
$(\zeta^{(1)}, \zeta^{(2)}, \zeta^{(3)}) \in \mathcal{Y}$ at $q$. 
Then we conclude that $\mathcal{N}^{(1)}_1 (q, \xi)$ equals to $\mathcal{N}^{(2)}_1 (q, \xi)$ up to null forms. 
\end{lm}
\begin{proof}
It suffices to show that if a symmetric quadratic form $\mathcal{N}_1$ satisfies $\mathcal{N}_1(\zeta, \zeta) - \sum_{j=1}^3 \mathcal{N}_1(\zj, \zj) = 0$  on $\mathcal{Y}$, then it is a null form. 
We follow the same ideas in the proof of \cite[Proposition]{Wang2019} and start from the set
\[
\widetilde{\mathcal{Y}}=\{(\zeta^{(1)}, \zeta^{(2)}, \zeta^{(3)}) \in (L^*_q M \setminus 0)^3: \ \zj \text{ are linearly independent} \}. 
\]
As an open subset of $ (L^*_q M )^3$, it is a smooth manifold of dimension $9$. 
Consider the smooth map $\mathcal{F}: \widetilde{\mathcal{Y}} \rightarrow \mathbb{R}$ given by ${\mathcal{F}}(\zeta^{(1)}, \zeta^{(2)}, \zeta^{(3)}) = g(\zeta, \zeta)$ with $\zeta =  \sum_{j=1}^3 \zj$. 
Then $\mathcal{Y} = \mathcal{F}^{-1}(0)$ and one can compute the Jacobian of $\mathcal{F}$ by 
\[
D_{(\zeta^{(1)}, \zeta^{(2)}, \zeta^{(3)})} \mathcal{F} = 2(G \zeta, G\zeta, G \zeta),
\]
where $G = (g^{ij}(x))$ is a $4 \times 4$ matrix. 
Since $G$ is non-degenerate, the map $\mathcal{F}$ has constant rank $1$. 
It follows that $\mathcal{Y}$ is a submanifold of dimension $8$. 
Now let the map $\mathcal{G}: \widetilde{\mathcal{Y}} \rightarrow \mathbb{R}$ be given by ${\mathcal{G}}(\zeta^{(1)}, \zeta^{(2)}, \zeta^{(3)}) =\mathcal{N}_1(\zeta, \zeta) - \sum_{j=1}^3 \mathcal{N}_1(\zj, \zj)$. 
It have the Jacobian
\[
D_{(\zeta^{(1)}, \zeta^{(2)}, \zeta^{(3)})} \mathcal{G} =2(GN_1G (\zeta -\zeta^{(1)}) , GN_1G (\zeta -\zeta^{(2)}), GN_1G (\zeta -\zeta^{(3)})),
\]
where $N_1 = (\mathcal{N}_1^{ij}(x))$ is a $4 \times 4$ matrix. 
With assumption $\mathcal{N}_1(\zeta, \zeta) - \sum_{j=1}^3 \mathcal{N}_1(\zj, \zj) = 0$  on $\mathcal{Y}$
we have $\mathcal{G}$ vanishes on $\mathcal{F}^{-1}(0)$.
Consider the map $f = (\mathcal{F}, \mathcal{G}): \widetilde{\mathcal{Y}} \rightarrow \mathbb{R}^2$ and we have $f^{-1}((0,0)) = \mathcal{Y}$ has dimension equal to $8$.
This implies $f$ has constant rank $1$, i,e.
$D \mathcal{F}$ and $D \mathcal{G}$ are linearly dependent. 
It follows there exists a constant $\lambda$ such that
\[
2GN_1G (\zeta -\zj) = 2 \lambda G \zeta, \text{ for } j = 1,2,3.
\]
We sum over $j$ to have $2GN_1G \zeta = 3 \lambda G \zeta$, for any lightlike $\zeta = \sum_{j=1} \zj$ with $(\zeta^{(1)}, \zeta^{(2)}, \zeta^{(3)}) \in \mathcal{Y}$. 
If follows that $\mathcal{N}_1(\zeta, \zeta) = \frac{3\lambda}{2} g(\zeta, \zeta) = 0$ for all lightlike vectors $\zeta$,
which implies it is a null form.

\end{proof}
This lemma and Proposition \ref{pp_Aijk} implies the the recovery of $\mathcal{N}_1$ up to null forms from the ND map with the assumption (\ref{assum_n0}), see Theorem \ref{thm_recoverNM}. 
Indeed, suppose  $\mathcal{N}^{(j)}_1 (x, \xi)$ is a symmetric quadratic form  in the Taylor expansion of $w^{(j)}(x, u, \xi)$, for $j = 1,2$. 
Consider the corresponding ND map $\Lambda_j$ for the semilinear boundary value problem (\ref{ihmNBC}) with the quadratic form  $w^{(j)}(x,u,\xi)$. 
If the ND map $\Lambda_j$ acting on $C^{m+1}([0,T] \times \partial N )$ are equal with $m \geq 5$, then we must have
$
\mathcal{N}^{(2)}_1 (x, \xi) = \mathcal{N}^{(2)}_1 (x, \xi) + c_1(x) g(\xi, \xi)$, for any $ x \in \mathbb{W}$ and $\xi \in T^*_x M$, according to (\ref{eq_lambdaps}).

%% file: file/newfour.tex
\section{The nonlinear interaction of four waves}\label{Sec_fourwaves}
In this section, we consider the interaction of four distorted plane waves. 
Recall $(\vec{x}, \vec{\xi})$ is the quadruplet of the covectors $(x_j, \xi_j)$ for $j = 1,2,3,4$, satisfying the assumptions (\ref{assp_causalindep}). 
Let $K_j$ 
be defined as in Section \ref{subsec_plane}.
We assume $K_1, K_2, K_3, K_4$ intersect 4-transversally  
as in \cite{Kurylev2018,Lassas2018}, i.e.
\begin{enumerate}[(1)]
    \item $K_i, K_j$ intersect transversally at a codimension $2$ manifold $K_{ij}$, for $i < j$;
    \item $K_i, K_j, K_k$ intersect at a codimension $3$ submanifold $K_{ijk}$, for $i < j < k$;
    \item $K_j$ for $j = 1, 2,3, 4$ intersect at a point $q$;
    \item for any two disjoint index subsets $I, J \subset \{1, 2, 3, 4\}$, 
    the intersection of $\cap_{i \in I} K_i$ and $\cap_{j \in J} K_j$ is transversal if not empty. 
\end{enumerate}
By Lemma \ref{lm_intM}, in $\nxxi \cap \ntxxi$ we must have $K_{ij}, K_{ijk}$ and $q$ contained in $\intM$. 
We consider four distorted waves $u_j$ associated with $(\vec{x}, \vec{\xi})$ such that
 \[
u_j \in I^{\mu}(\Lambda(x_j, \xi_j, s_0)), \quad j = 1,2,3,4
\] 
solves the linearized wave problem with the principal symbol nonvanishing along $\gamma_{x_j, \xi_j}$. 
Let $f_j = \partial_\nu u_j|_{\partial M}$ and we use the Neumann data $f = \sum_{j = 1}^4 \ep_j f_j$.
Consider $v_j$ solving (\ref{eq_v1}) and it follows that $v_j = u_j\mod C^\infty(M)$ as before. 
We define
\begin{align}\label{def_ufour}
\ufour = \partial_{\epsilon_1}\partial_{\epsilon_2}\partial_{\epsilon_3}\partial_{\epsilon_4} u |_{\epsilon_1 = \epsilon_2 = \epsilon_3 = \epsilon_4=0},
\end{align}
and combine (\ref{expand_u}, \ref{eq_A}) to have 
\begin{align}\label{u4}
& \ufour = - \sum_{(i,j,k,l) \in \Sigma(4)} \mathcal{A}_4^{ijkl} \\
&=- \sum_{(i,j,k,l) \in \Sigma(4)} 
\big( Q_g[\mathcal{N}_0(\nabla_g v_i, \nabla_g A_3^{jkl})] + Q_g[\mathcal{N}_0(\nabla_g A_2^{ij}, \nabla_g A_2^{kl})] +
 \nonumber \\
&  2 Q_g[v_i \mathcal{N}_1(\nabla_g v_j, \nabla_g A_2^{kl})] +  Q_g[A_2^{ij} \mathcal{N}_1(\nabla_g v_k, \nabla_g v_l)]
+  Q_g[v_i v_j \mathcal{M}(\nabla_g v_k, \nabla_g v_l)]
\big). \nonumber
\end{align}

The following proposition describes the singularities of $\ufour$ and those of the ND map. 
\begin{pp}\label{pp_ufour}
    Suppose the four submanifolds $K_1,K_2,K_3, K_4$ intersect 4-transversally at a point $q$. 
    Let $\ufour$ be defined in (\ref{u4}) and $\unionGamma$ defined in (\ref{def_Gamma}). Then in $\nxxi$ and $\ntxxi$ ((see (\ref{def_nxxi},\ref{def_ntxxi}))),  we have the following statements. 
    \begin{enumerate}[(a)]
        \item $\wfset_b(\ufour) \subset 
        \pi_b(\unionGamma \cup \Lambda_q \cup \Lambda_q^b)$. 
        \item Before the first reflection of $\Lambda_q^b$ on the boundary and away from $\unionGamma$, one can regard $\ufour$ as a paired Lagrangian distribution in $\tM$ with  
        \[
        \ufour \in I^{4 \mu + \frac{7}{2}, -\frac{1}{2}}  ( \Lambda_q ,\Lambda_q^{g} ).
        \]
        \item Let $(y, \eta) \in T^*M|_{\partial M}$ be a covector away from $\unionGamma$ and lying along the forward null-bicharacteristic of $\square_g$ starting at $(q, \zeta)$. 
        Then the principal symbol of $\epslam$ at $(y_|, \eta_|)$ is given by (\ref{eq_pslambdafour}), 
        where $(y_|, \eta_|)$ is the projection of $(y, \eta)$ on the boundary. 
    \end{enumerate}
\end{pp}
\begin{proof}
 We have five different types of terms for $\ufour$. 
 We denote them by $\mathcal{U}^{(4)}_j$, for $j = 1,\ldots, 5$. 
 We analyze each of them and compute the principal symbols in the five subsections \ref{subsec_U41d} - \ref{subsec_U45d} below. 
 For convenience, we use the notation $\epslamvec$ to denote $\epslam$ and 
 in \ref{subsec_U46d} we analyze its singularities.
 
\subsection{The analysis of $\mathcal{U}^{(4)}_1$}\label{subsec_U41d}
The most complicated term is the first one \[
\mathcal{U}^{(4)}_1 = \sum_{(i,j,k,l) \in \Sigma(4)}  Q_g[\mathcal{N}_0(\nabla_g v_l, \nabla_g A_3^{ijk})].
\]
With the decomposition of $A_3^{ijk} = w_0 + w_1 + w_2 + w_3$ given in (\ref{eq_Aijk}), we analyze the singularities of $\mathcal{N}_0(\nabla_g v_l, \nabla_g w_k)$ for $0 \leq k \leq 3$ in the following and then apply the solution operator $Q_g$ to each of them. Here we replace the notation $\tw_j$ by $w_j$ in the decomposition. 

    First, we compute $\mathcal{N}_0(\nabla_g v_l, \nabla_g w_0)$. 
    Since $w_0 \in I^{3\mu + 3}(\Lambda_{ijk})$, 
    by \cite[Lemma 3.10]{Lassas2018} one can write $\mathcal{N}_0(\nabla_g v_l, \nabla_g w_0)  = w_{0,1} + w_{0,2}$,
    where 
    \[
    w_{0,1} \in I^{4 \mu +5}(\Lambda_q),
    \quad  
    \wfset(w_{0,2}) \subset
    \Lambda_l(\epsilon) \cup \Lambda_{ijk} (\ep).
    \]
    The leading term $w_{0,1}$ has the principal symbol 
    \begin{align*}
    \sigma^{(p)}(w_{0,1})(q, \zeta)
     =(2\pi)^{-3} \mathcal{C}_1(\zeta^{(i)}, \zeta^{(j)}, \zeta^{(k)}, \zeta^{(l)})(\prod_{j=1}^4 \sigma^{(p)}(v_j) (q, \zeta^{(j)})),
    \end{align*}
    where $\zeta =\sum_{j=1}^4 \zeta^{(j)}$ and by (\ref{eq_tw0ps}) we have
    \begin{align}\label{C1}
    \mathcal{C}_1(\zeta^{(i)}, \zeta^{(j)}, \zeta^{(k)}, \zeta^{(l)}) 
    &= \frac{\mathcal{N}_0(\zeta^{(l)}, \zeta^{(i)}+ \zeta^{(j)}+ \zeta^{(k)})}{|\zi + \zj + \zk|^2_{g^*}} \times \\
    &(\frac{\mathcal{N}_0(\zeta^{(i)}, \zeta^{(j)} + \zeta^{(k)})}{|\zeta^{(j)} + \zeta^{(k)}|^2_{g^*}} \mathcal{N}_0(\zeta^{(j)}, \zeta^{(k)}) + \mathcal{N}_1(\zeta^{(j)}, \zeta^{(k)})). \nonumber
    \end{align}
    Then we apply $Q_g$ to $w_{0,1}$ and Proposition \ref{pp_wfb} with its corollary implies 
    \begin{align*}
    &\wfset_b(Q_g[w_{0,1}])
    \subset \pi_b(\Lambda_q \cup \Lambda_q^{b}) \subset \gsetint \cup \mathcal{H}.  
    \end{align*}
    To find the principal symbol,
    we decompose it as
    $ Q_g[w_{0,1}] = u^{\text{inc}} + u^{\text{ref}}$ as before,
    where 
    $u^{\text{inc}} $ is the incident wave before the reflection and $u
    ^{\text{ref}}$ is the reflected one. 
    Then $ w^{\text{inc}}  = \tQ_g[w_{0,1}]$ and we have
    \[
    w^{\text{inc}} \in I^{4 \mu + \frac{7}{2}, -\frac{1}{2}}  ( \Lambda_q ,\Lambda_q^{g} ).
    \]
    If $(y, \eta) \in T^*M|_{\partial M}$ lies along the forward null-bicharacteristic starting at $(q, \zeta)$, then it belongs to $\Lambda_q^b$ and is the first point where $\Lambda_q^b$ touches the boundary.
    Near $(y, \eta)$, the reflected wave satisfies $\square_g u^{\text{ref}} = 0$ and $(\partial_\nu u^{\text{inc}} +  \partial_\nu  u^{ \text{ref}})|_{\partial M} = 0 $. 
    We follow the same arguments as before, 
    and one can see 
    the principal symbols of the restrictions of $ u^{\text{inc}} $ and  $ u^{\text{ref}}$ to $\partial M$ agree. 
    The 4-transversal intersection assumption indicates each $\zeta \in \Lambda_q$ has a unique decomposition $\zeta = \sum_{j=1}^4 \zj$, with $\zj \in \Lambda(x_j, \xi_j, s_0)$. 
    Then the restriction of $Q_g[w_{0,1}]$ to the boundary has the principal symbol
    \begin{align}\label{eq_psrua}
    \sigma^{(p)}(\mathcal{R}(Q_g[w_{0,1}]))& (y_|, \eta_|) 
     = 2(2\pi)^{-2} \sigma^{(p)}(\mathcal{R})(y_|, \eta_|, y, \eta) \times\\
    & \sigma^{(p)}(\tQ_g)(y, \eta, q, \zeta)   
    \mathcal{C}_1(\zeta^{(i)}, \zeta^{(j)}, \zeta^{(k)}, \zeta^{(l)})(\prod_{j=1}^4 \sigma^{(p)}(v_j) (q, \zeta^{(j)})), \nonumber
    \end{align}
    if $(y, \eta)$ is away from $\unionGamma$. 
    
    For $Q_g[w_{0,2}]$,  
    we note that a broken bicharacteristic arc $\nu(t)$ depends on the initial point $\nu(0)$ in a continuous way, since from the definition it is the composition of the Hamiltonian flow of $\square_g$ and the reflection of transversal vectors on the boundary. 
    Thus, 
    the flow out of $\Lambda_l(\epsilon)$ 
    under the broken bicharacteristic arcs is a neighborhood of $\Lambda_l$, 
    and
    it tends to $\Lambda_l$ 
    when $\ep $ goes to $0$.
    Similarly the flow out of $\Lambda_{ijk}(\ep)$ under the broken bicharacteristic arcs tends to $\Lambda_{ijk}^b$ as $\ep $ goes to $ 0$. 
    
    Second, we analyze the singularities of $Q_g[\mathcal{N}_0(\nabla_g v_l, \nabla_g w_1)]$ in the following  lemma.
    \begin{lm}\label{lm_w1}
        As above assume we have  
        \[
        \wfset( v_l) \subset \Lambda_l, \quad \wfset_b( w_1) \subset \pi_b((\Lambda_{ijk}^g(\epsilon) \cap \Lambda_{ijk}) \cup \Lambda_{ijk}^{b}). 
        \]
        Then in $\nxxi$ and $\ntxxi$,
        the set
        $\wfset_b(Q_g(\mathcal{N}_0(\nabla_g v_l, \nabla_g w_1)))$ tends to a subset of 
        $\pi_b(\Lambda_l \cup \Lambda_{ijk} \cup \Lambda_{ijk}^{b})$, 
       as $\ep$ goes to zero.   
    \end{lm}
\begin{proof}
    First we consider all possible intersections of $K_l$ and 
    $\pi(\Lambda_{ijk}^b)$
    in $M$, where $\pi: T^*M \rightarrow M$ is the natural projection. 
    We claim that in $\nxxi \cap \ntxxi$ all intersections happens in $\intM$ 
    when $s_0$ is small enough. 
    Similarly to Lemma \ref{lm_intM}, we prove this by contradiction. 
    Assume there is $p \in \partial M$ 
    in the intersection of $K_l$ and $\pi(\Lambda_{ijk}^b)$. 
    If $p \notin \gamma_l $, then by choosing small enough $s_0$ one has $p \notin K_l$. 
    This implies that in $\ntxxi$ we must have $p = \gamma_l(t_l^0)$.
    In this case, one has $\gamma_l(t_l^0) \in \pi(\Lambda_{ijk}^b))$, 
    which is contained in the casual future of $K_{ijk}$ and therefore in the casual futures of $x_i,x_j,x_k$. 
    This contradicts with the assumption (\ref{assp_causalindep}). 
    
    Thus the intersection of the two wavefront sets happens only in $T^*\intM$. 
    We have two possibilities, at $q$ or at some $ p \in  \pi(\Lambda_{ijk}^b))\setminus K_{ijk}$. 
    Near $q$, one can locally regard $w_1$ as a distribution in $\tM$ 
    microlocally supported
    in $\Lambda_{ijk}^g(\epsilon) \cap \Lambda_{ijk}$, which is a small neighborhood of 
    $\Lambda_{ijk} \cap \Char(\square_g)$. 
    By \cite[Theorem 
    8.2.10]{Hoermander2003}, the singularities of $\mathcal{N}_0(\nabla_g v_l, \nabla_g w_1)$ locally near $q$ is in the span of $\Lambda_l$ and $\Lambda_{ijk}^g(\epsilon) \cap \Lambda_{ijk}$.
    Then these singularities will propagate along the broken bicharacteristic arcs after we apply $Q_g$ to $\mathcal{N}_0(\nabla_g v_l, \nabla_g w_1)$, if they are in $\Char(\square_g)$.
    Notice that if a covector 
    \[
    \zeta = \zeta^{(l)} + \zeta^{(ijk)}, \quad \text{with }(q, \zeta^{(l)}) \in \Lambda_l, 
    \ (q, \zeta^{(ijk)}) \in \Lambda_{ijk}^g(\epsilon) \cap \Lambda_{ijk}
    \]  
    is lightlike, 
    then as stated in \cite{Kurylev2018} one has $\zeta^{(ijk)}$ is not lightlike. 
    We can always choose $\ep$ small enough so that such $(q, \zeta^{(ijk)})$ is not contained in 
    the small neighborhood of 
    $\Lambda_{ijk} \cap \Char(\square_g)$. 
    Thus, we must have $\zeta$ is proportional to either $\zeta^{(l)}$ or $\zeta^{(ijk)} \in \Lambda_{ijk} \cap \Char(\square_g)$. 
    It follows that there are no new singularities produced under the broken flow. 
    
    Away from $q$, if $\gamma_l$ and $\pi(\Lambda_{ijk}^b)$ do not intersect, 
    then there are no new singularities produced either. 
    If $\gamma_l$ intersects with $\pi(\Lambda_{ijk}^b))$ at a point $p$, then the intersection must be transversal.
    Near $p$, we microlocally regard $w_1$ as a distribution supported in $\pi(\Lambda_{ijk}^{b}) \cap \intM$. 
    In this case the singularities of $\mathcal{N}_0(\nabla_g v_l, \nabla_g w_1)$ is in the span of $\Lambda_l$ and $\Lambda_{ijk}^{b}$ near $p$. The same argument shows that a light-like covector $\zeta$ as a sum of $\zeta^{(l)}$ in $\Lambda_l$ and $\zeta^{(ijk)}$ in $\Lambda_{ijk}^{b}$ must be parallel to one of them.
    Thus, we have no new singularities produced after applying $Q_g$ to $\mathcal{N}_0(\nabla_g v_l, \nabla_g w_1)$.       
    To conclude, when $\ep$ goes to zero, $\wfset_b(Q_g(\mathcal{N}_0(\nabla_g v_l, \nabla_g w_1))$ tends to a subset of $\pi_b(\Lambda_l\cup \Lambda_{ijk} \cup \Lambda_{ijk}^{b})$. 
\end{proof}


    Third, we analyze the singularities of $Q_g[\mathcal{N}_0(\nabla_g v_l, \nabla_g w_2)]$. 
    Recall 
    \[
    \wfset(w_2) \subset \pi_b(\Lambda^{(1)} \cup (\Lambda^{(1)}(\epsilon) \cap \Lambda_{ijk}) \cup \Lambda_{ijk}^{b}).
    \] 
    The same arguments in Lemma \ref{lm_w1} work for this case.
    To conclude, in $\nxxi$ and $\ntxxi$ we have 
         $
        \wfset_b(\mathcal{N}_0(\nabla_g v_l, \nabla_g w_2)) $ tends to a subset of
       $\pi_b(\Lambda^{(1)} \cup \Lambda_{ijk} \cup \Lambda_{ijk}^{b})$, as $\epsilon$ goes to zero. 
    
    Forth, we analyze the singularities of $Q_g[\mathcal{N}_0(\nabla_g v_l, \nabla_g w_3)]$. 
    In this case, we have
    \[
    \wfset(\nabla_g w_3)\subset \Lambda^{(1)} \cup \Lambda^{(2)}. 
    \]
    We use the same arguments in \cite[Propostion 3.11]{Lassas2018} to conclude 
    $
    \wfset_b(\mathcal{N}_0(\nabla_g v_l, \nabla_g w_3))$ tends to a subset of 
    $\pi_b(\Lambda^{(1)} \cup \Lambda_{ijk} \cup \Lambda_{ijk}^{b})$ as $\ep$ goes to zero.

     We emphasize that all the analysis above happens in the set $\nxxi \cap \ntxxi$.
     To summarize, the distribution $Q_g[\mathcal{N}_0(\nabla_g v_l, \nabla_g (w_1 + w_2 + w_3))]$ does not have new singularities other than those in 
     \[
     \pi_b(\Lambda^{(1)} \cup \Lambda^{(2)} \cup \Lambda^{(3)} \cup \Lambda^{(3),b}). 
     \]
     If we treat it as an element in $\mathcal{D}"(\tM)$, then it follows that $\mathcal{U}^{(4)}_1  = Q_g[w_{0,1}]$ away from $\unionGamma$, and therefore the conclusions of $ Q_g[w_{0,1}]$ in this subsection apply to $\mathcal{U}^{(4)}_1$. 
    
    
    
   
    \subsection{The analysis of $\mathcal{U}^{(4)}_2$}\label{subsec_U42d}
    Recall $\mathcal{U}^{(4)}_2 = Q_g[\sum_{(i,j,k,l)\in \Sigma(4)} h_{ijkl}]$, where we write 
    \[
    h_{ijkl} =\mathcal{N}_0(\nabla_g A_2^{ij}, \nabla_g A_2^{kl}).
    \]
    Away from $\bigcup_{j=1}^4 J^+(\gamma_{x_j, \xi_j}(t_j^b))$, we have 
    \[
    \nabla_g A_2^{ij} \in I^{\mu+1, \mu+2}(\Lambda_{ij}, \Lambda_i) + I^{\mu+1, \mu+2}(\Lambda_{ij}, \Lambda_j).
    \]
    The same analysis as in \cite{Kurylev2018, Lassas2018, Wang2019} applies to $h_{ijkl}$. 
    Indeed, by \cite[Lemma 3.8]{Lassas2018}, 
    one has
    \[
    h_{ijkl
    } = w_0 + w_1 + w_2,
    \]
    where 
    \begin{align*}
    w_0 \in I^{4 \mu + 5}(\Lambda_q) , 
    \quad \wfset(w_1) \subset
    \Lambda^{(1)} \cup \Lambda^{(2)} \cup \Lambda^{(3)},
    \quad \wfset(w_2) \subset \Lambda^{(1)}(\epsilon).
    \end{align*}
    The principal symbol for the leading term $w_0$ is 
    \[
    \sigma^{(p)}(w_0)(q, \zeta) = (2\pi)^{-3}  \mathcal{C}_2(\zeta^{(i)}, \zeta^{(j)}, \zeta^{(k)}, \zeta^{(l)})
    \Pi_{j=1}^4 \sigma^{(p)}(v_j) (q, \zeta^{(j)}),
    \]
    where 
    \begin{align}\label{C2}
    \mathcal{C}_2(\zeta^{(i)}, \zeta^{(j)}, \zeta^{(k)}, \zeta^{(l)})
    &=\mathcal{N}_0(\zeta^{(i)}+ \zeta^{(j)}, \zeta^{(k)}+ \zeta^{(l)})\frac{\mathcal{N}_0(\zk,\zl)}{| \zk + \zl|^2_{g^*}} \frac{\mathcal{N}_0(\zi,\zj)}{| \zi + \zj|^2_{g^*}}.
    \end{align}
    The same arguments in Subsection \ref{subsec_U41d} show that 
    \begin{align*}
    &\wfset_b(Q_g[w_{0}])
    \subset \pi_b(\Lambda_q \cup \Lambda_q^{b}), 
    \quad \wfset_b(Q_g[w_{1}])
    \subset \pi_b(\Lambda^{(1)} \cup \Lambda^{(2)} \cup  \Lambda^{(3)} \cup \Lambda^{(3),b}), 
    \end{align*}
    and $\wfset_b(Q_g[w_{2}]) $ tends to a subset of 
    $\pi_b(\Lambda^{(1)})$
    when $\ep$ goes to zero. 
    To find the principal symbol of the leading term $Q_g[w_{0}]$, we decompose it as
    $ Q_g[w_{0}] = u^{\text{inc}} + u^{\text{ref}}$ as before.
    The same arguments show that 
    if $(y, \eta) \in T^*M|_{\partial M}$ is away from $\unionGamma$ and lies along the forward null-bicharacteristic starting at $(q, \zeta)$, then we have 
    \begin{align}\label{eq_psrub}
    \sigma^{(p)}(\mathcal{R}(Q_g[w_{0}]))&(y_|, \eta_|) 
    = 2(2\pi)^{-2} \sigma^{(p)}(\mathcal{R})(y_|, \eta_|, y, \eta) \times\\
    &\sigma^{(p)}(\tQ_g)(y, \eta, q, \zeta)   
    \mathcal{C}_2(\zeta^{(i)}, \zeta^{(j)}, \zeta^{(k)}, \zeta^{(l)})(\prod_{j=1}^4 \sigma^{(p)}(v_j) (q, \zeta^{(j)})), \nonumber
    \end{align}
    where $(y_|, \eta_|) $ is the projection of $(y,\eta)$. 
    
    \subsection{The analysis of $\mathcal{U}^{(4)}_3$}\label{subsec_U43d}
    Recall $\mathcal{U}^{(4)}_3 = Q_g[\sum_{(i,j,k,l)\in \Sigma(4)} h_{ijkl}]$, where we write 
    \[
    h_{ijkl} =v_i \mathcal{N}_0(\nabla_g v_j, \nabla_g A_2^{kl}).
    \]
    Away from $\bigcup_{j=1}^4 J^+(\gamma_{x_j, \xi_j}(t_j^b))$, the same analysis as in \cite{Kurylev2018, Lassas2018, Wang2019} applies to $h_{ijkl}$. 
    Indeed, by the proof of \cite[Proposition 3.11]{Lassas2018}, see also \cite[Lemma 3.6, Lemma 3.10]{Lassas2018}, 
    one has
    \[
    h_{ijkl} = w_0 + w_1 + w_2,
    \]
    where 
    \begin{align*}
    &w_0 \in I^{4 \mu + 5}(\Lambda_q) , \quad \wfset(w_1) \subset
    \Lambda^{(1)} \cup \Lambda^{(2)} \cup \Lambda^{(3)}, \quad 
    \wfset(w_2) \subset \Lambda^{(1)}(\ep) \cup \Lambda^{(3)}(\ep)
    \end{align*}
    The principal symbol for the leading term $w_0$ is 
    \[
    \sigma^{(p)}(w_0)(q, \zeta) = (2\pi)^{-3}  \mathcal{C}_3(\zeta^{(i)}, \zeta^{(j)}, \zeta^{(k)}, \zeta^{(l)})
    \Pi_{j=1}^4 \sigma^{(p)}(v_j) (q, \zeta^{(j)}),
    \]
    where 
    \begin{align}\label{C3}
    \mathcal{C}_3(\zeta^{(i)}, \zeta^{(j)}, \zeta^{(k)}, \zeta^{(l)}) 
    &=2 \mathcal{N}_0( \zeta^{(j)}, \zeta^{(k)}+ \zeta^{(l)})\frac{\mathcal{N}_0(\zk,\zl)}{| \zk + \zl|^2_{g^*}}.
    \end{align}   
    The same arguments in Subsection \ref{subsec_U41d} show that 
    \begin{align*}
    \wfset_b(Q_g[w_{0}])
    \subset \pi_b(\Lambda_q \cup \Lambda_q^{b}), 
    \quad \wfset_b(Q_g[w_{1}])
    \subset  \pi_b(\Lambda^{(1)} \cup \Lambda^{(2)} \cup  \Lambda^{(3)} \cup \Lambda^{(3),b}),
    \end{align*}
    and $\wfset_b(Q_g[w_{2}])$ tends to a subset of 
    $\pi_b(\Lambda^{(1)} \cup \Lambda^{(3),b})$ 
    as $\ep$ goes to zero, since the broken flow is continuous. 

    Similarly, if $(y, \eta) \in T^*M|_{\partial M}$ is away from $\unionGamma$ and lies along the forward null-bicharacteristic starting at $(q, \zeta)$, then we have 
    \begin{align}\label{eq_psruc}
    \sigma^{(p)}(\mathcal{R}(Q_g[w_{0}]))&(y_|, \eta_|) 
    = 2(2\pi)^{-2} \sigma^{(p)}(\mathcal{R})(y_|, \eta_|, y, \eta) \times\\
    &  \sigma^{(p)}(\tQ_g)(y, \eta, q, \zeta)   
    \mathcal{C}_3(\zeta^{(i)}, \zeta^{(j)}, \zeta^{(k)}, \zeta^{(l)})(\prod_{j=1}^4 \sigma^{(p)}(v_j) (q, \zeta^{(j)})), \nonumber
    \end{align}
    where $(y_|, \eta_|) $ is the projection of $(y,\eta)$. 
    \subsection{Analyze $\mathcal{U}^{(4)}_4$}\label{subsec_U44d}
    Recall $\mathcal{U}^{(4)}_4 = Q_g[\sum_{(i,j,k,l)\in \Sigma(4)} h_{ijkl}]$, where we write 
    \[
    h_{ijkl} =A_2^{ij} \mathcal{N}_0(\nabla_g v_k, \nabla_g v_l).
    \]
    The same analysis as in \cite{Kurylev2018, Lassas2018, Wang2019} applies to $h_{ijkl}$. 
    Indeed, by \cite[Lemma 3.8]{Lassas2018}, 
    one has
    \[
    h_{ijkl
    } = w_0 + w_1 + w_2,
    \]
    where 
    \begin{align*}
        w_0 \in I^{4 \mu + 5}(\Lambda_q) , 
        \quad \wfset(w_1) \subset
        \Lambda^{(1)} \cup \Lambda^{(2)} \cup \Lambda^{(3)},
        \quad \wfset(w_2) \subset \Lambda^{(1)}(\epsilon).
    \end{align*}
    In this case, the principal symbol for the leading term $w_0$ is 
    \[
    \sigma^{(p)}(w_0)(q, \zeta) = (2\pi)^{-3}  \mathcal{C}_4(\zeta^{(i)}, \zeta^{(j)}, \zeta^{(k)}, \zeta^{(l)})
    \Pi_{j=1}^4 \sigma^{(p)}(v_j) (q, \zeta^{(j)}),
    \]
    where 
    \begin{align}\label{C4}
    \mathcal{C}_4(\zeta^{(i)}, \zeta^{(j)}, \zeta^{(k)}, \zeta^{(l)}) 
    &= \frac{\mathcal{N}_0(\zi,\zj)}{| \zi + \zj|^2_{g^*}} \mathcal{N}_1(\zk,\zl).
    \end{align}
    The same arguments in Subsection \ref{subsec_U41d} show that 
    \begin{align*}
    \wfset_b(Q_g[w_{0}])
    \subset \pi_b(\Lambda_q \cup \Lambda_q^{b}), 
    \quad \wfset_b(Q_g[w_{1}])
    \subset  \pi_b(\Lambda^{(1)} \cup \Lambda^{(2)} \cup  \Lambda^{(3)} \cup \Lambda^{(3),b}),
    \end{align*}
    and $\wfset_b(Q_g[w_{2}])$ tends to a subset of 
    $\pi_b(\Lambda^{(1)})$ 
    as $\ep$ goes to zero.
    Similarly, if $(y, \eta) \in T^*M|_{\partial M}$ is away from $\unionGamma$ and lies along the forward null-bicharacteristic starting at $(q, \zeta)$, then we have 
    \begin{align}\label{eq_psrud}
    \sigma^{(p)}(\mathcal{R}(Q_g[w_{0}]))&(y_|, \eta_|) 
    = 2(2\pi)^{-2} \sigma^{(p)}(\mathcal{R})(y_|, \eta_|, y, \eta) \times\\
    &  \sigma^{(p)}(\tQ_g)(y, \eta, q, \zeta)   
    \mathcal{C}_4(\zeta^{(i)}, \zeta^{(j)}, \zeta^{(k)}, \zeta^{(l)})(\prod_{j=1}^4 \sigma^{(p)}(v_j) (q, \zeta^{(j)})), \nonumber
    \end{align}
    where $(y_|, \eta_|) $ is the projection of $(y,\eta)$. 
    \subsection{The analysis of $\mathcal{U}^{(4)}_5$}\label{subsec_U45d}
    Recall $\mathcal{U}^{(4)}_5 = Q_g[\sum_{(i,j,k,l)\in \Sigma(4)} h_{ijkl}]$, where we write 
    \[
    h_{ijkl} = v_i v_j \mathcal{M}(\nabla_g v_k, \nabla_g v_l).
    \]
    The same analysis as in \cite{Kurylev2018, Lassas2018, Wang2019} applies to $h_{ijkl}$. 
    Indeed, by \cite[Lemma 3.8]{Lassas2018}, 
    one has
    \[
    h_{ijkl
    } = w_0 + w_1 + w_2,
    \]
    where 
    \begin{align*}
    w_0 \in I^{4 \mu + 5}(\Lambda_q) , 
    \quad \wfset(w_1) \subset
    \Lambda^{(1)} \cup \Lambda^{(2)} \cup \Lambda^{(3)},
    \quad \wfset(w_2) \subset \Lambda^{(1)}(\epsilon).
    \end{align*}
    The principal symbol for the leading term $w_0$ is 
    \[
    \sigma^{(p)}(w_0)(q, \zeta) = (2\pi)^{-3}  \mathcal{M}(q, \zeta^{(k)}, \zeta^{(l)}) 
    \prod_{j=1}^4 \sigma^{(p)}(v_j) (q, \zeta^{(j)}).
    \]
   The same arguments in Subsection \ref{subsec_U41d} show that 
   \begin{align*}
   \wfset_b(Q_g[w_{0}])
   \subset \pi_b(\Lambda_q \cup \Lambda_q^{b}), 
   \quad \wfset_b(Q_g[w_{1}])
   \subset  \pi_b(\Lambda^{(1)} \cup \Lambda^{(2)} \cup  \Lambda^{(3)} \cup \Lambda^{(3),b}),
   \end{align*}
   and $\wfset_b(Q_g[w_{2}])$ tends to a subset of 
   $\pi_b(\Lambda^{(1)})$ 
   as $\ep$ goes to zero.
    
    Next, 
    if $(y, \eta) \in T^*M|_{\partial M}$ is away from $\unionGamma$ and lies along the forward null-bicharacteristic starting at $(q, \zeta)$, then we have 
    \begin{align}\label{eq_psrue}
    \sigma^{(p)}(\mathcal{R}(Q_g[w_{0}]))&(y_|, \eta_|) 
    = 2(2\pi)^{-2} \sigma^{(p)}(\mathcal{R})(y_|, \eta_|, y, \eta) \times\\
    &  \sigma^{(p)}(\tQ_g)(y, \eta, q, \zeta)   
    \mathcal{M}(q, \zeta^{(k)}, \zeta^{(l)}) (\prod_{j=1}^4 \sigma^{(p)}(v_j) (q, \zeta^{(j)})), \nonumber
    \end{align}
    where $(y_|, \eta_|) $ is the projection of $(y,\eta)$. 
    
    
    \subsection{The analysis of $\epslamvec$} \label{subsec_U46d}
    With $\ufour$ defined in (\ref{def_ufour}), one has
    \[
    \epslam  = \ufour|_{\partial M} =\sum_{j=1}^5 \mathcal{R}(\mathcal{U}^{(4)}_j),
    \]
    where $\mathcal{R}$ is the restriction to $\partial M$. 
    Let $(y, \eta) \in \partial T^*M$ be a covector
    lying along the forward null-bicharacteristic starting at $(q, \zeta) \in \Lambda_q$ and away from $\unionGamma$. 
    Notice that $\zeta \in \Lambda_q$ has the unique decomposition $\zeta = \sum_{j=1}^4 \zj$ with $\zj \in \Lambda(x_j,\xi_j, s_0)$.  
    From the analysis above, we combine (\ref{eq_psrua}, \ref{eq_psrub}, \ref{eq_psruc}, \ref{eq_psrud}, \ref{eq_psrue}) and summarize over different permutations $(i,j,k,l)\in \Sigma(4)$ to have 
            \begin{align}\label{eq_pslambdafour}
            \sigma^{(p)}&(\epslamvec)(y_|, \eta_|) = -
              2 (2 \pi)^{-3}\sigma^{(p)}(\mathcal{R})(y_|, \eta_|, y, \eta) \times  \\
             &\sigma^{(p)}(\widetilde{Q}_g)(y, \eta, q, \zeta) (\mathcal{P} + \mathcal{C})(\zeta^{(1)}, \zeta^{(2)}, \zeta^{(3)}, \zeta^{(4)}) 
             (\prod_{j=1}^4 \sigma^{(p)}(v_j) (q, \zeta^{(j)})) \nonumber,
            \end{align}
            where $\zeta =\sum_{j=1}^4 \zeta^{(j)} $ and we have 
            \begin{align}\label{P}
            \mathcal{P}(\zeta^{(1)}, \zeta^{(2)}, \zeta^{(3)}, \zeta^{(4)})  &
            = \sum_{(i,j,k,l) \in \Sigma(4)} \mathcal{M}(q, \zeta^{(k)}, \zeta^{(l)}) \\
            &= 2 (\mathcal{M}(q, \zeta, \zeta) - \sum_{j=1}^4\mathcal{M}(q, \zeta^{(j)}, \zeta^{(j)})),\nonumber
            \end{align}
            as well as 
            \begin{align*}
            \mathcal{C}(\zeta^{(1)}, \zeta^{(2)}, \zeta^{(3)}, \zeta^{(4)}) 
            = \sum_{(i,j,k,l) \in \Sigma(4)}  
            (\mathcal{C}_1+ \mathcal{C}_2 +\mathcal{C}_3+\mathcal{C}_4) (\zeta^{(i)}, \zeta^{(j)}, \zeta^{(k)}, \zeta^{(l)})
            \end{align*}
            with $\mathcal{C}_j$ defined in (\ref{C1}, \ref{C2}, \ref{C3}, \ref{C4}) for $m = 1,2,3,4$.
    
\end{proof}
As the proof of \cite[propostion 4.3]{Wang2019} shows, the quantity $\mathcal{C}$ vanishes on $\mathcal{X}$ if $\mathcal{N}_0$ and $\mathcal{N}_1$ are null forms, where 
\[
\mathcal{X} = \{(\zeta^{(1)}, \zeta^{(2)}, \zeta^{(3)}, \zeta^{(4)}) \in (L_{q_0}^* M)^4: \ \zeta^{(j)} \text{ are linearly independent; } \sum_{j=1}^4 \zeta^{(j)} \text{ light-like} \}.
\]
In the following, we consider the assumption (\ref{assum_n0n1}) that $\mathcal{N}_0, \mathcal{N}_1$ are null forms and  $\mathcal{M}$ is not null.
Then by \cite[Propositon 4.3]{Wang2019}, we have the following proposition that shows the principal symbol of $\mathcal{U}^{(4)}$ is not vanishing on any open set of $\mathcal{X}$. 
Then we can always make small perturbation to have a nonvanishing principal symbol, see the proof of Lemma \ref{lm36}. This is an analog to 
\cite[Proposition 3.4]{Kurylev2018}.
\begin{pp}\label{pp34}
    If $\mathcal{M}(x,\zeta, \zeta)$ is not null, then the function 
    $\mathcal{P}(\zeta^{(1)}, \zeta^{(2)}, \zeta^{(3)}, \zeta^{(4)})$ defined in (\ref{P}) at the point $q$ is nonvanishing on any open set of $\mathcal{X}$.
\end{pp}
In particular, this proposition implies that if there are two symmetric quadratic form $\mathcal{M}^{(1)},\mathcal{M}^{(2)}$ such that the corresponding $\mathcal{P}^{(1)}$, $\mathcal{P}^{(2)}$ defined by (\ref{P}) are equal, then we must have 
$\mathcal{M}^{(1)} = \mathcal{M}^{(2)}$ up to null forms at $q$. 

Moreover, we formulate the following proposition that describes in which case a covector $(y, \eta)$ with $y \in \partial M$ is in the wave front set of the measurement $\epslamvec$,
if $(y, \eta)$ is not in the set $\unionGamma$. 
This is an analog to \cite[Thoerem 3.3]{Kurylev2018} (see also \cite[Proposition 4.2]{Wang2019}).
\begin{pp}\label{thm33}
    Assume $w(x,u,\xi)$ satisfies (\ref{assum_n0n1}). 
    Let $y \in \pmt$ such that $y \in \nxxi \cap \ntxxi$ defined in (\ref{def_nxxi}, \ref{def_ntxxi}). 
    Let $\eta \in T_y^* \tM$  such that $(y, \eta) \notin \unionGamma$, see (\ref{def_Gamma}). 
    Consider the projection $(\yr, \etar)$ of $(y, \eta)$ onto $T^* (\partial M)$. Then
    \begin{itemize}
        \item[(1)] If $(\capgamma) \cap \nxxi = \emptyset$, then $(\yr, \etar) \notin \wfset(\epslamvec)$. 
        \item[(2)] If $(\capgamma) \cap \nxxi = \{ q \}$ but $y \notin J^+(q)$, 
        then $(\yr, \etar) \notin \wfset(\epslamvec)$. 
        \item[(3)] Suppose $(\capgamma) \cap \nxxi = \{ q \}$ such that $q = \gamma_{x_j, \xi_j}(t_j)$ with $0 < t_j < \rho(x_j, \xi_j)$,
        for $j = 1, 2, 3, 4$. If $  \dot{\gamma}_{x_j, \xi_j}(t_j)$ are linearly independent 
        and $(y, \eta) \in \mathcal{C}^{\text{reg}}_M(q)$,
        then $(\yr, \etar) \in \wfset(\epslamvec)$. 
    \end{itemize}
\end{pp}
\begin{proof}
First from the analysis above, we can relate the singularities of ND map data with that of $\ufour$ by the trace operator  $\mathcal{R}$. 
Indeed, a covector $(y_|, \eta_|) \in T^*\partial M \setminus 0$ belongs to $\epslamvec$ if and only if $(y, \eta)$ is in $\wfset(\ufour)$ over $\tM$.     

For (1), with $s_0$ small enough, we can assume $\cap_{j=1}^4 K_j$ is empty. 
By the assumption (\ref{assp_causalindep}), in $\ntxxi$ the intersection of any $K_i$ and $K_j$ happens in $\intM$, if $s_0$ is sufficiently small. 
Moreover by \cite{Kurylev2018}, one has any $K_i$ and $K_j$ intersect at most once in $\nxxi$. 
Now we consider the rest of cases when $K_j$ may intersect for $j =1,2,3,4$. 
Suppose there are transversal interactions of three waves, for example, $K_{ijk}$ is not empty. 
We repeat the same analysis as in Section \ref{subsec_U41d} - \ref{subsec_U45d},  with additionally assuming $v_l$ to be smooth near $K_{ijk}$. 
It follows that $\ufour$ is reduced to triple interactions and has the boundary wave front set contained in $\unionGamma$. 
We note that $K_l$ may intersect with the flow out from $\Lambda_{ijk}$ under the broken bicharacteristics.
But the analysis in the proof of Lemma \ref{lm_w1} shows that their intersection in $\ntxxi$ should always happen in $\intM$. 
Thus we are in the same situation as in \cite[Thoerem 3.3]{Kurylev2018} and \cite[Proposition 4.2]{Wang2019}. 
If $K_i,K_j,K_k$ intersect at $K_{ijk}$ but not transversally, then by the same arguments as in \cite[Proposition 4.2]{Wang2019} we have $\wfset_b({\ufour}) \subset \pi_b(\Lambda^{(1)} \cup \Lambda^{(2)})$. 
If there are only the interactions of two waves, then $\wfset_b({\ufour})$ is contained in the same set.

For (2), from the analysis above one can see no matter whether $K_1, K_2, K_3,K_4$ intersect at $q$ transversally or not, the boundary wave front set of $\ufour$ is contained in  $\unionGamma \cup \Lambda_q \cup \Lambda_q^b$. Thus, if $y \notin J^+(q)$, then $(y, \eta)$ is not in the wave front set of $\ufour$ in $\tM$. 

For (3), with $\dot{\gamma}_j(t_j)$ linearly independent, we can assume $K_1, K_2, K_3,K_4$ intersect at $q$ transversally with sufficient small $s_0$. By Proposition \ref{pp_ufour}, we have the desired result. 
\end{proof}

\subsection{The proof of Theorem \ref{thm_recoverNM}}
By \cite[Lemma 3.5]{Kurylev2014}, for each $q \in \mathbb{W}$, there exist $(x_j, \xi_j) \in L^+V$, $j=1,2,3,4$, such that the four null geodesics corresponding to $(\vec{x}, \vec{\xi})$ intersect at $q$, with linearly independent tangent vectors and $q \in \nxxi$. 
The proof and construction in \cite[Section 3.2]{Hintz2020} indicate the same statement for three null geodesics. 
Then for each $q$,
first we use the interaction of three waves to recover the symmetric form $\mathcal{N}_1(x, \xi)$ at $q$. 
By Lemma \ref{lm_recoverN1}, one can determine $\mathcal{N}_1(q, \xi)$  up to null forms from the principal symbol of the ND map. 
Next we use the interaction of four waves to recover the symmetric form $\mathcal{M}(q, \xi)$.
From the analysis above, in this case the ND map determines the quantity $\mathcal{P} + \mathcal{C}$ 
at any $(\zeta^{(1)}, \zeta^{(2)}, \zeta^{(3)}, \zeta^{(4)})$ in the set $\mathcal{X}$.
By plugging $\mathcal{N}_1(q, \xi)$ module a null form into $\mathcal{C}$, we can assume we are recovering with the assumption (\ref{assum_n0n1}) that $\mathcal{N}_0, \mathcal{N}_1$ are null. 
By Proposition \ref{pp34}, the quantity $\mathcal{P}$ vanishes on $\mathcal{X}$ if and only if $\mathcal{M}$ is null. 
Thus, one can determine $\mathcal{M}$ up to null forms.

%% file: file/newmetric.tex
\section{Determination of the earliest light observation set.}\label{obs}
Throughout this section, we make the following assumptions. 
We will show these assumptions can be fulfilled by the extension later on in Section \ref{sec_recovery}, if the ND map is known in a proper open set on the boundary, especially in  $(0, T) \times \partial N$. 


\begin{assumption} \label{assump}
	We assume there exists a time-oriented globally hyperbolic 
    Lorentzian manifold $(\tM, \tg)$ such that $M$ embeds isometrically as a submanifold with boundary and 
\begin{itemize}
    \item[(A1)] we are given a timelike path $\hat{\mu}: \ [0, 1] \mapsto \tM \setminus M$ and its open neighborhood $V \subset  \tM \setminus M$, such that we can extend $\hat{\mu}$ a little bit to have $V$ contained in  the open set $I(\hat{\mu}(-\epsilon), \hat{\mu}(1+\epsilon))$, where $\epsilon$ is a small positive number;
    \item[(A2)] we know the extended metric $\tg$ in the set $I(\hat{\mu}(-\epsilon), \hat{\mu}(1+ \epsilon)) \cap \tM \setminus M $ and the ND map in a neighborhood $V_{\partial}$ of the set $I(\hat{\mu}(-\epsilon), \hat{\mu}(1+ \epsilon)) \cap \partial M$,  see Figure \ref{fig_assump}.
\end{itemize}
\end{assumption}
\begin{figure}[h]
    \centering
    \includegraphics[height=0.25\textwidth]{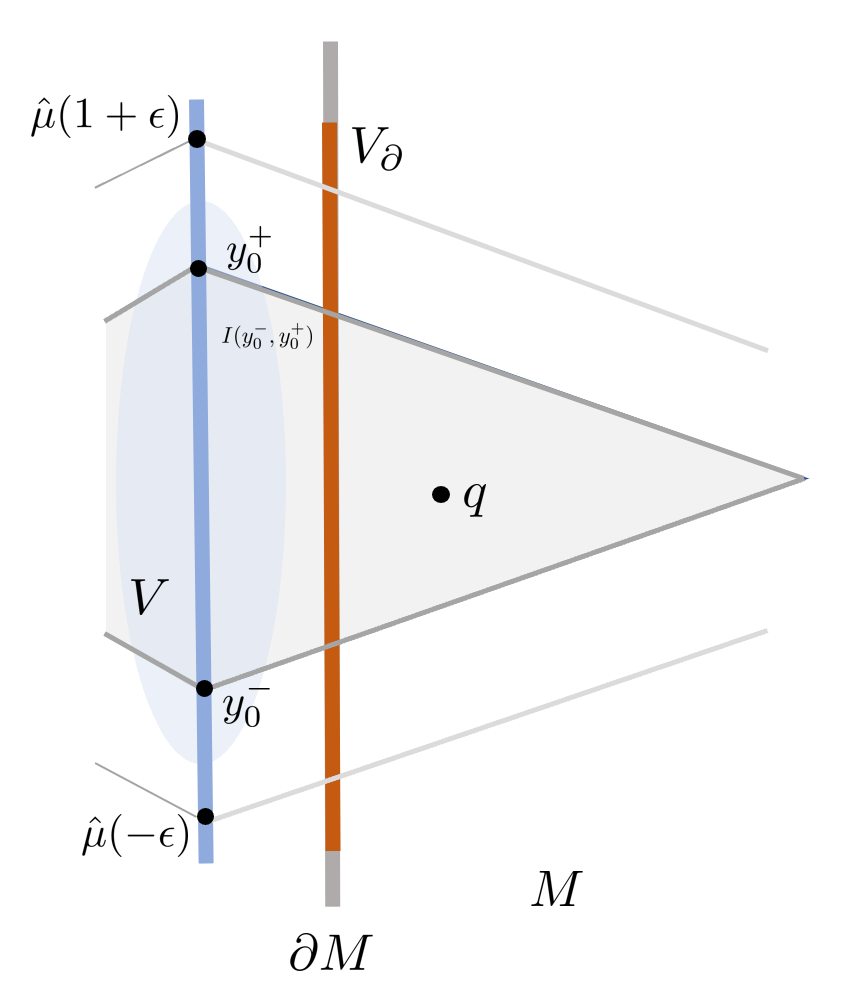}
    \caption{Illustration of $\hat{\mu}$, $V$, $V_\partial$, and $I(\hat{\mu}(-\epsilon), \hat{\mu}(1+ \epsilon))$.}
    \label{fig_assump}
\end{figure}
With these assumptions, 
we are in the similar situation as \cite{Kurylev2018}. 
We would like to recover the conformal class of the metric in $M^\mathrm{o}$ from the earliest observation sets (for the definition see (\ref{def_eob})) in $V$. 
As in \cite{Kurylev2018}, we consider a suitable smaller observation subset of $V$ by 
choosing a family of future pointing smooth time-like paths $\mu_a:[0,1] \rightarrow V$, indexed by $a \in \bar{\mathcal{A}}$, 
where  $\bar{\mathcal{A}}$ is the completion of a connected precompact metric space $\mathcal{A}$, see \cite[Section 2]{Kurylev2018}.
We define this smaller observation set $U \subset V$ as a union of time-like paths
\[
U = \bigcup_{a \in \mathbb{A}}\mu_a((0, 1)). 
\]
Let $\yu^{\pm} = \hat{\mu}(s_\pm)$ with $0 < s_- < s_+ <1$. 
We will show in the following that one can recover the earliest light observation set in $U$ for any points in $I(\yu^-, \yu^+)$  from four pieces of information according to \cite[Theorem 4.5]{Kurylev2018}. 
Instead of directly observing in $U$, we use the information on the boundary given by the ND map to recover this family of earliest light observation sets. 
Then by \cite[Theorem 1.2]{Kurylev2018} the conformal class of the metric in $I(\yu^-, \yu^+)$ can be determined uniquely up to diffeomorphisms. 

\subsection{Detection of singularities on the boundary.}

We follow the same idea as in \cite{Kurylev2018} to detect singularities that are produced by the interaction of four waves in the set $U$. 
The difference is that in their setting the observation happens in $U$ by using the source-to-solution map, while in our case we use the information on the boundary given by the ND map to detect these singularities in $U$. 

For this purpose, we define the following set which connects a covector in $T^*U$ with covectors on the boundary under the null bicharacteristics. 
\begin{df}
    For $(\yu, \etau) \in L^{*,+}U$, let $\mathcal{A}(\yu, \etau)$ be the subset of $L_{V_\partial}^*M $ consisting of the future-pointed lightlike covectors $(y, \eta)$ such that 
    \begin{itemize}
        \item[(1)] $(y, \eta) \in \LcMpo$,
        \item[(2)] there exists $t_0 > 0$ satisfying $\yu = \gamma_{y, \eta^\#}(t_0), \etau = (\dot{\gamma}_{y, \eta^\#}(t_0))^b$.
    \end{itemize}
\end{df}
We note that $\mathcal{A}(\yu, \etau)$ might be empty and might contain more than one element $(y, \eta)$.
\begin{lm}\label{lm_setA}
The set $\mathcal{A}(\yu, \etau)$ is determined by the restricted ND map $\Lambda|_{V_\partial}$ and the extended metric $\tg$ 
in $I(\hat{\mu}(-\epsilon), \hat{\mu}(1+ \epsilon)) \cap \tM \setminus M$. 
\end{lm}
\begin{proof}
We consider the null geodesic passing $\yu$ in the direction of $(\etau)^\#$ in $\tM$.
For convenience, we denote this null geodesic $\gamma_{\yu, (\etau)^\#}(t)$ by $\gamma(t)$ in the following. 
We would like to find all possible intersections of $\gamma(t)$ with the set $V_\partial \subset \partial M$ for $t<0$
by the steps below. 

Step 1, knowing the metric $\tg$ in $I(\hat{\mu}(-\epsilon), \hat{\mu}(1+ \epsilon)) \cap \tM \setminus M $, we can determine the latest time $t_y = \sup\{t<0, \ \gamma \in V_{\partial} \}$ when this null geodesic enters $M$ from $V_\partial$. 
Let $(y, \eta^\#) = (\gamma(t_y), \dot{\gamma}(t_y))$ if such $t_y$ exists. 
Otherwise we do not need to add new element to $\mathcal{A}(\yu, \etau)$. 

Step 2, there are two cases: $\gamma(t)$ hits $\partial M$ transversally or tangentially at $(y, \eta^\#)$. 
In the first case,  it follows that $(y, \eta) \in \LcMpo$, for the definition see (\ref{def_Lbundle}) . 
Then for sufficiently small $\ep>0$, one has $\gamma(t_y - \ep) \in \intM$. 
Let $t_z = \sup \{ t < t_y, \ \gamma(t) \in V_\partial \}$ and  $(z, \zeta^\#) = (\gamma(t_z), \dot{\gamma}(t_z))$ if such $t_z$ exists. 
Otherwise we do not need to add new covectors to $\mathcal{A}(\yu, \etau)$. 
Since $\dot{\gamma}(t_z +\ep) \in \intM$ for $\ep>0$ small enough, by \cite[Propositon 2.4]{Hintz2017} we must have $(z, \zeta)$ in $\LcMng$. 
By Lemma \ref{lm_lens}, the lens relation $(y, \eta) = \mathcal{L}(z, \zeta)$ is locally determined by $\Lambda|_{V_\partial}$ and is invertible. 
Thus, knowing $(y, \zeta) \in \LcMpo$, we can recover $(z, \zeta)$ from inverting the lens relation.
Next, we consider the null geodesic $\gamma_{z, \zeta^\#}(t) = \gamma(t+ t_z)$ with $t <0$ to find the intersections before $(z, \zeta)$. 
By \cite[Propositon 2.3]{Hintz2017}, there exists $\ep>0$ such that $\gamma_{z, \zeta}(t)$ is contained in $\tM \setminus M$ for $-\ep <t <0$. Then we proceed to Step 1 to continue the procedure. 
This finishes the construction for the first case. 
In the second case, suppose $(y,\eta) \in T^* \partial M$. 
By Claim \ref{cl_tancase}, there exists $\ep >0$ such that 
$\gamma(t) \in \tM \setminus M$ for $t_y - \ep< t < t_y$. 
Thus the null geodesic segment $\gamma((t_y - \ep,t_y))$ is determined by $\tg$ in ${I(\hat{\mu}(-\epsilon), \hat{\mu}(1+ \epsilon)) \cap \tM \setminus M}$. 
This time by considering $\gamma(t)$ for $t < t_y - \epsilon$,  one can proceed to Step 1 to continue the procedure.
\end{proof}
\begin{claim}\label{cl_tancase}
Let $(y, \eta) \in T^*\partial M$ be a lightlike covector. 
With $\partial M$ strictly null-convex, there exists $\ep >0$ such that 
$\gamma_{y, \eta^\#}(t) \in \tM \setminus M$, for $t \in (-\ep,0)$. 
\end{claim}
\begin{proof}
First by \cite[Propositon 2.3]{Hintz2017}, we can find $t_0 > 0$ such that $\gamma_{y, \eta^\#}(t) \in \tM \setminus \intM$, for $t \in (-t_0,t_0)$.
Then let $h$ be a smooth boundary defining function in $\tM$ in which case $\partial M$ is given by $h^{-1}(0)$ and we have $h> 0$ for $\intM$ in a small open neighborhood $\mathcal{O}$ of $y$. 
With $\mathcal{O}$ small enough, there exists a positive constant $C$ and the strictly null-convex condition implies that $\kappa(v,v) > C$ for any lightlike tangent vector $v$  on $\partial M$. 
Moreover, we can assume $\gamma_{y, \eta^\#}(t) \in \mathcal{O}$ for $t \in (-t_0,t_0)$. 
Note the outward pointing unit normal to $\partial M$ is $\nu = - \frac{\nabla h}{|\nabla h|}$ and we have
\[
\kappa(v,v) 
=- g(\nabla_v \nu, v)
= - |\nabla h|^{-1}(\nabla^2 h)(v,v).
\]
It follows that $(\nabla^2 h)(v,v) <  - C |\nabla h|$ for in $\mathcal{O}$. 
Now we prove the statement by contradiction. 
Assume there is no such $\ep$. 
Then we can find a sequence $\{ t_j \}$ in $(-t_0, 0)$ satisfying $\lim_{j \rightarrow \infty} t_j = 0$ and  
$\gamma_{y, \eta^\#}(t_j) \in \partial M$. 
We follow the arguments in \cite{Hintz2017} and define the auxiliary function 
\[
f(t) = h \circ \gamma_{y, \eta^\#}(t). %
\]
Note that $f(t) = 0$ exactly when $\gamma_{y, \eta^\#}(t) \in \partial M$ and $f(t) = f'(t) =\diff h(\dot{\gamma}_{y, \eta^\#}(t)) = 0$ exactly when $(\gamma_{y, \eta^\#}(t), \dot{\gamma}_{y, \eta^\#}(t)) \in T\partial M$. 
For each $j$, if $f'(t_j) = \diff h(\dot{\gamma}_{y, \eta^\#}(t_j)) \neq 0$, then there exists sufficiently small $\delta >0$ such that $f(t_j + \delta)$ or $f(t_j - \delta)$ is positive, with $f(t_j) = 0$. 
This implies $\gamma_{y, \eta^\#}(t_j + \delta)$ or $\gamma_{y, \eta^\#}(t_j - \delta)$ is contained in $\intM$, 
which contradicts with the choice of $t_0$. 
It follows that we have $f'(t_j) = 0 $ for any $t_j$ in the sequence. 
Then, for any $j$, there exists $s_j$ with $t_j>s_j > t_{j+1}$ such that $f''(s_j) = 0$. 
Thus, we get a sequence $s_j \rightarrow 0$ with $f''(s_j) = 0$.
On the other hand, we compute 
\[
f''(t) = \nabla^2 h(\dot{\gamma}_{y, \eta^\#}(t), \dot{\gamma}_{y, \eta^\#}(t)).  
\]
Especially, we have $f''(0) <-C |\nabla h(y)|$,
with $|\nabla h(y)|$ nonzero. 
Then the smoothness of $f$ indicates that for sufficiently small $t$ we have $f''(t) < -C |\nabla h(y)|/2$. 
This contradicts with the existence of the sequence $s_j \rightarrow 0$ with $f''(s_j) = 0$.
\end{proof}

\begin{df}
    We say that the \textbf{interaction condition (Ia)} is satisfied for $(\yu, \etau) \in T^*U$ with light-like vectors $(\vec{x}, \vec{\xi})$ and parameters $q, w, t$,  
    if there exists $q \in \nxxi$ in $M$ and $w \in L_q^+ M$, $t \geq 0$ such that $\yu = \gamma_{q, w}(t)$ and $\etau = (\dot{\gamma}_{q, w}(t))^b$, see Figure \ref{fig_setting}.
\end{df}

\begin{df}\label{defDa}
    We say a covector $(\yu, \etau) \in T^*U$ satisfies the \textbf{condition (Da)}  with light-like vectors $(\vec{x}, \vec{\xi})$ and $\hat{s} >0$, 
    if for any $s_0, s_1 \in (0, \hat{s})$ and $j = 1, 2, 3, 4$, 
    there exist 
    \begin{enumerate}
        \item[(a)] $(x_j', \xi_j')$ in the $s_1$-neighborhood of $(x_j, \xi_j )$ with $u_j \in I^{\mu}(\Lambda(x_j',\xi_j',s_0))$ satisfying (\ref{assp_causalindep}-ab), such that $\square_g u_j \in C^\infty$,
        \item[(b)] $(y, \eta)\in \mathcal{A}(\yu, \etau) $,
    \end{enumerate}
    such that
    $(\yr, \etar) \in \wfset(\epslamvec)$, where we define $f = \sum_{j=1}^{4} \epsilon_j f_j$ with $f_j = \partial_\nu u_j|_{\partial M}$ and 
    $(\yr, \etar)$ is the orthogonal projection of $(y, \eta)$ onto $T^* (\partial M)$.
\end{df}
The condition (Da) is the singularity detection condition, by which one can tell if the singularity produced by the interaction of four waves can be detected at the given covector in $T^*U$. 
It is closely related to the interaction condition (Ia) by the following lemma. 
In particular, from the ND map one can determine if the condition (Da) is valid. 

\begin{lm}[{Analog to  \cite[Lemma 3.6]{Kurylev2018}}]\label{lm36}
    Assume that $\yu \in U $ and $\yu \in \nxxi \cap \ntxxi$ with $(\yu, \etau) \notin \unionGamma$. Then
    \begin{itemize}
        \item[(1)] If the geodesics corresponding to $(\vec{x}, \vec{\xi})$ either do not intersect in $\nxxi$ or the geodesics intersect at a point $q \in \nxxi$ with $\yu \notin J^+(q)$,
        then $(\yu, \etau)$ does not satisfy condition (Da).
        \item[(2)] Assume $(\yu, \etau)$ satisfies condition (Ia) with $(\vec{x}, \vec{\xi})$ and the parameters $q, w, t$. Additionally, assume $0 < t < \rho(q, w)$. Then $(\yu, \etau)$ satisfies condition (Da) with $(\vec{x}, \vec{\xi})$ and $\hat{s} >0$ small enough. 
        \item[(3)] 
        Using the ND map $\Lambda$ restricted to the open set $V_\partial$ on the boundary, one can determine whether the condition (Da) is valid for  $(\yu, \etau)$ with $(\vec{x}, \vec{\xi})$ and $\hat{s} >0$.
    \end{itemize}
\end{lm}
\begin{proof}
    We follow the same arguments as in \cite{Kurylev2018} to prove the first two statements.
    For (1),  
    to check if $(\yu, \etau)$ satisfies the condition (Da), assume $s_0, s_1, (x_j', \xi_j')$, and  $f$ as in Definition \ref{defDa}. 
    If $\mathcal{A}(\yu, \etau) = \emptyset$, then $(\yu, \etau)$ does not satisfy condition (Da). Otherwise, suppose there exists $(y, \eta) \in \mathcal{A}(\yu, \etau) $.
    When $s_0, s_1$ are small enough, one has $\yu \in {\mathcal{N}(\vec{x}', \vec{\xi}')}$. Moreover, if the geodesics corresponding to $((x'_j, \xi_j'))_{j=1}^4$ intersect at a point $q' \in \mathcal{N}(\vec{x}', \vec{\xi}')$, then $\yu \notin J^+(q')$ . 
    For any $(y, \eta) \in \mathcal{A}(\yu, \etau)$, it follows that  $y \notin J^+(q')$ is true. 
    By Lemma \ref{connecty} one also has $y \in \nxxi \cap \ntxxi$ and $(y, \eta) \notin \unionGamma$.
    Then Proposition \ref{thm33} implies $(\yr, \etar) \notin \wfset(\epslamvec)$ for each $(y, \eta) \in \mathcal{A}(\yu, \etau)$ . 
    This proves that $(\yu, \etau)$ does not satisfy condition (Da).
    
    For (2), the assumptions that $(\yu, \etau) \in T^*U$ satisfies condition (Ia) with $(\vec{x}, \vec{\xi}), q, w$ and $0 < t_0 < \rho(q, w)$ indicate that $(\yu, \etau) \in  \mathcal{C}^{\text{reg}}_U(q)$. 
    Let $t_y = \inf \{ s > 0; \ \gamma_{q, w}(s) \in \tM \setminus M\}$ be the first time leaving $M$
    and we define $(y , \eta)= (\gamma_{q, w}(t_y), (\dot{\gamma}_{q, w}(t_y))^b)$. 
    Note that $0< t_y < t$ and $(y, \eta) \in  \mathcal{C}^{\text{reg}}_M(q)$. 
    By \cite[Proposition 2.4]{Hintz2017}, with $\gamma_{q,w}((0, t_y)) \subset \intM$ one has $(y, \eta^\#) \in T^+_{\partial M} M$ and therefore $(y, \eta) \in \mathcal{A}(\yu, \etau)$. 
    Indeed, with Assumption \ref{assump}, one has $y \in J(q, \yu) \cap \partial M$, which implies $y \in V_\partial$. 
    Then in the following we show that for small $s_0, s_1$, 
    there exists $(x_j', \xi_j')$ in the $s_1$-neighborhood of $(x_j, \xi_j )$ and $f_j, f$ defined correspondingly
    such that $(\yr, \etar) \in \wfset(\epslamvec)$.
    Indeed, with $b^j = (\dot{\gamma}_{x_j, \xi_j}(t_j))^b$, the covectors $(b^j)_{j=1}^4$ might be linear dependent, but choosing arbitrarily small perturbation one can find light-like covectors $(\hat{b}^j)_{j=1}^4$ which are linear independent. 
    Then let $\zeta = w^b$ and we decompose it as $\zeta = \sum_{j=1}^4 \hat{\zeta}^j$, where $\hat{\zeta}^j = \alpha_j \hat{b}^j$ with some scalars $\alpha_j$.
    It follows that $(\hat{\zeta}^j)_{j=1}^4 \in \mathcal{X}$ and by Proposition \ref{pp34} there exists arbitrarily small perturbation $({\zeta}^j)_{j=1}^4 \in \mathcal{X}$ such that $\mathcal{P}(\zeta^1,\zeta^2,\zeta^3,\zeta^4 ) \neq 0$.
    We define $(x_j', \xi_j')$ by reversing the flow from $(q, \frac{1}{\alpha_j}(\zeta^j)^\#)$, i.e. define
    $x_j' = \gamma_{q,\frac{1}{\alpha_j}(\zeta^j)^\#}(-t_j)$ and $\xi_j' = \dot{\gamma}_{q,\frac{1}{\alpha_j}(\zeta^j)^\#}(-t_j)$. 
    By \cite[Lemma 3.1]{Kurylev2018} one can construct $u_j \in I^{n-\frac{1}{2}}(\Lambda(x_j',\xi_j',s_0))$ such that $\square_g u_j \in C^\infty$ and define $f_j, f$ as in Definition \ref{defDa} . 
    Then by Proposition \ref{thm33} we have $(\yr, \etar) \in \wfset(\epslamvec)$.
    
    For (3), given $(\yu, \etau) \in T^*U$ one can find every $(y, \eta) \in \mathcal{A}(y_0, \eta^0)$ in $V_\partial$ from the ND maps restricted there according to Lemma \ref{lm_setA}, if such $(y, \eta)$ exists.  
    Then by Proposition \ref{thm33} the principal symbol of the ND map determines if the projection $(\yr, \etar) \in \wfset(\epslamvec)$ and therefore determines if $(\yu, \etau)$ satisfies condition (Da). 
\end{proof}
\subsection{Surfaces of earliest singularities.}
In the following, by \cite{Kurylev2018} we define a microlocal version of the \textbf{earliest stable singularities} in $\ntxxi$  as 
\begin{align*}
\mathcal{S}^{\text{m}}(\vec{x}, \vec{\xi}) = &
\{
(\yu, \etau) \in T^*U; \ \text{there is } \hat{s}>0 \text{ such that the condtion (Da)}\\
&\text{ is valid for } (\yu, \etau), (\vec{x}, \vec{\xi}), \text{ and } \hat{s},
\text{ and } \yu \in \ntxxi
\}.
\end{align*}
We will show in Lemma \ref{lm_R} that one can determine if a given covector in $T^*U$ belongs to $\ntxxi$ with Assumption \ref{assump}. 
Then we define the projection $\mathcal{S}(\vec{x}, \vec{\xi})= \pi(\mathcal{S}^{\text{m}}(\vec{x}, \vec{\xi}))$ onto $U$ as the \textbf{earliest stable singularities}. 
Let $\mathcal{S}_H(\vec{x}, \vec{\xi})$ be the set of such points $y_0 \in U$ that for any small neighborhood $W$ of $y_0$ the Hausdorff dimension of the intersection $W \cap \mathcal{S}(\vec{x}, \vec{\xi})$ is at least three. 
Note as a subset of $\ntxxi$, the set $\mathcal{S}_H(\vec{x}, \vec{\xi})$ we define here might not be closed in $U$. 
Next, we denote by $\text{cl}(S)$ the closure of a set $S$ and define
\[
\mathcal{S}_e(\vec{x}, \vec{\xi}) = \bigcup_{a \in \mathcal{A}} \textbf{e}_a(\text{cl}(\mathcal{S}_H(\vec{x}, \vec{\xi})))
\]
as the \textbf{surface of the earliest stable singularities} produced by the interaction of four waves. 
Here the set
\[
\textbf{e}_a(S) = \{\mu_a(\inf\{s \in [-1,1]:\ \mu_a(s) \in S \})\}
\]
is defined as the earliest points of set $S$ on the path $\mu_a = \mu_a([-1,1])$, if $\mu_a \cap S \neq \emptyset$; otherwise $\textbf{e}_a(S)$ is defined as the empty set. 
The following lemma shows the relation of $\mathcal{S}_e(\vec{x}, \vec{\xi})$ and 
$\mathcal{E}_U(q)$ if $q$ is the intersection point away from the casual future of the first cut points. In particular, the last statement implies that $\mathcal{S}_e (\vec{x}, \vec{\xi})$ satisfies the properties (P1) and (P2) in \cite[Theorem 4.5]{Kurylev2018}. 

\begin{lm}[{Analog to \cite[Lemma 4.4]{Kurylev2018}}]\label{lm44}
    Assume $\yu \in U$.
    We have following statements. 
    \begin{itemize}
        \item[(1)] Assume $(\yu, \etau) \in T^*U$ with $\yu \in \nxxi \cap \ntxxi$ satisfies condition (Ia) with $(\vec{x}, \vec{\xi})$ and the parameters $q, w, t$ such that $0 \leq t \leq \rho(q, w)$. They $\yu \in \text{cl} (\mathcal{S}_H (\vec{x}, \vec{\xi}))$. 
        \item[(2)] Assume the geodesics corresponding to $(\vec{x}, \vec{\xi})$ either do not intersect in $\nxxi$ or intersect at $q$ and $\yu \notin J^+(q)$. Then $\yu \notin \mathcal{S}_H (\vec{x}, \vec{\xi})$. 
        \item[(3)] The sets $\mathcal{S}_e (\vec{x}, \vec{\xi})$ satisfies
        \begin{align*}
        & \mathcal{S}_e (\vec{x}, \vec{\xi}) = \EUq \subset \nxxi, \text{ if } (\capgamma) \cap \nxxi = \{ q \},\\
        & \mathcal{S}_e (\vec{x}, \vec{\xi}) \subset M_0 \setminus \nxxi, \text{ if }(\capgamma) \cap \nxxi = \emptyset.
        \end{align*}
    \end{itemize}
\end{lm}
\begin{proof}
    We follow similar arguments to \cite{Kurylev2018} to prove these statements. 
    For (1), by \cite{Kurylev2018} one can see that the Hausdorff dimension of the set $\pi(\unionGamma) \cap \nxxi$ is at most 2. 
    If we assume $(\yu, \etau)$ satisfies conditions in (1) and additionally $(\yu, \etau) \notin \unionGamma $ with $ t < \rho(q, w)$, 
    then by Lemma \ref{lm36} statement (2) we have$(\yu, \etau) $ satisfies the condition (Da).
    Moreover, the covector $(\yu, \etau) \in \mathcal{C}_U^\text{reg}(q)$, which implies $\yu \in \EUqr$ and has a neighborhood $W$ such that $\EUqr \cap W$ is a 3-dimensional submanifold. It follows $\yu \in \mathcal{S}_H({\vec{x}, \vec{\xi}})$, see \cite[Section 2]{Kurylev2018}. 
    The general case satisfying the assumptions in (1) without the extra conditions can be obtained by the same density arguments as in \cite{Kurylev2018}, since $\text{cl} (\mathcal{S}_H (\vec{x}, \vec{\xi}))$ is closed. 
    
    For (2),
    in the case when the geodesics corresponding to $(\vec{x}, \vec{\xi})$ do intersect in $\nxxi$, we denote $\mathcal{N}_1 = \nxxi \setminus J^+(q)$. 
    In the case when the geodesics corresponding to $(\vec{x}, \vec{\xi})$ do not intersect in $\nxxi$, we denote $\mathcal{N}_1 = \nxxi$.
    By Lemma \ref{lm36} statement (1) the condition (Da) is not satisfied for any points in $\mathcal{N}_1 \setminus \pi(\unionGamma)$. Recall that the Hausdorff dimension of the set $\pi(\unionGamma) \cap \nxxi$ is at most 2, and therefore $\mathcal{N}_1 \cap  \mathcal{S}_H({\vec{x}, \vec{\xi}}) = \emptyset$.
    
    For (3), suppose $(\capgamma) \cap \nxxi = \{ q \}$. 
    First, by Lemma \ref{lm_ntxxi} one has $\EUq \setminus \uniongamma \subset \ntxxi$. 
    Second, we can show that $\EUq \setminus \uniongamma $ is contained in $\nxxi$, by the exactly same arguement as that in \cite[Lemma 4.4]{Kurylev2018}. 
    Thus, by (1) one obtains 
    $\EUq \setminus \uniongamma \subset \text{cl}(\mathcal{S}_H (\vec{x}, \vec{\xi}))$. 
    The density argument implies that
    $\EUq \subset  \text{cl}(\mathcal{S}_H (\vec{x}, \vec{\xi}))$. 
    One the other hand, by (2) one has $\mathcal{S}_H (\vec{x}, \vec{\xi}) \subset J^+(q)$ and so does its closure, which implies $\mathcal{S}_e (\vec{x}, \vec{\xi}) \subset \EUq$ by \cite[Lemma 2.4]{Kurylev2018}. 
    This proves that the two sets are equal. 
    In the case that $(\capgamma) \cap \nxxi = \emptyset$, by (2) the set $\mathcal{S}_H (\vec{x}, \vec{\xi})$ is empty in $\nxxi$. 
    It follows that $\mathcal{S}_e (\vec{x}, \vec{\xi}) \subset M \setminus \nxxi$ as desired. 
\end{proof}
To apply \cite[Theorem 4.5]{Kurylev2018}, we need the following lemmas. 
\begin{lm}\label{lm_R}
    The statements (A1-A2) in Assumption \ref{assump} determine 
    if a point $\yu \in U$ belongs to the set $\ntxxi$ or not. 
\end{lm}
\begin{proof}
    By (\ref{def_ntxxi}), we need to check if $\yu$ is in the casual future of 
    $J^+(\gamma_{x_j, \xi_j}(t_j^b))$, where $t_j^b$ is defined in (\ref{def_bpep}), for $j = 1, 2, 3, 4$. 
    We claim that either $V \cap J^+(\gamma_{x_j, \xi_j}(t_j^b)) = \emptyset$ or $\gamma_{x_j, \xi_j}(t_j^b) \in V_\partial$ for each $j$. 
    Indeed, if this is not true, then we can find $j$ such that there is a causal path from $\gamma_{x_j, \xi_j}(t_j^b)$ to $\yu$ and $\gamma_{x_j, \xi_j}(t_j^b) \notin V_\partial$. 
    This implies $\gamma_{x_j, \xi_j}(t_j^b) \in J(x_j, \yu)$. 
    Note that $x_j, y_0 \in V \subset I(\hat{\mu}(-\ep), \hat{\mu}(1+\ep))$ and therefore 
    $\gamma_{x_j, \xi_j}(t_j^b) \in I(\hat{\mu}(-\ep), \hat{\mu}(1+\ep))) \cap \partial M \subset V_\partial$, by Assumption \ref{assump}. This proves the claim.
    
    Now we would like to determine if $\yu \in U$ is in $J^+(\gamma_{x_j, \xi_j}(t_j^b))$. 
    We consider $(z, \zeta) =(\gamma_{x_j, \xi_j}(t_j^0), (\gamma_{x_j, \xi_j}(t_j^0))^b) \in \LcMng$ and use the lens relation $ \mathcal{L}$ in $V_\partial$, which is determined by the restricted ND map $\Lambda|_{V_\partial}$ by Lemma \ref{lm_lens}. 
    If there is no lightlike covector $(p, \eta) \in \LcMpo$ satisfying $(p,\eta) = \mathcal{L}(z, \zeta)$, 
    then we conclude that $V \cap J^+(\gamma_{x_j, \xi_j}(t_j^b)) = \emptyset$ and therefore any $\yu \in V$ is not in $J^+(\gamma_{x_j, \xi_j}(t_j^b))$. 
    If there exists lightlike covector $(p,\eta) = \mathcal{L}(z, \zeta)$ in $V_\partial$, then we have the case $p = \gamma_{x_j, \xi_j}(t_j^b) \in V_\partial$. 
    Using the extended metric and $\Lambda|_{V_\partial}$, we can find the set $\mathcal{L}^+(p) \cap V$ and then $J^+(p) \cap U$. 
    This tells us if $\yu \in U$ is in the causal future of $\gamma_{x_j, \xi_j}(t_j^b)$ or not.   
\end{proof}
\begin{lm}[{Analog to \cite[Lemma 4.3]{Kurylev2018}}]\label{lm43}
    The statements (A1-A2) in Assumption \ref{assump} uniquely determine 
    \begin{itemize}
        \item[(1)] the set $\mathcal{S}_e(\vec{x},\vec{\xi})$ for all $(\vec{x},\vec{\xi}) \in (L^+U)^4$,  
        \item[(2)] the set  $G_{\hat{\mu}} = \{(x,\xi) \in L^+ U; \ x \in U, \ \gamma_{x, \xi}(\mathbf{R}_+) \cap \hat{\mu} \neq \emptyset\}$,
        \item[(3)] the causality relation $R^<_U = \{(p_1, p_2) \in U \times U: p_1 < p_2\}$.
    \end{itemize}
\end{lm}
\begin{proof}
    The first statement follows from Lemma \ref{lm36}, \ref{lm44}, \ref{lm_R}. 
    We first prove (2) and then the third statement can be proved by the same arguments as in \cite{Kurylev2018}.  
     
    To show (2), notice if $\gamma_{x, \xi}(\mathbf{R})_+) \cap M = \emptyset$, then the information whether $\gamma_{x, \xi}(\mathbf{R}_+) \cap \hat{\mu} \neq \emptyset$ is determined by the extended metric $\tilde{g}$ in 
    $I(\hat{\mu}(-\epsilon), \hat{\mu}(1+ \epsilon)) \cap \tM \setminus M $, with the assumption (A2).
    In the case that $\gamma_{x, \xi}(\mathbf{R}_+) \cap M \neq \emptyset$,
    one can follow the same arguments in Lemma \ref{lm_R}. 
    Note that if a null geodesic $\gamma_{x, \xi}(\mathbf{R}_+)$ enters $M$ and then returns to $U$, then it intersects $\partial M$ exactly in $V_\partial$. 
    Using the lens relation $\mathcal{L}$ on $V_\partial$ determined by the restricted ND map $\Lambda|_{V_\partial}$, one can tell if it intersect $\hat{\mu}$. 
    
    To show (3), we note that for any $p \in U$, we can find $\mathcal{L}^+(p) \cap U$ by (2). 
    This determines the set $J^+(p) \cap U$ and therefore determines the causality relation.     
\end{proof}
Combining Lemma \ref{lm44} and \ref{lm43} with \cite[Theorem 4.5, Theorem 1.2]{Kurylev2018}, we have the following proposition.
\begin{pp}\label{pp_recoverEOS}
    Assume everything in Assumption \ref{assump}. Let $\yu^{\pm} = \hat{\mu}(s_\pm)$ with $0 < s_- < s_+ <1$.
    Then these data determine uniquely the family $\{\mathcal{E}_U(q); \ q \in I(\yu^-, \yu^+)\}$
    of earliest light observation sets, and therefore they determine the conformal class of the metric $g$ at any point in $I(\yu^-, \yu^+)$ up to diffeomorphisms.
\end{pp}

%% file: file/ND.tex
\subsection{Determining the metric on the boundary}\label{subsecbm}
In this part, we determine the jet of the metric on the subset ${V_\partial}$ of the boundary from the first order linearization $\epslamone$.
Indeed, by the asymptotic expansion in $(\ref{expand_u})$, we have $\epslamone = v_1|_{\partial M}$, where $v_1$ solves the boundary value problem of wave equation (\ref{eq_v1}) with Neumann boundary data $f_1$.  
This implies $\partial_{\epsilon_1} \Lambda(\ep_1 \bullet) |_{\epsilon_1=0}$ is the ND map for the wave equation.

Fix some $(y_|, \eta_|) \in T^*(\partial M)$ with $y_| \in V_\partial$, 
there exists a unique $(y, \eta) \in \LcMpo$
such that $(y_|, \eta_|)$ is the orthogonal projection of $(y, \eta)$ to $\partial M$. 
In the following, we use the semigeodesic coordinates $(\bx, x^3)$ near $y \in \partial M$, for more details see \cite[Lemma 2.3]{Stefanov2018}.
The dual variable is denoted by $(\bxi, \xi_3)$ and in this coordinate system the metric tensor $g$ takes the form
\[
g(x) = g_{\alpha \beta} (x)\diff x^\alpha \otimes \diff x^\beta + \diff x^3 \otimes \diff x^3, \quad \alpha, \beta \leq 2. 
\]
We focus on a small conic neighborhood $\Gamma_\partial$ of $(y_|, \eta_|)$.
Let $\chi(x_|, \xi_|)$ be a cutoff $\Psi$DO of order zero microlocally supported in  $\Gamma_\partial$, which equals to $1$ near $(y_|, \eta_|)$. 
Since there are no periodic null geodesics, the microlocally restricted map $\chi \partial_{\epsilon_1} \Lambda(\ep_1 \chi \bullet) |_{\epsilon_1=0}$ can be regarded as the local version of the ND map, which 
maps $f_1 \in \mathcal{E}'(\partial M)$ to 
$v|_{\partial M}$ restricted near $(y_|, \eta_|)$, 
with $\wfset(f_1) \subset \Gamma_\partial$ and 
\[
\square_g v \in C^\infty(M) \text{ near } y, \quad \partial_\nu v|_{\partial M} = f_1 \mod C^\infty(M).
\]

Now suppose there are two metric $g_j$ and the corresponding restricted ND map $\Lambda_{g_j}|_{V_\partial}$ for the semilinear problem (\ref{ihmNBC}), for $j = 1,2$. 
Here we use the notation $\Lambda_{g_j}$ to emphasize the ND map associated with the metric $g_j$.  
By \cite{Stefanov2018}, one can choose a diffeomorphism $\Psi$ that fixes $\partial M$ such that $g_1$ and $\Psi^*g_2$ have common semigeodesic normal coordinates near $\partial M$. 
Since $\Lambda_{g_2} = \Lambda_{\Psi^*g_2}$, we can replace $g_2$ by $\Psi^*g_2$ and assume $g_1, g_2$ 
share the same semigeodesic coordinates $(\bx, x^3)$ near $y \in \partial M$. 
\begin{lm}
Suppose  $\Lambda_{g_1}|_{V_\partial} = \Lambda_{g_2}|_{V_\partial}$.   
In the fixed common semigeodesic coordinates $(\bx, x^3)$, we have
one has 
\[
g_1|_{V_\partial} = g_2|_{V_\partial}, \quad \partial^j_{3}g_1|_{V_\partial} = \partial^j_{3} g_2|_{V_\partial}, \ j = 1,2, \ldots.
\]
\end{lm}
\begin{proof}
Let $(y_|, \eta_|) \in T^*(\partial M)$ with $y_| \in V_\partial$ and we focus on the small conic neighborhood $\Gamma_\partial$ defined as above. 
Recall $\chi(x_|, \xi_|)$ is the cutoff $\Psi$DO of order zero supported in  $\Gamma_\partial$.

Consider the local version of the DN map $\Upsilon^{\text{loc}}$ for the linearized problem defined in \cite{Stefanov2018}.
It is shown to be a $\Psi$DO with a nonzero principal symbol
in $\Gamma_\partial$
and therefore has a parametrix $(\Upsilon^{\text{loc}})^{-1}$. 
One can see that $(\Upsilon^{\text{loc}})^{-1}$\
equal to the local ND map up to smoothing operators. 
Thus, we can reconstruct the full symbol $\sigma(\Upsilon^{\text{loc}})$ by the construction of parametrices for elliptic $\Psi$DOs from the local ND map, with an error in $S^{-\infty}$. 
If we denote by $\Upsilon_j^{\text{loc}}$ the local ND map with the metric $g_j$ for $j = 1,2$, then one has
\[
\sigma(\Upsilon_1^{\text{loc}}) = \sigma(\Upsilon_2^{\text{loc}})\quad \mod S^{-\infty}.
\]
It suffices to show the metric and its normal derivatives on $\partial M$ near $(y_|, \eta_|)$ are determined by the symbol $\sigma(\Upsilon^{\text{loc}})$ up to $S^{-\infty}$. 

We follow the proof of \cite[Theorem 3.2]{Stefanov2018} and choose a special designed function
\[
h(x_|) =e^{i\lambda x_| \cdot \xi_|} \chi(x_|, \xi_|)
\]
with large parameter $\lambda$,
where $\chi$ is a smooth cutoff function supported near $(y_|, \eta_|)$ and homogeneous in $\xi_|$ of order $0$. 
For $j =1,2$, we construct a sequence of geometric optics approximations of the local outgoing solutions near $(y_|, \eta_|)$ of the form 
\[
u_{N,j} (x) = e^{i \lambda \phi_j(x, \xi_|)} \sum_{k=0}^{N} \frac{1}{\lambda^k}a_{j,k}(x, \xi_|),
\]
where $\phi(x, \xi_|)$ is the phase function 
satisfying the eikonal equation and initial condition 
\begin{align*}
\partial_3\phi_j(x) = \sqrt{-g_j^{\alpha\beta}(x) \partial_\alpha \phi_j(x) \partial_\beta \phi_j (x)}, \text{ for } \alpha, \beta \leq 2, \quad 
\phi_j(x_|,0) = x_|\cdot \xi_|,
\end{align*}
and $a_j(x, \xi_|) $ is the amplitude with the asymptotic expansion $a_j = \sum_{k \leq 0} a_{j,k}$.
Near $y_|$, we can assume each $a_{j,k}(x, \xi_|)$ is homogeneous in $\xi_|$ of order $-k$ and satisfies the transport equation with the initial condition 
\begin{align*}
X_j a_{j,0} &= 0, \quad a_{j,0} (x_|, 0, \xi_|) = \chi,\\
X_j a_{j,k} &= r_k, \quad a_{j,k}(x_|, 0, \xi_|) = 0, \text{ for } k <0. 
\end{align*}
Here $X = (2 g_j^{\alpha \beta} \partial_\alpha \phi_j)\partial_\beta + \square_{g_j} \phi_j $ and $r_k$ is the term involving the derivatives w.r.t. $a_{j,0}, a_{j,1}, \ldots, a_{j,k-1}$ and $\phi_j$ of order no more than $k$.
In boundary normal coordinates, the local DN map is given by 
\[
\Upsilon_j^{\text{loc}}(h)(x_|) = - e^{i\lambda x_| \cdot \xi_|} (i \lambda \partial_3 \phi(x_|,0,\xi_|) + \sum_{k =0}^N \frac{1}{\lambda^k}a_{j,k}(x_|,0, \xi_|)) + O(\lambda^{-N-1}).
)
\]
On the other hand, by \cite[Theorem 3.1]{Stefanov2018} 
near $(y_|, \eta_|)$ we can show $\Upsilon_j^{\text{loc}}$ has the symbol 
\[
\sigma(\Upsilon_j^{\text{loc}})(y_|, \eta_|) = -i \sqrt{-g_j^{\alpha\beta}(y_|, 0)\eta_\alpha \eta_\beta } - \partial_n a_j(y_|,0,\eta_|),
\]
with the phase functions and amplitudes above. 
The symbol of $\Upsilon_j^{\text{loc}}$ coincide for $j=1,2$ asymptotically up to $S^{-\infty}$. 
Note the metric and its normal derivatives restricted on the boundary near $y$ are determined by analyzing 
this asymptotic expansion.
In particular, from the principal symbol 
we have 
\[
\partial_3\phi_1(x_|,0,\xi_|) = \partial_3\phi_2(x_|,0,\xi_|), \quad g_1^{\alpha\beta}(y_|, 0)\eta_\alpha \eta_\beta = g_2^{\alpha\beta}(y_|, 0)\eta_\alpha \eta_\beta,\ \alpha, \beta \leq 2. 
\]
As is stated in \cite{Stefanov2018}, by choosing $n(n-1)/2$ vectors in $\Gamma_\partial$ one can uniquely determine the boundary metric $g_j|_{\partial M}$ at $(y_|, \eta_|)$ from this quadratic form and therefore  $g_1|_{\partial M} =  g_2|_{\partial M}$ at  $(y_|, \eta_|)$. 
The normal derivative $\partial^k_3 g_{\alpha\beta}(y_|, 0)$ can be determined by an inductive procedure as in \cite{Stefanov2018}, for more details see also \cite{Hintz2021} and \cite{Sylvester1987}.
One can see this determination is local and only depends on the symbol of the local DN amp module $S^{-\infty}$ and therefore we have the desired result. 
%
\end{proof}

%% file: file/lens.tex
\subsection{Recovering the lens relation on the boundary}\label{subseclens}
Let $(z, \zeta) \in \LcMng$ be an inward pointing lightlike covectors on the boundary, see  (\ref{def_Lbundle}) for the definition of $\LcMng$.
The projection $(z_|, \zeta_|)$ of $(z, \zeta)$ on $\partial M$ is timelike.
Suppose the null geodesic $\gamma(t)$ starts from $(z, \zeta^\#)$ on the boundary and 
hits $\partial M$ transversally at the covector $(y, \eta) \in \LcMpo$. 
Note that the projection $(y_|, \eta_|)$ of $(y, \eta)$ on the boundary is also timelike. 
By \cite[Section 4]{Stefanov2018} the lens relation $\mathcal{L}$ mapping from a small conic neighborhood of 
$(z_|, \zeta_|)$ to a small conic neighborhood of  $(y_|, \eta_|)$ can be recovered from the Dirichlet-to-Neumann map of the linearized equation $\square_g v = 0$. 
In the lemma below, we follow the similar ideas to recover the lens relation for lightlike covectors in $T^\pm_{\partial M}M$ from the ND map of the boundary value problem (\ref{ihmNBC}).

\begin{lm}\label{lm_lens} 
Let $V_{\partial}$ be an open subset of $\partial M$ with $z, y \in V_{\partial}$. 
Suppose $(z, \zeta) \in \LcMng$ and $(y, \eta) \in \LcMpo$.
Then $(y_|, \eta_|) = \mathcal{L}(z_|, \zeta_|)$ is uniquely determined by the restricted ND map $\Lambda|_{V_\partial}$. 
For convenience, we abuse the notation to write $(y, \eta) = \mathcal{L}(z, \zeta)$ if there is no confusion. 
\end{lm}
\begin{proof}
    One can follow the same arguments as in \cite{Stefanov2018} or we follow the ideas in \cite{Kurylev2018}. 
    For the light-like covector $(z, \zeta)$, 
    we construct a distribution $f_1 \in \mathcal{E}'(\partial M)$ of which the wave front set is exactly the half line 
    $\{(z_|,sz_|) \in T^*(\partial M); \ s > 0\}$. 
    Then $\partial_{\epsilon_1} \Lambda(\epsilon_1 f_1)|_{\epsilon_1 = 0}$ is the solution to the wave equation with Neumann boundary data $f_1$, i.e. $\partial_{\epsilon_1} \Lambda(\epsilon_1 f_1)|_{\epsilon_1 = 0} = v_1$ which solves (\ref{eq_v1}). 
    We write $v_1$ as the sum of the incident wave $v_1^{\text{inc}}$ and the reflected wave $v_1^{\text{ref}}$ as before. 
    The incident part $v_1^{\text{inc}}$ has singularities along the bicharacteristic of $\square_g$, which is locally given by $(\gamma_{x, \xi}(t), (\dot{\gamma}_{z, \zeta}(t))^\#)$ before $\gamma_{z, \zeta}(t)$ hits the boundary transversally at $(y, \eta^\#)$ again. 
    The reflected part $v_1^{\text{ref}}$ vanishes before the intersection and satisfies 
    $\square_g v_1^{\text{ref}}= 0$ with $\partial_\nu  (v_1^{\text{inc}} +   v_1^{ \text{ref}})|_{\partial M} = f_1$ near $(y, \eta)$. 
    Actually in a small conic neighborhood of $(y,\eta)$, we have $\partial_\nu ( v_1^{\text{inc}} +   v_1^{ \text{ref}})|_{\partial M}  = 0 \mod C^\infty$ according to the wave front set of $f_1$. 
    Following the analysis in \cite{Stefanov2018} and \cite{Hintz2020}, 
    we have the principal symbol of the restriction of $v_1$ at $(y_|, \eta_|)$ is 
    \[
    \sigma^{(p)}(\mathcal{R}(v_1))(y_|, \eta_|) = 2\sigma^{(p)}(\mathcal{R})(y_|, \eta_|, y, \eta) \sigma^{(p)}(v_1^{\text{inc}})(y, \eta),
    \]
    where $\mathcal{R}$ is trace operator on the boundary. The ND map determines \[
    \sigma^{(p)} (\partial_{\epsilon_1} \Lambda(\epsilon_1 f_1))(y_|, \eta_|) = \sigma^{(p)}(\mathcal{R}(v_1))(y_|, \eta_|)
    \] and therefore determines $(y,\eta)$ since $\mathcal{R}(y_|, \eta_|, y, \eta)$ is nonzero for light-like covectors. 
    \end{proof}

%% file: file/newrecovery.tex
\section{The proof of Theorem \ref{thm_recoverg}}\label{sec_recovery}
In this section, let $g_1$, $g_2$ be two metrics on a smooth manifold $M$ such that $(M, g_1)$ and $(M, g_2)$ are globally hyperbolic Lorentzian manifolds with timelike and strictly null-convex boundaries.
For $j = 1, 2$, consider the Neumann-Dirichlet map $\Lambda_j$ associated with the boundary value problem (\ref{ihmNBC}) in $(M, g_j)$. 
Let $\hmupj: \ [0,1] \rightarrow (0,T) \times \partial N$ be two timelike paths 
with $y^{\pm,(j)} = \hmupj(s_\pm)$ on the boundary,  where $ 0 < s_- < s_+ <1$, see Figure \ref{fig_phi}.
Let $O_j$ be a connected open set in $\partial N$ and $V_{\partial,j} = (0, T) \times O_j$
be small neighborhoods of $\hmupj((0,1))$ satisfying $I(\hmupj(0), \hmupj(1)) \subset V_{\partial, j}$. 
\begin{figure}[h]
    \centering
    \includegraphics[height=0.2\textwidth]{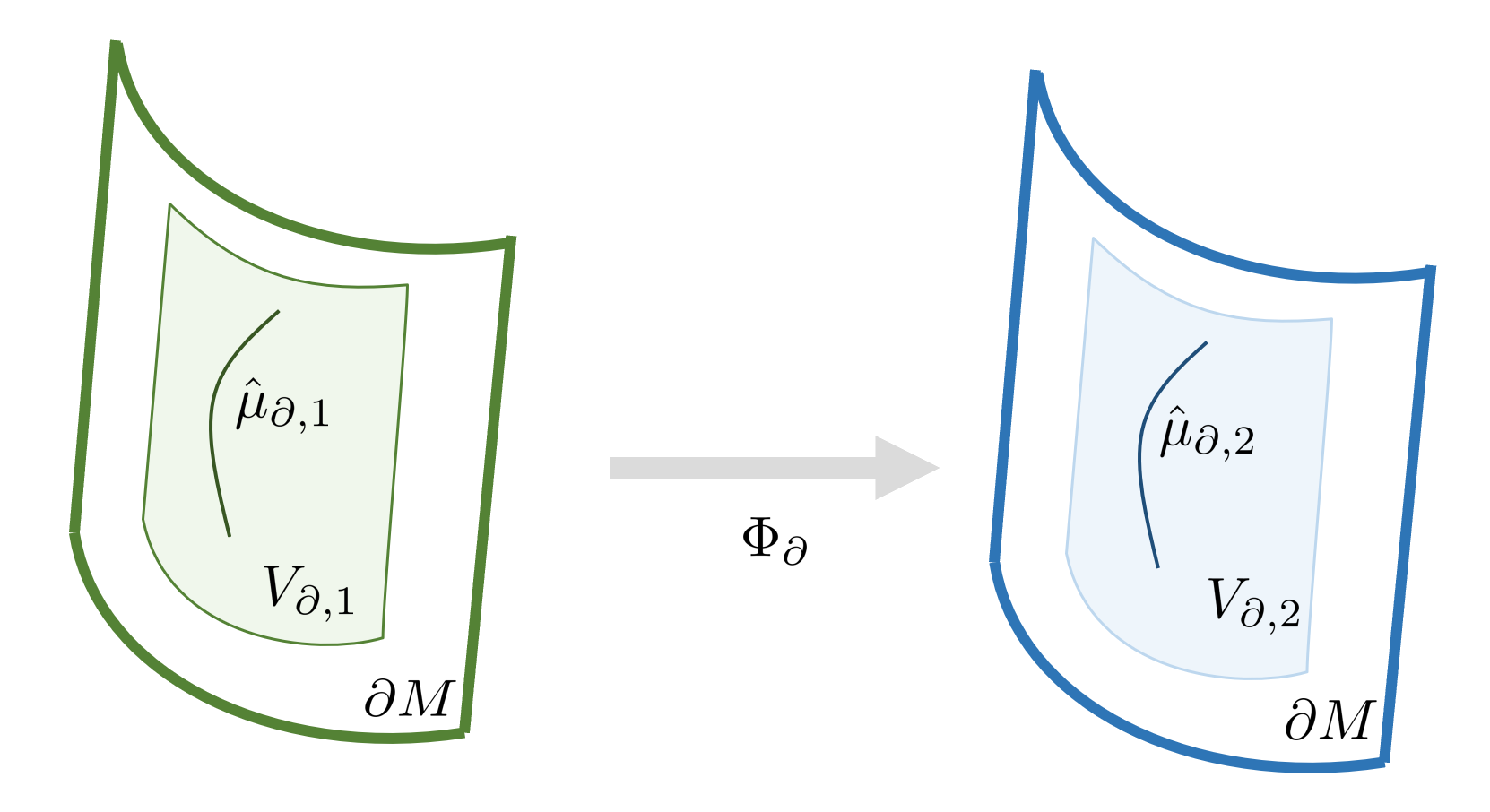}
    \caption{Illustration of the timelike paths $\hmupj$ and the diffeomorphism $\Phi_\partial$. }
    \label{fig_phi}
\end{figure}
Suppose there exists 
a diffeomorphism $\Phi_\partial: V_{\partial, 1} \rightarrow V_{\partial, 2}$ such that 
$\Phi_\partial(y^{\pm,(1)}) = y^{\pm,(2)}$  
as well as
\begin{align}  \label{difflam}
(\Phi_\partial^{-1})^* \circ \Lambda_1|_{V_{\partial, 1}} \circ \Phi_\partial^* = \Lambda_2|_{V_{\partial, 2}}.
\end{align}
By Section \ref{subsecbm}, we can recover the boundary metric $g_j|_{V_{\partial, j}}$ from $\Lambda_j|_{V_{\partial, j}}$. 
Using the Egorov's Theorem one can find from (\ref{difflam}) that the principal symbols of the local ND map for the wave equation are related and therefore
$
g_1|_{V_{\partial, 1}} = \Phi_\partial^* \ g_2|_{V_{\partial, 2}}.
$
In the following, we would like to show the statements (A1-2) in Assumption \ref{assump} are satisfied for $(M, g_j)$, if we ae given $\Lambda_j|_{V_{\partial, j}}$ and (\ref{difflam}), for $j = 1,2$.

%

\subsection{Extension}
In the subsection we extend $(M,g_j)$ to a time-oriented Lorentzian manifold $(\widetilde{M}_j, \widetilde{g}_j)$, for $j =1,2$.
%
%
 First, we choose the semigeodesic normal coordinates $(t, x'', x^n)$ for $g_1$ in a small neighborhood of a fixed point $y \in \partial M$.
 The existence of such coordinates can be found in \cite[Lemma 2.3]{Stefanov2018}.
 Then we extend $g_1$ slightly beyond $\partial M$ near $y$ by the following lemma. 
%

	\begin{lm}
	Suppose $(M,g_1)$ is a smooth globally hyperbolic Lorentzian manifold with boundary as above. 
	For fixed  $y \in \partial M$, 
	let $S$ be 
	a small open set of $\partial M$ where locally near $y$ we identify $M$ with $S \times [0,1)$. 
	Then there exists a positive number $\epsilon$ 
	such that we can extend $g_1$ to a smooth Lorentzian metric $\tilde{g}_1$ in $S \times (-\epsilon, 1)$ 
	with 
	\begin{itemize}
		\item[(1)] $\tg_1|_{S \times [0, \epsilon)} = g_1$,
		\item[(2)] $\tg_1$  takes the product form. 
	\end{itemize} 
	\end{lm}
	\begin{proof}
	In the semigeodesic coordinate system the metric $g_1$ on the boundary has the form 
	\[
	g_1|_S = -\beta(t, x'', 0 ) \diff t \otimes \diff t + \kappa(t, x'', 0) \diff x^\alpha \otimes \diff x^\beta, \text{ where } 0\leq \alpha, \beta \leq 2.
	\]
	We extend $g_1$ slightly beyond $\partial M$ near $y$
	by extending the smooth function  $\beta,\kappa$ to $\widetilde{\beta},\widetilde{\kappa}$ for $x^n < 0$ in some specific way, and therefore we have
	\[
	\tilde{g}_1 = -\widetilde{\beta}(t, x'', x^n) \diff t \otimes \diff t + \widetilde{\kappa}(t, x'', x^n) \diff x^\alpha \otimes \diff x^\beta + \diff x^n \otimes \diff x^n.
	\]
	This smooth extension depends on the values and the derivatives of $\beta, \kappa$ restricted to $S$.
	Since $\beta(t, x'', 0 ) >0$, there exists $\epsilon$ such that the extended function $\widetilde{\beta}(t, x'', x^n ) > 0$ for $x^n > -\epsilon$. Thus, the extended metric $\tilde{g}_1 $ is a Lorentzian one with the product form. 
	
\end{proof}
 Using a partition of unity, we can extend $(M, g_1)$ to a Lorentzian manifold $(\widetilde{M}, \widetilde{g}_1)$ in a small collar neighborhood of $\partial M$. 
 The manifold $(\widetilde{M}, \widetilde{g}_1)$ is time-oriented and moreover 
 it is globally hyperbolic since the extended metric $\widetilde{g}_1$ takes the product form.  
 
 Next, we consider the diffeomorphism $\Phi_\partial: V_{\partial, 1} \rightarrow V_{\partial, 2}$ , which 
 can be extended to a diffeomorphism $\Phi: \tM \rightarrow \tM$ by \cite{Palais1960}.
 Additionally, we can require $\Phi_\partial$ to be extended to $\Phi$ such that
 \[
 \Phi : {V_{\partial, 1}} \times (-\ep_1, 0) \rightarrow {V_{\partial, 2}} \times (-\ep_2, 0),
 \]
 for some $\ep_1, \ep_2 >0$.
 The following lemma shows with (\ref{difflam}) the manifold $(\widetilde{M}, (\Phi^{-1})^*\widetilde{g}_1)$ is a smooth extension of $(M, g_2)$ in the extended region ${V_{\partial, 2}} \times (-\ep_2, 0)$. 
 \begin{lm}\label{lm_phistarg}
     Let $\Phi_\partial$ be a diffeomorphism from $V_{\partial, 1}$ to $V_{\partial, 2}$
     and let $\Phi$ be a diffeomorphism from $\tM$ to itself such that $\Phi|_{V_{\partial,1}} = \Phi_\partial$.
     Suppose $(\Phi_\partial^{-1})^* \circ \Lambda_1|_{V_{\partial, 1}} \circ \Phi_\partial^* = \Lambda_2|_{V_{\partial, 2}}$. 
     Then $(\Phi^{-1})^*g_1$ coincide with $g_2$ on $V_{\partial, 2}$ and their derivatives of any order are equal on $V_{\partial, 2}$. 
 \end{lm}
 \begin{proof}
 We show $(\Phi_\partial^{-1})^* \circ \Lambda_1|_{V_{\partial, 1}} \circ \Phi_\partial^* $  can be regarded as a ND map with the metric  $(\Phi^{-1})^*g_1$ locally on $V_{\partial, 2}$. 
 Section \ref{subsecbm} shows that the boundary metric and its normal derivatives of any order are determined by the ND map locally. Thus we must have $(\Phi^{-1})^*g_1$ and $g_2$ are equal on $V_{\partial, 2}$ and so are their derivatives there.
 
 Indeed, for any $f \in \mathcal{E}'(V_{\partial,2})$, consider $h = \Lambda_1|_{V_{\partial, 1}}  (\Phi_\partial^*f)$.
 This implies there exists an outgoing solution $u$ to the boundary value problem 
 \begin{align*}
 \square_{g_1} u + w(x,u, \nabla_{g_1} u) &= 0, \quad \text{on } M,\\
 \partial_\nu u(x) &= \Phi_\partial^*f, \quad  \text{on } \partial M, 
 \end{align*}
 and $ u|_{V_{\partial, 1}} = h$. 
 Now let $x$ be the local coordinate in $\tM$ near $V_{\partial, 1}$ and $y$ be that near $V_{\partial, 2}$ with $y = \Phi(x)$. 
 Moreover, when $x \in V_{\partial, 1}$, one has $\Phi(x) = \Phi_\partial (x) \in  V_{\partial, 2}$.
 Consider the pull back $v = (\Phi^{-1})^*u$.
 One can find that 
 \[
 \nabla_{(\Phi^{-1})^*g_1} v = (\Phi^{-1})^*(\nabla_{g_1} u), \quad 
 \square_{(\Phi^{-1})^*g_1} v \circ \Phi = \square_{g_1} u, 
 \]
 and 
 \begin{align*}
 \partial_\nu u = \nu(\Phi^* v) =( (\Phi_* \nu) v) \circ \Phi = \tilde{\nu}(v)\circ \Phi, 
 \end{align*}
 see \cite[Lemma 2.1]{Stefanov2018}. 
 It follows that 
 \[
  \partial_{\tilde{\nu}} v|_{\partial M}= f, 
  \quad (\Phi^{-1})^*w(y, v, \nabla_{(\Phi^{-1})^*g_1} v) = w(x, u, \nabla_{g_1} u).
 \]
 Additionally, one has $ h = u|_{V_{\partial, 1}} = \Phi^* v|_{V_{\partial, 1}}$ and therefore \[
 v|_{V_{\partial, 2}} = ( \Phi_\partial^{-1})^* h = (\Phi_\partial^{-1})^* \circ \Lambda_1|_{V_{\partial, 1}} \circ \Phi_\partial^*f. 
 \]
 Thus, we show that $v$ is the solution to the boundary value problem
 \begin{align*}
 \square_{(\phi^{-1})^*g_1} v  + (\Phi^{-1})^*w(y, v, \nabla_{(\Phi^{-1})^*g_1} v)  &= 0, \quad \text{on } M,\\
 \partial_{\tilde{\nu}} v(y) &= f, \quad  \text{on } \partial M, 
 \end{align*}
 which implies $(\Phi_\partial^{-1})^* \circ \Lambda_1|_{V_{\partial, 1}} \circ \Phi_\partial^*$ is the ND map with the metric $(\Phi^{-1})^*g_1$ restricted to ${V_{\partial, 2}}$. This finishes the proof. 
 \end{proof}
 The same idea in the proof above shows the gauge invariance of the ND map under diffeomorphisms fixing the boundary. Indeed, by taking $\Phi^{-1} = \Psi$, one has the following lemma. 

\begin{lm}\label{lm_gauge}
     Suppose $\Psi: M \rightarrow M$ is a diffeomorphism with $\Psi|_{\partial M} = Id$.
     Let $\Lambda_{g,w}$ be the ND map for the boundary value problem (\ref{ihmNBC}) with the metric $g$ and quadratic form $w$. 
     Then we must have
     $
     \Lambda_{g,w} = {\Lambda}_{\Psi^*g, \Psi^* w}. 
     $
\end{lm}

\subsection{The construction of the observation set.}\label{subsec_constrobs}
Now we would like to extend $\hmupj$ to a timelike path $\hmuj$ in $ {V_{\partial, j}} \times (-\ep_j, 0)$ 
for $j = 1,2$ and choose a small neighborhood $V_j$ of $\hmuj$. 

\begin{figure}[h]
    \centering
    \includegraphics[height=0.35\textwidth]{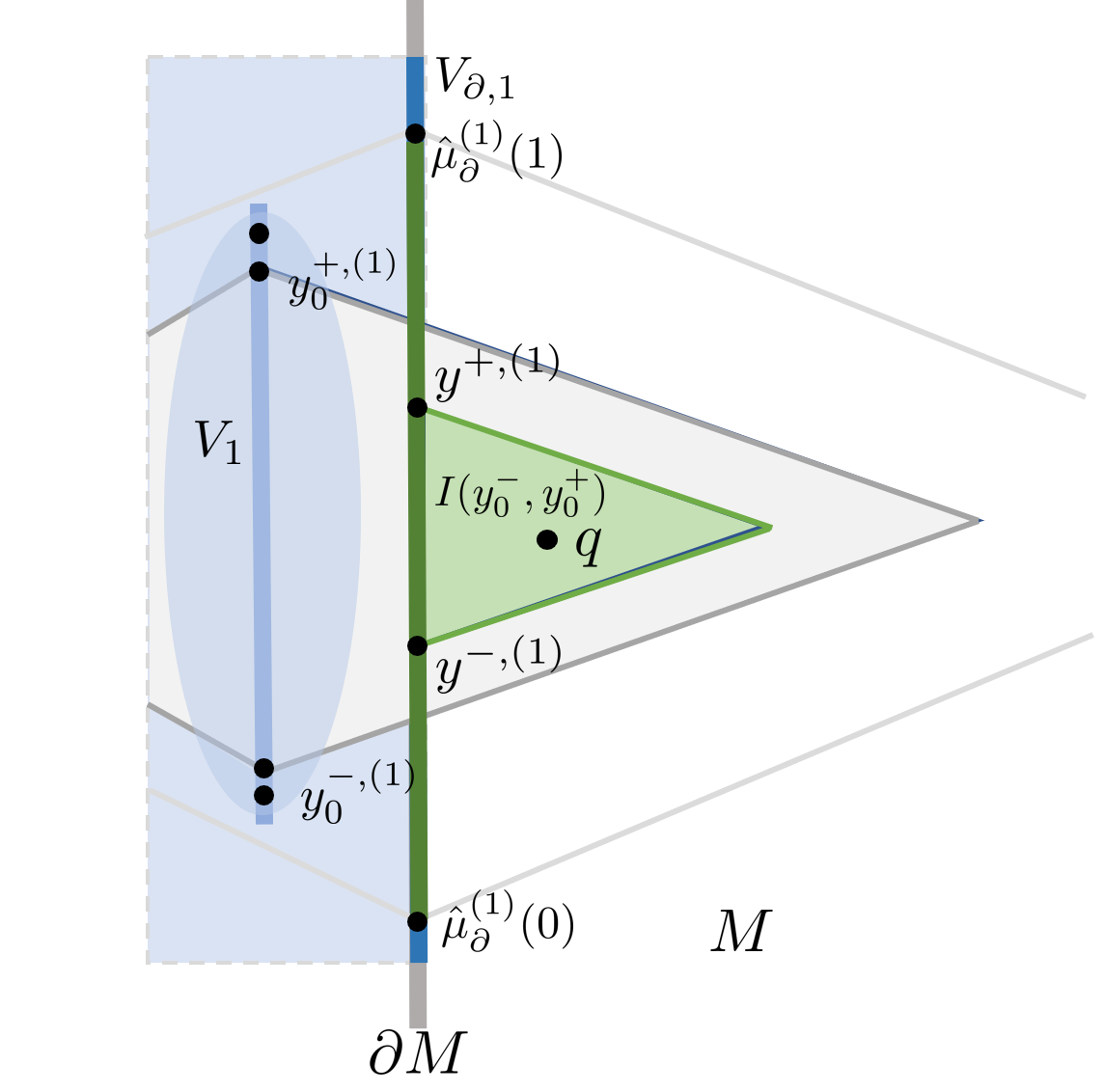}
    \caption{Illustration of the construction of $V_j$ for $j=1$. }
    \label{fig_V}
\end{figure}

First, we extend everything for $j = 1$, see Figure \ref{fig_V},  and then use the diffeomorphism $\Phi$ to get the corresponding extension for $j = 2$. Consider the set $J(\hmupa(0), y^{-,(1)})$ and $J(y^{+,(1)}, \hmupa(1))$. Both of them have a nonempty intersection with  $\widetilde{M} \setminus M$ and we can pick two points from the intersections respectively and call them $\hmua(0)$ and $\hmua(1)$. 
Pick $\hmua$ as any timelike path connecting $\hmua(0)$ and $ \hmua(1)$.
Then let 
\[
y_0^{-,(1)} \in I^+(\hmua(0)) \cap J^-(y^-), \quad y_0^{+,(1)} \in I^-(\hmua(1)) \cap J^+(y^+).
\]

It follows that
\[
\hmupa(0) \ll \hmua(0) \leq y_0^{-,(1)} \leq  y^{-,(1)} \ll y^{+,(1)}\leq y_0^{+,(1)} \leq \hmua(1) \ll \hmupa(1),
\]
and therefore
\[
I(y^{-,(1)},  y^{+,(1)}) \subset  I(y_0^{-,(1)},  y_0^{+,(1)}) \subset I(\hmua(0), \hmua(1)) \subset I(\hmupa(0), \hmupa(1)).
\]
Since all sets above are open, 
we can choose an open neighborhood $V_1$ of $\hmua$ such that 
\[
\hmua((0,1)) \subset V_1 \subset I(\hmupa(0), \hmupa(1)).
\]
To extend $\hmub$, we define $\hmub(t) = \Phi(\hmua(t))$ and $V_2 = \Phi(V_1)$.  
This construction of $\hmuj$ and $V_j$ for $j = 1, 2$ guarantees that 
if a null geodesic $\gamma_{x,v}(\mathbb{R}_+)$ with $(x,v) \in L^+V_j$ intersect the boundary before it returns to $V_j$, then its  
intersection with the boundary is contained in the region 
\[
(I(\hmupj(0), \hmupj(1)) \cap \partial M) \subset V_{\partial, j}, \quad j = 1,2.
\]


\subsection{Determine the conformal class of the metric.}
We extend $(M, g_j)$ to $(\tM, \tg_j)$ as above, for $j = 1,2$. 
In particular, we can assume the metric $\tg_2$ coincides with $(\Phi^{-1})^*\widetilde{g}_1$ in the extended region $V_{\partial,2} \times (-\ep_2, 0)$, for some $\ep_2 >0$. 
The statement (A1) is fulfilled by Section \ref{subsec_constrobs}. 
To verify (A2), note that $t$ is the smooth time function for the global splitting.
It follows that $t$ is strictly increasing along any future-pointed casual curves and strictly decreasing along any past-pointed ones. 
We denote the region $V_{\partial,j} \times (-\ep_j, 0]$ by $(0, T) \times W_j$. 
Then there exists small $\epsilon > 0$ such that in $\tM \setminus \intM$ we have
\[
I(\hat{\mu}^{(j)}(-\epsilon), \hat{\mu}^{(j)}(1+ \epsilon))  \subset (0, T) \times W_j, 
\] 
for $j = 1,2$. This implies that the assumption (A2) is valid for both $j=1,2$. 
  One actually has 
 \[
 (V_2, \widetilde{g}_2|_{V_2}) = (\Phi(V_1), (\Phi^{-1})^* \widetilde{g}_1|_{\Phi(V_1)}).
 \]
 The conditions for \cite[Theorem 4.5]{Kurylev2018} are satisfied by Section \ref{obs}, see Proposition \ref{pp_recoverEOS}.
 Then by \cite[Theorem 1.2]{Kurylev2018} the differential structure of $I(\hat{\mu}^{(j)}(-\epsilon), \hat{\mu}^{(j)}(1+ \epsilon))$ and the conformal class of $g$ can be  uniquely determined up to diffeomorphisms.
 %
 The arguments above are true especially when $V_{\partial, 1} = V_{\partial, 2} = (0, T) \times \partial N$.  Therefore,  
 we prove the first part of Theorem \ref{thm_recoverg}. 
 
 For the second part, in the case when $g_1,g_2$ are Ricci flat, the conformal diffeomorphism is an isometry according to \cite[Corrallory 1.3]{Kurylev2014}. 
 In the case when $\mathcal{M}^{(1)} =\mathcal{M}^{(2)}$ are independent of $x$, we follow the analysis in the proof of \cite[Theorem 1.1]{Wang2019}. 
 Set $\mathcal{M}^{(1)} =\mathcal{M}^{(2)} = \mathcal{M}$.
 It suffices to show 
 when $g_1 = e^{2 \gamma} g_2$ and $V_{\partial, 1} = V_{\partial,2}$,  we actually have $\gamma = 0$ from $\Lambda_1|_{V_{\partial, 1}} = \Lambda_2|_{V_{\partial, 2}}$. 
 We note that from the analysis above the assumption that the two ND maps are equal implies that 
 \[
 \gamma|_{V_\partial } = 0, \quad \partial_\nu^j \gamma |_{V_\partial } = 0, \quad j = 1, 2, \ldots,
 \]
 where $V_\partial = V_{\partial, 1} = V_{\partial,2}$, by comparing the normal derivatives of $g_1 = e^{2 \gamma} g_2$ restricted to $V_\partial$ in fixed semigeodesic coordinates. 
 Without loss of generosity, we can assume $\gamma =  0 $ in $\tM \setminus M$. 
 
 Now with $g_1 = e^{2 \gamma} g_2$, according to \cite[Proposition 4.5]{Lassas2018},
 the null bicharacteristic sets and Lagrangians obtained by the flow out of the Hamiltonian flow are the same for $g_1, g_2$.
 We choose different $u^{(k)}_j \in  I^\mu (\Lambda(x_j, \xi_j, s_0))$ for $k=1,2$ by \cite[Lemma 3.1]{Kurylev2014}, such that 
 \[
 u^{(k)}_j = \tQ_{g_k} h_j,\quad \text{with } h_j \in I^{\mu +1}(\mathcal{W}(x_j, \xi_j, s_0)). 
 \]
 Since $\gamma$ vanishes in $\tM \setminus M$, 
 one has
 \[
 \sigma^{(p)}(u^{(1)}_j) (p_j, \varrho^{(j)}) = \sigma^{(p)}(u^{(2)}_j) (p_j, \varrho^{(j)}), \quad 
 \partial_\nu u^{(1)}_j|_{V_\partial} =  \partial_\nu u^{(2)}_j |_{V_\partial} = f_j,  %
 \]
 for any $p_j \in V_\partial$ and $j = 1,2,3,4$. 
 Thus, let $f = \sum_{j=1}^4 \ep_j f_j$ as the same Neumann boundary condition. 
 We compute and compare the principal symbols of $\partial_{\vec{\epsilon}} \Lambda_k(f) |_{\vec{\epsilon}=0}$, for $k = 1,2$ in the following.   
 
 First, suppose $(p_j, \varrho^{(j)})$ with $p_j \in V_\partial$ and $(q, \zeta^{(j)})$ are joined by the same forward bicharacteristics in $\nxxi$. 
 It follows that 
 \begin{align*}
 &\sigma^{(p)}(u^{(k)}_j) (q, \zeta^{(j)}) = \sigma^{(p)}(\tQ_{g_k})(q, \zeta^{(j)},p, \varrho^{(j)})\sigma^{(p)}(u^{(k)}_j)(p_j, \varrho^{(j)}), \quad k = 1,2.
 \end{align*}
 By \cite[Propositon 4.5]{Lassas2018}, one has
 \begin{align*}
 \sigma^{(p)}(\tQ_{g_1})(q, \zeta^{(j)}, p_j, \varrho^{(j)}) = e^{-\gamma(q)}\sigma^{(p)}(\tQ_{g_2})(q, \zeta^{(j)},p_j, \varrho^{(j)}) e^{3\gamma(p_j)},
 \end{align*}
 which implies that 
 \begin{align}\label{eq_uj}
 \sigma^{(p)}(u^{(1)}_j) (q, \zeta^{(j)}) = e^{-\gamma(q)}\sigma^{(p)}(u^{(2)}_j) (q, \zeta^{(j)}). 
 \end{align}
 Recall $v_j  = u_j \mod C^{-\infty}(M)$ and they have the same principal symbol over $M$.
 By Proposition \ref{pp_ufour}, if $(y, \eta) \in T^*M|_{\partial M}$ is a covector away from $\unionGamma$ and lying along the forward null-bicharacteristic of $\square_g$ starting at $(q, \zeta)$, 
 then the principal symbol of ${\partial_{\epsilon_1}\partial_{\epsilon_2}\partial_{\epsilon_3}\partial_{\epsilon_4} \Lambda_k(f) |_{\epsilon_1 = \epsilon_2 = \epsilon_3 = \ep_4=0}}$ at the projection $(y_|, \eta_|)$ equals to 
\begin{align}\label{eq_ll}
-2 (2 \pi)^{-3}\sigma^{(p)}&(\mathcal{R})(y_|, \eta_|, y, \eta) \times  \\
&\sigma^{(p)}(\widetilde{Q}_{g_k})(y, \eta, q, \zeta) \mathcal{P}^{(i)}(\zeta^{(1)}, \zeta^{(2)}, \zeta^{(3)}, \zeta^{(4)}) 
(\prod_{j=1}^4 \sigma^{(p)}(v^{(k)}_j) (q, \zeta^{(j)})) \nonumber,
\end{align}
for $k = 1,2$ with $\zeta =\sum_{j=1}^4 \zeta^{(j)} $, since $\mathcal{C}_m$ defined in (\ref{C1}, \ref{C2}, \ref{C3}, \ref{C4}) vanishes for null covectors, for $m = 1,2,3,4$. 
By \cite{Wang2019}, we have 
\begin{align}\label{eq_pp}
\mathcal{P}^{(1)}(\zeta^{(1)}, \zeta^{(2)}, \zeta^{(3)}, \zeta^{(4)})   = e^{-4\gamma(q)}\mathcal{P}^{(2)}(\zeta^{(1)}, \zeta^{(2)}, \zeta^{(3)}, \zeta^{(4)}). 
\end{align}
Combining (\ref{eq_uj}, \ref{eq_ll}, \ref{eq_pp}) and 
\[
\quad \sigma^{(p)}(\tQ_{g_1})(y,\eta, q, \zeta^{(j)}) = \sigma^{(p)}(\tQ_{g_2})(y,\eta, q, \zeta^{(j)}) e^{3\gamma(q)},
\] we conclude that
\[
1= e^{3\gamma(q) -4\gamma(q) -4\gamma(q)} = e^{-5\gamma(q)},
\]
and therefore $\gamma(q)=0$ for any $q \in \mathbb{W}_1=\mathbb{W}_2$. 